\documentclass{article}
\usepackage[utf8]{inputenc}
\usepackage[T1]{fontenc}
\usepackage[latin9]{luainputenc}
\usepackage{geometry}
\usepackage{babel}
\usepackage{amsmath}
\usepackage{amsthm}
\usepackage{amssymb}
\usepackage[scr=esstix]
       {mathalpha}
\usepackage[unicode=true,pdfusetitle,
 bookmarks=true,bookmarksnumbered=false,bookmarksopen=false,
 breaklinks=false,pdfborder={0 0 1},backref=false,colorlinks=false]
 {hyperref}
\usepackage{dsfont,graphicx}
\usepackage[dvipsnames]{xcolor}

\usepackage{esint}

\makeatletter
\numberwithin{equation}{section}
\numberwithin{figure}{section}
\theoremstyle{plain}
\newtheorem{thm}{\protect\theoremname}[section]
\theoremstyle{plain}
\newtheorem{prop}[thm]{\protect\propositionname}
\theoremstyle{plain}
\newtheorem{cor}[thm]{\protect\corollaryname}
\theoremstyle{plain}
\newtheorem{lem}[thm]{\protect\lemmaname}
\theoremstyle{definition}
\newtheorem{defn}[thm]{\protect\definitionname}
\theoremstyle{remark}
\newtheorem{rem}[thm]{\protect\remarkname}

\@ifundefined{date}{}{\date{}}

\def\vep{\varepsilon}

\def\de{\delta}

\def\be{\begin{equation}}
\def\ee{\end{equation}}
\def\rife#1{\eqref{#1}}

\def\intr{\int_{\R}}
\def\intt{\int_{\T}}
\def\cP{{\mathcal P}}
\def\cE{{\mathcal E}}

\def\cB{{\mathcal B}}

\usepackage[nottoc,notlot,notlof]{tocbibind}

\DeclareMathOperator*{\argmin}{arg\,min}
\makeatother

\providecommand{\corollaryname}{Corollary}
\providecommand{\definitionname}{Definition}
\providecommand{\lemmaname}{Lemma}
\providecommand{\propositionname}{Proposition}
\providecommand{\remarkname}{Remark}
\providecommand{\theoremname}{Theorem}

\title{Free boundary regularity and support propagation in mean field games and optimal transport} 

\author{Pierre Cardaliaguet \thanks{Universit\'e Paris-Dauphine, Place du Mar\'echal de Lattre de Tassigny, 75016 Paris, France. \texttt{cardaliaguet@ceremade.dauphine.fr}} \and Sebastian Munoz \thanks{Department of Mathematics, University of Chicago, 60637 Chicago, USA. \texttt{sbstn@math.uchicago.edu}} \and Alessio Porretta \thanks{Dipartimento di Matematica, Universit\'a di Roma Tor Vergata. Via della Ricerca Scientifica 1, 00133 Roma, Italy. \texttt{porretta@mat.uniroma2.it}}}
\begin{document}
\global\long\def\T{R\mathbb{T}}%

\global\long\def\ue{u^{\vep}}%

\global\long\def\me{m^{\vep}}%

\global\long\def\R{\mathbb{R}}%

\global\long\def\ox{\overline{x}}%

\global\long\def\mg{v}%

\global\long\def\vg{\overline{v}}%

\global\long\def\oal{\overline{\alpha}}%
 
\maketitle

\begin{abstract}
    We study the behavior of solutions to the first-order mean field games system with a local coupling, when the initial density is a compactly supported function on the real line. Our results show that the solution is smooth in regions where the density is strictly positive, and that the density itself is globally continuous. Additionally, the speed of propagation is determined by the behavior of the cost function near small values of the density. When the coupling is entropic, we demonstrate that the support of the density propagates with infinite speed. On the other hand, for a power-type coupling, we establish finite speed of propagation, leading to the formation of a free boundary. We prove that under a natural non-degeneracy assumption, the free boundary is strictly convex and enjoys $C^{1,1}$ regularity. We also establish sharp estimates on the speed of support propagation and the rate of long time decay for the density. Moreover, the density and the gradient of the value function are both shown to be H\"older continuous up to the free boundary. Our methods are based on the analysis of a new elliptic equation satisfied by the flow of optimal trajectories. The results also apply to mean field planning problems, characterizing the structure of minimizers of a class of optimal transport problems with congestion.
\end{abstract}


\tableofcontents{}

\section{Introduction}

In this paper, we analyze the following first-order system of partial differential equations in the unknown scalar functions $(u,m)$:
\begin{equation}
\begin{cases}
-u_{t}+\frac{1}{2}u_{x}^{2}=f(m) & (x,t)\in\R\times(0,T),\\
m_{t}-(mu_{x})_{x}=0 & (x,t)\in\R\times(0,T),
\end{cases}\label{mfgsys}
\end{equation}
where $f$ is an increasing function.  Systems such as \rife{mfgsys} appear in the theory of mean field games (MFG for short), which describes Nash equilibria in differential games with strategic interactions between a large number of agents (\cite{LL1, LL2}). In that context, $m$ represents the distribution density of a population of infinitely many indistinguishable agents, while $u$ is the value function of  any generic agent, whose running costs depend on the density distribution $m$. 

The system shall also be complemented with initial and terminal conditions, which can take the form of a given initial distribution $m_0$ and a  final pay-off for the agents, which may itself be dependent on  the density of the population at the final time. In that case, the complete form of the system is as follows:
\begin{equation}
\tag{{MFG}}\begin{cases}
-u_{t}+\frac{1}{2}u_{x}^{2}=f(m) & (x,t)\in\mathbb{R}\times(0,T)\\
m_{t}-(mu_{x})_{x}=0 & (x,t)\in\mathbb{R}\times(0,T)\\
m(x,0)=m_{0}(x),\,\,u(x,T)=g(m(x,T)), & x\in\mathbb{R}.
\end{cases}\label{eq:MFGR}
\end{equation}
In \rife{eq:MFGR},  any such agent,  typically represented by a dynamical state $x(t)$, optimizes some cost given by the kinetic energy plus a congestion cost depending on the distribution of mass, that is, 
\[
u(x,t)=\inf_{\gamma:[t,T] \rightarrow \mathbb R, \,\gamma(t)=x} \quad\int_t^T \frac{1}{2} |\Dot \gamma(s)|^2+f(m(\gamma(s),s)) ds + g(m(\gamma(T),T)).
\]  
In turn, the density evolves according to the optimal feedback control of the agents, ($-u_x$).  At equilibrium, the value function of any single agent, and the density distribution of the population, satisfy the system \rife{eq:MFGR} coupling the Hamilton--Jacobi (HJ) equation for $u$ with the continuity equation for $m$.  

Alternatively to prescribing a final cost,  one may instead prescribe both the initial and final distributions of the agents. This is the so-called mean field optimal planning problem, which takes the complete form: 
\begin{equation}
\tag{{MFGP}}\begin{cases}
-u_{t}+\frac{1}{2}u_{x}^{2}=f(m) & (x,t)\in\mathbb{R}\times(0,T)\\
m_{t}-(mu_{x})_{x}=0 & (x,t)\in\mathbb{R}\times(0,T)\\
m(x,0)=m_{0}(x),\,\,m(x,T)=m_{T}(x), & x\in\mathbb{R},
\end{cases}\label{eq:MFGPR}
\end{equation}
for some prescribed mass distributions $m_0, m_T$.

In this latter case, the system \rife{eq:MFGPR} can also be viewed as the first-order optimality conditions of an optimal transport problem on the Wasserstein space of measures. In this context, one seeks to minimize the functional
\begin{equation}\label{OTF}
{\cB(m,v)}:= \int_0^T\!\! \!\int_{\R^d} \frac 12\, |v|^2 dmdt+  \int_0^T\!\! \!\int_{\R^d} F(m) dxdt
\quad \text{subject to } \begin{cases} m_t - (vm)_x=0 & \\ m(0)=m_0\,, m(T)=m_T & \end{cases}
\end{equation}
where $F(s)=\int_0^s f(r)dr$. This problem appears as a generalization of the dynamic formulation of the  mass transport problem (cf. \cite{BB}). Indeed, the classical  dynamic version of the Monge-Kantorovich problem  corresponds to the case $f=0$, whereas the resulting optimizer $m(t)$ is the geodesic, in the Wasserstein space, that connects $m_0$ and $m_T$. The function $u$ appears from the dual formulation, and  $u(0),   u(T)$ are the so-called Kantorovich potentials.  Classical monographs on optimal transport theory are \cite{AGS, Sa, Vi}. The additional cost term $F(m)$, which is convex whenever $f$ is increasing, can be interpreted as an incentive to avoid congested areas. Similar optimal transport problems with congestion have been widely investigated in the literature, including their connection with MFG theory, since the solution of \rife{eq:MFGPR} provides the minimizer $(m,u_x)$ of $\cB$; see, for example,   \cite{Gomes2, CMS, GMST, LaSa, OPS}.  
 
Since the courses given by P.-L. Lions at Coll\`ege de France, devoted to  MFG theory, it was observed that the increasing character of $f$ would produce a regularizing effect  in the solutions $(u,m)$ of system  \rife{mfgsys}; indeed, the first order system can be reformulated as a single degenerate elliptic equation for $u$ in the space-time variables (\cite{L-college}).
The approach suggested by Lions was recently  developed and extended in \cite{MimikosMunoz, Munoz, Munoz2, Porretta}. So far, in arbitrary dimensions, it is well understood that solutions are smooth under the blow-up assumption \begin{equation} \label{eq:blowup} \lim_{m\rightarrow 0^+}f(m)=-\infty, \end{equation} 
provided that the marginals $m_0, m_T$ are (strictly) positive, say for positive measures on a compact domain (e.g. on the flat torus) or for Gaussian-like measures on the whole space. It was also established in \cite{MimikosMunoz} that, for the one-dimensional case, assumption \eqref{eq:blowup} may be removed, thus requiring only the positivity of the marginals. 

Much less is known about what happens in case of compactly supported marginals. Regularizing effects of the type $L^1 \to L^\infty$ have been proven to hold (see \cite{LaSa, Porretta}), but the propagation of the support of the solution, and even basic matters of regularity such as the continuity of the density, have largely remained open issues. 

The purpose of this article is to investigate this question in the one-dimensional case of the space variable: we address both the case of finite and infinite speed of propagation of the support of initial measures.  Roughly speaking, those two cases correspond to two model choices for the coupling function, namely $f(m)= m^\theta$ for some $\theta>0$, or $f(m)= \log (m)$. In the latter case, there is infinite speed of propagation, and the solution starting with compact support becomes instantaneously positive and smooth. By contrast, when $f(m)= m^\theta$, we observe finite speed of propagation, and the solution evolves with compact support.  This leads to new interesting questions concerning the study of the free boundary $\partial \{m(t)>0\}$, which is the main focus of our paper.

By way of analogy, which is also natural from the optimal transport viewpoint,  for $f=\log(m)$  the evolution of $m$ is reminiscent of a nondegenerate diffusion, such as the heat equation. On the other hand, the case of a power nonlinearity $f$ resembles the behavior of degenerate slow diffusions such as the flow through a porous medium (see \cite{vazquez2007porous}).  This analogy becomes more compelling as in fact, when $f(m)= m^\theta$,  we exhibit a family of self-similar solutions which evolves from a Dirac mass into a compactly supported measure. These solutions are given by the formula
\begin{equation} \label{eq:self similar m}
m(x,t)= t^{-\overline{\alpha}}\left( R- \frac{\overline{\alpha}(1-\overline{\alpha})}{2} \left(\frac{x}{t^{\overline{\alpha}}}\right)^2\right)_+^{1/\theta}\,,\quad \overline{\alpha}= \frac 2{2+\theta},
\end{equation}
which is strongly reminiscent of the famous Barenblatt solution for the porous medium equation \cite{Bar}.
The behavior exhibited by this class of compactly supported solutions serves as a prototype for our  analysis of problems \rife{eq:MFGR} and \rife{eq:MFGPR}.  In order to describe our main results on the propagation of the support and the characterization of the free boundary, we assume henceforth that
\be\label{power}
f(m)= m^\theta\,, \quad \theta>0\,,
\ee
and that the initial measure $m_0$ is a continuous, compactly supported, probability density, with a bump-like shape:
\be\label{bump0}
\{m_{0}>0\}=(a_{0},b_{0})\quad \text{ and }\quad 
 \frac{1}{C_{0}}\text{dist}(x,\{a_{0},b_{0}\})^{\alpha_0}\leq m_{0}(x)\leq C_{0}\text{dist}(x,\{a_{0},b_{0}\})^{\alpha_0},
\ee
for some $\alpha_0, C_0>0$. 
In order to keep our main statement in a simpler form,   we   will assume here a consistent condition on the terminal density, in case of problem \rife{eq:MFGPR},
\begin{equation}\label{bumpT}
\{m_{T}>0\}=(a_{1},b_{1}),  \quad \text{ and }\quad \frac{1}{C_{1}}\text{dist}(x,\{a_{1},b_{1}\})^{\alpha_0}\leq m_{T}(x)\leq C_{1}\text{dist}(x,\{a_{1},b_{1}\})^{\alpha_0}.
\end{equation}
However, more general situations  will be considered later, allowing for the behavior of $m_T$ at the boundary of its support to differ from the behavior of $m_0$. Similarly, for  problem   \rife{eq:MFGR}, we will require here, for simplicity, consistency between $f$ and the terminal cost coupling $g$, namely 
\be\label{g=f}
g(s)= c_T\, s^\theta\,,\qquad \hbox{for some $c_T\geq 0$.}
\ee
We may now state our main result, which proves that  the unique  solution of \rife{eq:MFGR} or \rife{eq:MFGPR} has a  compactly supported density and the free boundary $\partial \{m(t)>0\}$ consists of  two Lipschitz curves, which are $C^{1,1}$ under a suitable non-degeneracy assumption at the initial time. Those curves can be characterized in terms of the flow of optimal trajectories for the agents' optimization problem. 

In fact, we will show  that $u$ is smooth inside the support of $m$, and the characteristic flow
$$
\begin{cases}
\dot \gamma ( x,\cdot)=-u_{x}(\gamma( x,\cdot),\cdot)\\
\gamma( x,0)=x
\end{cases}
$$
is well defined starting from $x$ in the support of $m_0$. Finally,  we also show that the left and right free boundary curves are, respectively,  convex and concave, and in problem \rife{eq:MFGR} the support spreads outward in time.

\begin{thm}\label{thm.intro1} Let $f$ be given by \rife{power}, and let $0<\alpha<1$. Assume that $m_0:\mathbb{R} \rightarrow [0,\infty)$ satisfies \rife{bump0}, $m_0^\theta \in C^{1,\alpha}(a_0,b_0),$  and $m_0^{\theta}$ is semi--convex.
In case of problem \rife{eq:MFGPR}, assume also that $m_T^\theta \in C^{1,\alpha}(a_0,b_0)$ satisfies \rife{bumpT}.
Let $(u,m)$ be the solution to \rife{eq:MFGPR}, or to \rife{eq:MFGR} with $g$ satisfying \eqref{g=f}. Then $(u, m)\in C_{\emph{loc}}^{2,\alpha}((\mathbb{\mathbb{R}\times}[0,T])\cap\{m>0\})\times C_{\emph{loc}}^{1,\alpha}((\mathbb{\mathbb{R}\times}[0,T])\cap\{m>0\})$, and the following holds:

\begin{enumerate}

\item There exist two functions $\gamma_{L}<\gamma_{R}\in W^{1,\infty}(0,T)$,
such that 
\begin{equation}
\{m>0\}=\{(x,t)\in\mathbb{R}\times[0,T]:\gamma_{L}(t)<x<\gamma_{R}(t)\}.
\end{equation}
Moreover, the flow $\gamma$ of optimal trajectories is well defined
on $(a_{0},b_{0})\times[0,T]$, we have
$$
\gamma\in W^{1,\infty}((a_{0},b_{0})\times(0,T))\cap C_{\emph{\text{loc}}}^{2,\alpha}((a_{0},b_{0})\times[0,T]),\,\,\,\gamma_{x}>0,\,\,\,\gamma_{L}(t)=\gamma(a_{0},t),\,\gamma_{R}(t)=\gamma(b_{0},t),
$$
 and $\gamma$ is a classical solution in $(a_{0},b_{0})\times(0,T)$
to the elliptic equation
\begin{equation}
\gamma_{tt}+\frac{\theta m_{0}^\theta}{(\gamma_{x})^{2+\theta}}\gamma_{xx}=\frac{(m_{0}^\theta)_{x}}{(\gamma_{x})^{1+\theta}}.\label{eq:flow0}
\end{equation}

\item If we assume further the concavity condition
\be\label{concavefm0}
(m_{0}^\theta)_{xx}\leq0\text{ in }\{x\in (a_0,b_0) : \emph{\text{dist}}(x,\{a_{0},b_{0}\})<\delta\}\text{ for some }\text{\ensuremath{\delta>0}}, 
\ee
then we have $\gamma_{L},\gamma_{R}\in W^{2,\infty}(0,T)$, and there  exists $K>0$ such that, for a.e. $t\in [0,T]$,
$$
\frac{1}{K}\leq \Ddot \gamma_{L}(t)\leq K,\,\text{ and }\,-K\leq\ \Ddot \gamma_{R}(t)\leq-\frac{1}{K},
$$ 
where  $K$ depends on $T, C_{0},  \theta, \de^{-1}, \|((m^\theta_0)_{xx})^-\|_\infty$, $|(m_0^{\theta})_x(a_0^{+})|$,$|(m_0^{\theta})_x(b_0^{-})|$  (and additionally on $c_T$ for problem \rife{eq:MFGR}, and on $C_1$ for problem \rife{eq:MFGPR}).

Moreover, when $(u,m)$ solves \eqref{eq:MFGR}, we have, for $t\in [0,T],$
\begin{equation*} -K(c_T+(T-t))\leq \dot \gamma_L(t)\leq -\frac{1}{K}(c_T+(T-t)),\,\text{ and }\,\frac{1}{K}(c_T+(T-t))\leq \dot \gamma_R(t)\leq K(c_T+(T-t)).
\end{equation*}
\end{enumerate}
\end{thm}



In relation to the main text, Theorem \ref{thm.intro1} is a combination of Theorem \ref{thm:well-posedness theorem} (for the existence of the solution $(u,m)$ and its regularity in $\{m>0\}$), Theorem \ref{thm:characteriz free boundary} (for the description of the free boundary) and Theorem \ref{thm:free boundary regularity} (for the  regularity and convexity of the free boundary).

\begin{rem} We now discuss the nondegeneracy conditions required on $m_0$ at the boundary of its support. First, we note that the  $C^{1,\alpha}(a_0,b_0)$ (and therefore $W^{1,\infty}(\mathbb{R})$) condition on $m_0^\theta$, together with \rife{bump0}, implies that $\alpha_0\geq \frac1\theta$ in \rife{bump0}. In turn, the concavity assumption \rife{concavefm0} further restricts the behavior of $m_0^\theta$, forcing $\alpha_0= \frac1\theta$ in \rife{bump0}. However, this condition of a linear,  nondegenerate behavior of $m_0^\theta$ is  natural (a case in point being the self-similar solution itself), and should be compared with standard nondegeneracy conditions on the initial data in other free boundary problems (e.g. in the study of the moving free boundary  for the porous medium equation, see \cite{vazquez2007porous}).

 We also wish to highlight that, when dealing with problem \rife{eq:MFGPR}, some asymmetry can be observed when requiring some conditions (e.g. concavity-type assumptions) on $m_0$ but not on $m_T$. This kind of asymmetry arises because we are referring to the forward flow $\gamma(x,t)$ in our statement. Of course, similar results will hold when reversing the time flow and exchanging the roles of $m_0, m_T$.
\end{rem}

\begin{rem} We stress that the first part of Theorem \ref{thm.intro1} remains true under more general conditions than \rife{bumpT} (respectively, \rife{g=f}). We refer the reader to Theorem \ref{thm:characteriz free boundary},  which allows for the behavior of  $m_T$ at the boundary of its support to be different from the behavior of $m_0$ (respectively, in case of problem \rife{eq:MFGR}, for the function $g(s)$ to be a different power than $f$).
\end{rem}

The result in the second part of Theorem \ref{thm.intro1} corresponds exactly to the picture described by the self-similar solution \eqref{eq:self similar m}. Indeed, our next result shows that the free boundary  propagates with strictly convex (resp. concave) behavior at the left (resp. right) free boundary curve. In fact, if we strengthen the concavity assumption on $m_0^\theta$, we show that the free boundary evolves with the optimal speed given by the self-similar solution. Moreover, the long time decay of the density occurs with the same rate, as exhibited by  \eqref{eq:self similar m}.

\begin{thm}\label{thm.intro2} Under the assumptions  of Theorem \ref{thm.intro1},  let $(u,m)$ be the unique  solution to \eqref{eq:MFGR} or \eqref{eq:MFGPR}, and let $\gamma$ be the associated flow of optimal trajectories. Assume in addition  that $-K \leq (m_0^\theta)_{xx}\leq -\frac{1}{K}$ in $(a_0,b_0)$ for some $K>0$, and, in case of problem \rife{eq:MFGR}, assume that $c_T= \kappa_1 T$ in \rife{g=f}.   If we define
$$
\oal= \frac{2}{2+\theta}\,\,; \qquad\qquad \mathscr{d}(t) = 
\begin{cases} 
t & \text{if } u \text{ solves \eqref{eq:MFGR}}, \\
\emph{\text{dist}}(t, \{0, T\}) & \text{if } u \text{ solves \eqref{eq:MFGPR},}
\end{cases}
$$
then there exists a constant $C>0$ such that for every $(x,t)\in [a_0,b_0] \times [0,T]$,
\begin{equation}  \frac{1}{C} (1+\mathscr{d}(t)^{\oal})\leq |\emph{supp}(m(\cdot,t))|\leq C(1+\mathscr{d}(t)^{\oal}),\quad|\gamma(x,t)|\leq C(1 +\mathscr{d}(t)^{\oal}),    
\end{equation}
\begin{equation}   \frac{1}{C} \frac{m_0(x)}{(1+\mathscr{d}(t)^{\oal})}\leq m(\gamma(x,t),t)\leq C\frac{m_0(x)}{(1+\mathscr{d}(t)^{\oal})},   
\end{equation}
where $
C=C\left(C_{0}, \kappa_1, \kappa_1^{-1},|a_0|,|b_0|,   K\right)$ in case of problem  \eqref{eq:MFGR}, and $
C=C\left(C_{0},  C_1,|a_0|,|b_0|,|a_1|, |b_1|, K\right) $ in case of \eqref{eq:MFGPR}. 
\end{thm}
\vskip1em
Theorem \ref{thm.intro2} is nothing but Theorem \ref{thm:long time} below. Let us stress that a crucial role in the proof of the above results is played by the equation satisfied by   the flow of optimal curves $\gamma$, namely \rife{eq:flow0}. 
 In particular, the Lipschitz regularity of $\gamma$ is obtained by a maximum principle argument applied to $\gamma_x$, which is derived from \rife{eq:flow0}. We obtain further insight by studying the equation of $m$ in Lagrangian coordinates. Indeed, the function $v= f(m(\gamma(x,t), t))$ satisfies the (degenerate) elliptic equation 
\begin{equation}\label{eq:v eq intro} 
  - \left(\gamma_x(\theta v)^{-1}v_t\right)_t-\left(\gamma_x^{-1}v_x\right)_x= 0,
\end{equation}
where  one can prove, assuming    \rife{concavefm0}, that the  positive quantity $\gamma_x$ is bounded below and above.  This elucidates the key distinction between the present problem and slow diffusions of porous medium type; equation \eqref{eq:v eq intro} is diffusive (rather than parabolic) in the time variable. 
Relying on this equation, we establish the regularity of $m$ up to the free boundary. Namely, we prove that $m$ is H\"older continuous,  through an application of the intrinsic scaling regularity method (see \cite{DiB, DiB2, Urbano}). In turn, we show that $Du$ is  H\"older continuous as well. We can summarize these regularity results, contained in Theorem \ref{thm:f(m) holder} and \ref{thm: Du Holder} respectively, as follows:

\begin{thm}  Under the assumptions of Theorem \ref{thm.intro2}, we have $ f(m) \in  C^{\beta}_{\text{\emph{loc}}}(\R\times(0,T))$ and $u\in C^{1,\frac\beta 2}_\text{\emph{loc}}(\R\times(0,T))$ for some $\beta\in (0,1)$. 
\end{thm}
Finally, our last result shows that the solutions of \rife{mfgsys} exhibit a different behavior when
\be\label{flog}
f(m)= \log(m)\,.
\ee
In contrast with the case of a power nonlinearity, the  unbounded payoff as $m\downarrow 0$ given by\rife{flog} implies that the support of the density propagates with infinite speed. This behavior is observable in both problem \rife{eq:MFGR} (with $g(s)= c_T\log(s)$, $c_T\geq 0$) and in the planning problem \rife{eq:MFGPR}, which corresponds to  the optimal transport functional \rife{OTF} with the entropy term $F(m)= m\log(m)$. 
More specifically, under the assumption of $m_0$ being continuous with compact support (and similarly for $m_T$ in the case of \eqref{eq:MFGPR}), we establish the existence of classical solutions $(u,m)$ with $m>0$ in $(0,T)$. 

Compared to  Theorem \ref{thm.intro1}, the positivity of solutions on the whole space now makes it much more delicate to use the flow of optimal curves $\gamma(x,t)$, which are no longer confined in a bounded set. This difficulty leads us to  require an extra symmetry  and monotonicity assumption, namely that $m_0$  is even and nonincreasing in $(0,\infty)$ (and the same for the terminal density $m_T$). We are able to take advantage of this assumption by showing that the solution $m(\cdot,t)$ preserves this property for all $t\in (0,T)$, which is in itself a non-trivial feature of the MFG system (see Lemma \ref{lem.convexity}). However,  we no longer require
any special behavior of $m_0$ when vanishing at the boundary of its support, avoiding conditions such as \rife{bump0}, \rife{bumpT}. In fact, the support now propagates instantly, regardless of the flatness of $m_0$.

\begin{thm}\label{thm.intro-log} Let $f$ be given by \rife{flog},  let $\alpha \in (0,1)$, and assume that $m_0$ is a continuous, compactly supported, density on $\R$, which is  $C^{1,\alpha}_{\emph{loc}}$ in the set $\{m_0>0\}$, even and nonincreasing on $[0,\infty)$.

\begin{enumerate}

\item If   $g(s)= c_T \log (s)$, for some $c_T\geq 0$, then there exists a unique  classical solution $(u,m)\in C^2(\R\times (0,T])\times C^1(\R\times (0,T])$ of \rife{eq:MFGR}  
  such that  $m$ is continuous   and bounded on  $\R\times [0,T]$, positive on $\R\times (0,T)$ with $m(0)=m_0$, and $|x|^2m(t) \in L^1(\R)\,, \frac {u(t)}{(1+ |x|^2)}\in L^\infty(\R)$, for every $t\in (0,T)$.

\item If  $ m_T\in C_c(\R)$   is  even and nonincreasing on $[0,\infty)$,  then there exists a unique (up to addition of a constant to $u$) classical solution $(u,m)\in C^2(\R\times (0,T))\times C^1(\R\times (0,T))$ of \rife{eq:MFGPR}   such that  $m$ is continuous   and bounded on  $\R\times [0,T]$, positive on $\R\times (0,T)$ with $m(0)=m_0, m(T)=m_T$, and $ |x|^2m(t) \in L^1(\R)\,,\frac {u(t)}{(1+ |x|^2)}\in L^\infty(\R)$, for every $t\in (0,T)$. 

\end{enumerate}
\end{thm}

Theorem \ref{thm.intro-log} is Theorem \ref{thm.caselogRsym} below. As mentioned before, the symmetry and monotonicity assumption on $m_0$, which is required in Theorem \ref{thm.intro-log}, allows us to overcome certain difficulties in the
obtention of classical positive solutions in the whole space. These difficulties disappear in compact domains, as is seen in Theorem \ref{ex-log-torus}, where we prove the existence of classical periodic solutions, with $m>0$ in $(0,T)$, under the only condition that $m_0$ (and $m_T$) are continuous and compactly supported.

\vskip1em
 \noindent {\bf Outline of the paper.} The paper is organized as follows. In Section \ref{sec:self similar} we exhibit the class of self-similar solutions which will serve as a prototype for our main results. Section \ref{sec: periodic} presents the key features of smooth, periodic solutions with a positive density:  structural properties, displacement convexity, Lipschitz estimates, and a modulus of continuity for the density. Section \ref{sec:finitespeed}, dealing with compactly supported solutions, is the heart of the paper: starting with the existence of solutions (Subsection \ref{subsec:wellposed}), it culminates with the regularity,  geometric properties, and long time behavior of the free boundary (Subsection \ref{free-analysis}), and the H\"{o}lder regularity of $m$ and $Du$ (Subsection \ref{subsec:reguH}). Section \ref{subsec:infinitespeed} is devoted to the entropic coupling ($f=\log$) and the infinite speed of propagation. Appendix \ref{sec:appendix} contains the computations for the self-similar solutions.



\section{Self-similar solutions} \label{sec:self similar}

We now exhibit a family of compactly supported, self-similar solutions of the system
\begin{equation}
\begin{cases}
-u_{t}+\frac{1}{2}u_{x}^{2}=m^\theta & (x,t)\in\R\times(0, \infty),\\
m_{t}-(mu_{x})_{x}=0 & (x,t)\in\R\times(0, \infty),
\\
\int_{\R}  m(t)dx =1 & \,\, t\in (0, \infty),
\end{cases}\label{powersys}
\end{equation}
where $\theta>0$. By a solution of \eqref{powersys} we mean here that $u$ is Lipschitz continuous and $m$ continuous and nonnegative, the first equation being understood in the sense of viscosity solutions, while the second equation is satisfied in the sense of distributions. The solution is described in the following result.

\begin{prop}\label{self-building} For $\theta>0$, let us set
$$
\overline{\alpha}= \frac{2}{2+\theta}\,\,  
$$
and let $R$ be the unique positive number such that  $\int_{\R}   \left( R-\frac12(\overline{\alpha}-\overline{\alpha}^2)y^2\right)_+^{1/\theta}dy=1$. A solution of \rife{powersys} is given by $(u,m)$, with
$$
m(x,t)=  t^{-\overline{\alpha}} \phi(x/t^{\overline{\alpha}}) \,,\qquad \hbox{where} \quad \phi(y)= \left( R-\frac12(\overline{\alpha}(1-\overline{\alpha}) )y^2\right)_+^{1/\theta}\, 
$$
%
%
and $u$ defined as follows:

(i) either $\theta=2$ and 
\be\label{u2}
u(x,t)=\begin{cases}
-\frac{1}{4t}x^{2}-R\log t & \text{if \ensuremath{\Delta\leq0}}\\
-\frac{2R|x|}{|x|-\sqrt{\Delta}}-2R\log(\frac{|x|-\sqrt{\Delta}}{\sqrt{8R}}) & \text{if \ensuremath{\Delta>0}}
\end{cases}
\ee
where $\Delta=x^{2}-8Rt$,

(ii) or $\theta\neq 2$ and
\be\label{uneq2}
u(x,t)=\begin{cases}
-\overline{\alpha}\frac{x^{2}}{2t}-R\frac{1}{2\overline{\alpha}-1}t^{2\overline{\alpha}-1} & \text{if \ensuremath{\Delta\leq0}}\\
\frac{-R\overline{\alpha}}{(1-\overline{\alpha})(2\overline{\alpha}-1)}S^{2\overline{\alpha}-1}-\frac{\overline{\alpha} R}{1-\overline{\alpha}}S^{2\overline{\alpha}-2}(t-S) & \text{if \ensuremath{\Delta>0},}
\end{cases}
\ee
where $\Delta=|x|-\sqrt{\frac{2R}{\overline{\alpha}(1-\overline{\alpha})}}t^{\overline{\alpha}}$ and the function $S=S(x,t)$ is defined implicitly by the equation:
\begin{equation}
S\sqrt{\frac{2R}{\overline{\alpha}(1-\overline{\alpha})}}-|x|S^{1-\overline{\alpha}}+\sqrt{\frac{2R\overline{\alpha}}{1-\overline{\alpha}}}(t-S)=0.\label{eq:1}
\end{equation}
%
%
%
\end{prop}

The explicit construction of $u$ can be understood by distinguishing two regions. First, one shows that, on the support of $m$, $u$ must be given by
\be\label{usuppo}
u(x,t)= -\overline{\alpha}\frac{x^2}{2t}+c(t), \quad \hbox{with}\quad   c'(t)= -R t^{-2\theta/(2+\theta)}\,.
\ee
Outside the support of $m$, the values of $u$ are extended along the optimal curves, which are straight lines in the set $\{m=0\}$. This leads to formula \rife{u2} if $\theta=2$, or to formula \rife{uneq2} if $\theta\neq 2$. Note that $\Delta\leq 0$ corresponds with the support of $m$.
Finally, we point out that $u$ is defined up to an additive constant. The proof of the statements made in Proposition \ref{self-building} will be presented in Appendix \ref{sec:appendix}.

We note the following relevant facts about this solution, which will serve as a model for our later assumptions and results:

\begin{itemize}

\item  At time $t=0$, the measure corresponds to a Dirac mass at $x=0$. For positive times $t>0$, in general, the density $m$ is merely H\"older continuous.
\item For each $t>0$, the function $f(m)$ is always Lipschitz (away from $t=0$). Moreover, $f(m(\cdot,t))$ is strictly concave within the support, and, in particular, $f(m(\cdot,t))_x$ is non-zero at the endpoints. A weaker, local version of these conditions will serve as our non-degeneracy assumption on the initial distribution $f(m_0)$ (see  \eqref{eq:concavity assumption} and \eqref{eq:nondegeneracy}), in order to prove our main regularity result for the free boundary (Theorem \ref{thm:free boundary regularity}), and the H\"older continuity of $m$ (Theorem \ref{thm:f(m) holder}).  Moreover, the full strict concavity assumption on $f(m_0)$ will yield our result on the optimal speed of support propagation and long time decay of $m$ (Theorem \ref{thm:long time}).

\item The value function $u$ is smooth on the support of $m$ but $u_{xx}$ blows-up at the interface (see Remark \ref{rem:uxx unbounded}), at least when approaching from outside the support. In fact, it is shown in Proposition \ref{prop:self similar regu} that $u\in C^{1,s}$ for a certain $0<s<1$. Accordingly, our general results will show that the $C^{1,s}$ regularity exhibited by this explicit solution in fact holds for arbitrary solutions of the MFG system, at least under the aforementioned non-degeneracy assumption (Theorem \ref{thm: Du Holder}).


\end{itemize}

\section{Structure and a priori estimates in the periodic setting}\label{sec: periodic}
\subsection{Structural properties of the MFG system}
The results of the paper are systematically obtained by establishing a priori estimates on a regularized system which has a smooth solution. In this subsection, we will explain the structural properties of the system \rife{mfgsys}, deriving the fundamental identities used throughout the paper. Here we assume that  $f'>0$ on $(0,\infty)$, and $(u,m)$ is a classical solution to \rife{mfgsys}, with $m$ being positive, and $u_x$ has at most a linear growth. Note that we only assume $m(\cdot,0)=m_0$ to be smooth and positive (no condition on the mass), so that the results of this part are valid for equations with periodic boundary conditions as well as in the whole space. 

We begin with the elliptic equation satisfied by $u$, first derived by Lions in \cite{L-college} (see also  \cite{Munoz, Porretta}).  It is obtained by simply eliminating $m=f^{-1}(-u_t+u_x^2/2)$ from the system, thanks to the fact that $f'>0.$
\begin{lem}  Let $(u,m)$ be a classical solution to \rife{mfgsys}. The map $u$ satisfies the quasilinear elliptic equation 
\be\label{eq.elluBIS}
-u_{tt}+2u_xu_{xt} -(u_x^2+mf'(m))u_{xx}=0 \qquad \text{in $\R\times (0,T)$}
\ee
with $m=f^{-1}(-u_t+u_x^2/2)$. 
\end{lem}
We observe that \eqref{eq.elluBIS} is a degenerate elliptic equation for $u$; the uniform ellipticity being lost when $m$  vanishes. The study of \rife{eq.elluBIS}  is the starting point of the regularity theory developed in \cite{L-college, Munoz, Porretta} under conditions ensuring a positive control from below on $m$. In that case, equation \rife{eq.elluBIS} turns out to be equivalent to the system (\ref{mfgsys}), at least for classical solutions. As it is customary for quasilinear problems, a key role is played by gradient estimates, which are obtained through the maximum principle. We will recall this approach in Subsection \ref{subsection:stime}. For this purpose, it is convenient to introduce the linear second order operator  
\begin{equation}
Q(v):=-v_{tt}+2u_{x}v_{xt}-(u_{x}^{2}+mf'(m))v_{xx},
\label{eq:Q definition}
\end{equation}
and accordingly, the linearized operator generated from \rife{eq.elluBIS}:
\begin{equation}
L(v)=Q(v)+2(u_{xt}-u_{x}u_{xx})v_{x}-\left(\frac{mf''(m)}{f'(m)}+1\right)u_{xx}(-v_{t}+u_{x}v_{x}).\label{eq:linearization}
\end{equation}
For the rest of this section, the solutions will be tacitly assumed to be sufficiently smooth to justify the computations below.
\begin{lem}  Let $u$ be a classical solution to \rife{eq.elluBIS}, and $Q,L$ be defined by \rife{eq:Q definition} and \rife{eq:linearization}. Then we have
$$
L(u_t )=0 \,,\quad L(u_x)=0
$$
and the function $w:=f(m)$ satisfies
\begin{equation}
Q(w)-w_{x}^{2}+m(mf''(m)+2f'(m))u_{xx}^{2}=0.\label{eq:f(m) equation}
\end{equation}
\end{lem}

\begin{proof}  The equations satisfied by $u_t, u_x$ are obtained by differentiation of (\ref{eq.elluBIS}). From $L(u_x)=0$ we also obtain, by the chain rule,
\begin{equation}
L\left(\frac{1}{2}u_{x}^{2}\right)=-(u_{xt}-u_{x}u_{xx})^{2}-mf'(m)u_{xx}^{2}. 
\label{eq:-3}
\end{equation}
Hence,
\be\label{eq:---4}
L(f(m))=L(-u_{t}+\frac{1}{2}u_{x}^{2})=-(u_{xt}-u_{x}u_{xx})^{2}-mf'(m)u_{xx}^{2}=-f(m)_{x}^{2}-mf'(m)u_{xx}^{2}.
\ee
Now, using (\ref{eq:linearization}), we also have
\be\label{eq:--5} \begin{split}
L(f(m))= & Q(f(m))-2f(m)_{x}^{2}-\left(\frac{mf''(m)}{f'(m)}+1\right)u_{xx}(-f(m)_{t}+u_{x}f(m)_{x})\\
= & Q(f(m))-2f(m)_{x}^{2}+\left(\frac{mf''(m)}{f'(m)}+1\right)u_{xx}(mf'(m)u_{xx}),
\end{split}
\ee
where we used the equation of $m$ in the last step. Putting together \rife{eq:---4} and \rife{eq:--5} yields (\ref{eq:f(m) equation}).
\end{proof}

As our goal is to understand what happens when $m$ vanishes or gets close to $0$, we need to introduce more geometric quantities related to the first order system \rife{mfgsys}. The first of these is the family of optimal trajectories associated to the HJ equation satisfied by $u$: we define  $\gamma:\R\times [0,T]\to \R$ as the solution to 
\be\label{def.gammaBIS}
\gamma_t(x,t)=-u_x(\gamma(x,t),t) \quad \text{in $\R\times [0,T]$}, \qquad \gamma(x,0)=x \; \qquad \text{in $\R$}.
\ee
By the standard theory of Hamilton-Jacobi equations, it is known that $\gamma(x, \cdot)$ is the minimizer of the problem 
$$
\inf_{\alpha\in H^1, \; \alpha(t)=x} \int_t^T \frac12 |\dot \alpha|^2+f(m(\alpha,s))\ ds + u(\alpha(T),T).
$$
The fundamental properties of  $\gamma$ are given next. 
\begin{lem}\label{lem.gamma}  Let $(u,m)$ be a classical solution to \rife{mfgsys}. One has
\be\label{eq.EulerBIS}
\gamma_{tt}(x,t)= f'(m(\gamma(x,t),t))m_x(\gamma(x,t),t), \qquad \forall (x,t)\in \R\times (0,T)\qquad \text{(Euler equation),}
\ee
and
\be\label{masspreservationBIS}
\gamma_x(x,t)= \frac{m_0(x)}{m(\gamma(x,t),t)}\qquad \forall (x,t)\in \R\times (0,T)\qquad \text{(conservation of mass)},
\ee
or, equivalently, 
\be\label{masspreservationTER}
\int_{\gamma(x_1,t)}^{\gamma(x_2,t)} m(x,t)dx = \int_{x_1}^{x_2}m_0(x)dx \qquad \forall x_1,x_2\in \R, \; t\in [0,T].
\ee
Moreover $\gamma$ solves the quasilinear elliptic equation
\be\label{eq.ell.gammaBIS}
-\frac{m_0 f'(m_0/\gamma_x)}{(\gamma_x)^3}\gamma_{xx}- \gamma_{tt}= -\frac{(m_0)_x f'(m_0/\gamma_x)}{(\gamma_x)^2}
\qquad \text{in $\R\times (0,T)$.}
\ee
\end{lem}

Before proving the lemma, it will be convenient to associate  to $(u,m)$ the solution $M$ to the transport equation 
\be\label{eq.transport}
M_t -u_x M_x = 0 \qquad \text{in $\R\times [0,T]$},\qquad M(x,0)=-\int_0^x m_0(y)dy \qquad \text{in $\R$}.
\ee
 Notice that \rife{def.gammaBIS} defines $\gamma$ as the  curve of characteristics associated to this transport equation.

\begin{lem}\label{lem.eqMMBIS}  Let $(u,m)$ be a classical solution to \rife{mfgsys}. One has $M_x=-m<0$ and $M$ satisfies the quasilinear elliptic equation 
\begin{equation} \label{eq.MMBIS}
- \frac{M_t^2}{M_x^2}  M_{xx}+2 \frac{M_t}{M_x}M_{xt}-M_{tt} +M_xf'(-M_x))M_{xx}=0 \qquad \text{in} \; \R\times (0,T).
\end{equation}
\end{lem}

\begin{proof} Differentiating \eqref{eq.transport} in space, we see that $\mu:=-M_x$ satisfies $\mu_t -(u_x\mu)_x=0$, with initial condition $\mu(0,x)=m_0$: this is exactly the equation satisfied by $m$, so that $M_x=-\mu=-m<0$. On the other hand, by definition $u_x= M_t/M_x$. Taking the derivative w.r.t. $x$ of the equation for $u$ \rife{eq.elluBIS}, we obtain \eqref{eq.MMBIS}. 
\end{proof}

\begin{proof}[Proof of Lemma \ref{lem.gamma}] As $\gamma$ solves \eqref{def.gammaBIS}, we have 
$$
\gamma_{tt}= -u_{xx}(\gamma,t)(-u_x(\gamma,t))-u_{xt}(\gamma,t) = f'(m(\gamma,t)) m_x(\gamma,t), 
$$
where the second equality comes from the derivation in space of the HJ  equation. This is \eqref{eq.EulerBIS}. As $M_x=-m$, \eqref{masspreservationBIS} comes from  the derivative in space of the transport equality $M(\gamma(x,t),t)=M_0(x)$. Then \eqref{masspreservationTER} follows by the integration in space of \eqref{masspreservationBIS}. 

By \eqref{eq.EulerBIS} and then \eqref{masspreservationBIS}, $m_x(\gamma,t)= \gamma_{tt}/ f'(m(\gamma,t))= \gamma_{tt} / f'(m_0/\gamma_x)$. On the other hand, taking the derivative in space of \eqref{masspreservationBIS} gives (using again \eqref{masspreservationBIS}  and the expression above for $m_x(\gamma,t)$):
$$
\gamma_{xx}=  \frac{(m_0)_x}{m(\gamma,t)}- \frac{m_0m_x(\gamma,t)\gamma_x}{(m(\gamma,t))^2}=\frac{(m_0)_x\gamma_x}{m_0} - \frac{\gamma_x^3}{f'(m_0/\gamma_x) m_0}\gamma_{tt}.
$$
This is \eqref{eq.ell.gammaBIS}.
\end{proof}

We finally compute the equation satisfied by $v(x,t)= f(m(\gamma(x,t),t))$. The map $v$ is the r.h.s. of the HJ equation viewed from the lens of the optimal trajectories. 

\begin{lem}\label{eqfmgam}  Let $(u,m)$ be a classical solution to \rife{mfgsys}. The map $v(x,t)= f(m(\gamma(x,t),t))$ satisfies the quasilinear elliptic equation in divergence form:
\be\label{eq.v=fmgamma}
-\left(\frac{v_x}{\gamma_x}\right)_x  - \left(\frac{\gamma_x^2}{m_0f'(m_0/\gamma_x)}v_t\right)_t= 0 .
\ee
If  $f(m)=m^\theta$,  this equation simplifies into  
\be\label{eq.v=fmgammamtheta}
(i)\; -\left(\frac{v_x}{\gamma_x}\right)_x  - \left(\frac{\gamma_x}{\theta v}v_t\right)_t= 0,\qquad\text{or}\qquad  (ii)\; -v_{tt}-\frac{\theta v}{\gamma_x^2}v_{xx}+v_x\frac{\theta v}{\gamma_x^3}\gamma_{xx} + \frac{\theta+1}{\theta} v^{-1}v_t^2=0 , 
\ee
while if $f(m)=\log(m)$, it becomes 
\be\label{eq.v=fmgammalog}
-\left(\frac{v_x}{\gamma_x}\right)_x  - \left(\gamma_x v_t\right)_t= 0.
\ee
\end{lem}

\begin{proof} 

 Observe that \eqref{eq.EulerBIS} may be written as
\be\label{gatt}
 \gamma_{tt}=\frac{v_x}{\gamma_x}.
\ee
We now compute the time derivative of $v(x,t)= f(m_0(x))/\gamma_x(x,t))$ to find 
\be\label{hkajzejrsbdfn}
v_t= f'(m_0/\gamma_x)\left(-\frac{m_0 \gamma_{xt}}{\gamma_x^2}\right).
\ee
Thus,  putting together \rife{gatt} and \rife{hkajzejrsbdfn} we get
$$
 \left(\frac{v_x}{\gamma_x}\right)_x = \gamma_{xtt} = - \left(\frac{v_t\gamma_x^2}{m_0f'(m_0/\gamma_x)}\right)_t,
$$

which gives \eqref{eq.v=fmgamma}, and thus \eqref{eq.v=fmgammamtheta}-(i)  and \eqref{eq.v=fmgammalog}. Finally, if $f(m)=m^\theta$,  developing \eqref{eq.v=fmgammamtheta}-(i) and using \eqref{hkajzejrsbdfn}---which implies that $\gamma_{xt}= -(\gamma_xv_t)/(\theta v)$---leads to  \eqref{eq.v=fmgammamtheta}-(ii). 
\end{proof}

\subsection{Displacement convexity estimates on $m$, Lipschitz estimates on $u$ and existence result}\label{subsection:stime}
Before studying the problem with compactly supported marginals, we
will begin obtaining some estimates for the simpler
periodic setting with strictly positive marginals. We are mostly interested
in a priori estimates that are independent of $\min m_{0}$ and $\min m_{T}$.
By approximation with positive densities $m_{0}^{\vep},m_{T}^{\vep}$,
these estimates will hold for the case in which $m_{0}$ and $m_{T}$
are compactly supported. 

Throughout the remainder of this section, $R\geq1$ will denote a fixed constant, and we will 
analyze the MFG system on the one-dimensional torus of length $R$, denoted by $\T$.
Functions defined on $\T$ are meant to be $R$-periodic functions on $\R$. 
We will consider the problem
\begin{equation}
\begin{cases}
-u_{t}+\frac{1}{2}u_{x}^{2}=f(m) & (x,t)\in\T\times(0,T),\\
m_{t}-(mu_{x})_{x}=0 & (x,t)\in\T\times(0,T),
\end{cases}\label{eq:Per MFG}
\ee
complemented either with  the  initial-terminal conditions
\begin{equation}\label{couplingT}
m(x,0)=m_{0}(x),\,\,\quad u(x,T)=g(m(x,T)) \,,\quad  x\in\T,
\ee
or with the prescribed marginal conditions of the planning problem:
\be\label{per-MFGP}
m(x,0)=m_{0}(x),\quad \,\,m(x,T)=m_{T}(x)\,,\quad  x\in\T.
\end{equation}
The functions $f,g:(0,\infty)\rightarrow\mathbb{R}$ are assumed to
satisfy $f',g'>0$, with $f,g \in C^{2}(0,\infty)$ 
and
\begin{equation}
\limsup_{m\rightarrow0^{+}}mf'(m)<\infty,\,\,\limsup_{m\rightarrow0^{+}}\frac{m|f''(m)|}{f'(m)}<\infty.
\label{eq:polynomial growth}
\end{equation}
The initial (and terminal) data 
$m_{0},m_{T}:\T\rightarrow(0,\infty)$ are understood to be $C^{1}$
functions satisfying
\begin{equation} \label{perMassAssumption}
\int_{\T}m_{0}=\int_{\T}m_{T},\quad  m_0,m_T >0.
\end{equation}
 Moreover, we also assume that
\begin{equation}
f(m_{0}),f(m_{T})\in C^{1,\alpha}(\T).\label{eq:marginals C1,alpha}
\end{equation}
The fact that the system \rife{eq:Per MFG} has smooth solutions when the marginals
are strictly positive and the data are sufficiently smooth is already
known (see \cite[Theorem 1.1]{Munoz}).   In fact, we will see later
(Theorem \ref{thm:existence}) that, under the present assumptions,
(\ref{eq:Per MFG})--\rife{couplingT} and (\ref{eq:Per MFG})--\rife{per-MFGP} both admit  a classical solution
$(u,m)\in C^{2,\alpha}(\T\times(0,T))\times C^{1,\alpha}(\T\times(0,T))$. 

 We begin by recalling the so-called displacement convexity formula (see \cite{GomesSeneci,MimikosMunoz}), as well as an identity which will later be useful to obtain energy estimates on the density.
\begin{prop}  Assume that $f\in C^{1}(0,\infty)$, and
let $(u,m)$ be a classical solution to \rife{eq:Per MFG}.
Then we have
\begin{equation}
mf'(m)u_{xx}^{2}+f(m)_{x}^{2}=(f(m)_{t}u_{x})_{x}-(f(m)_{x}u_{x})_{t}.\label{eq:displacement 1}
\end{equation}
Moreover, if  $h:(0,\infty)\rightarrow\mathbb{R}$ is twice differentiable,
then
\begin{equation}
\frac{d^{2}}{dt^{2}}\int_{\T}h(m)=\int_{\T}mh''(m)(mu_{xx}^{2}+f'(m)m_{x}^{2}).\label{eq:displacement integral}
\end{equation}
\end{prop}

\begin{proof}
We start by multiplying the continuity equation by $u_{xx}$, which
yields

\[
mu_{xx}^{2}+m_{x}u_{x}u_{xx}-m_{t}u_{xx}=0.
\]
As a result, differentiating the HJ equation for the term $u_xu_{xx}$ we obtain

\begin{equation}
mu_{xx}^{2}+m_{x}(u_{xt}+f'(m)m_{x})-m_{t}u_{xx}=0,\label{eq:}
\end{equation}
which, after multiplying by $f'(m)$, yields

\[
mf'(m)u_{xx}^{2}+f'(m)^{2}m_{x}^{2}=(f(m)_{t}u_{x})_{x}-(f(m)_{x}u_{x})_{t}.
\]
This proves (\ref{eq:displacement 1}). We note that (\ref{eq:displacement integral})
is merely special case of \cite[Proposition 3.1]{MimikosMunoz}, but we give a proof for the reader's convenience.  Multiplying both sides
of (\ref{eq:}) by $mh''(m)$, we obtain

\[
mh''(m)(mu_{xx}^{2}+f'(m)m_{x}^{2})=h''(m)mm_{t}u_{xx}-h''(m)m_{x}u_{xt}=H(m)_{t}u_{xx}-H(m)_{x}u_{xt},
\]
where $H(m)=mh'(m)-h(m)$. This may be rewritten as 

\begin{equation}
mh''(m)(mu_{xx}^{2}+f'(m)m_{x}^{2})=(H(m)_{t}u_{x})_{x}-(H(m)_{x}u_{x})_{t}\label{eq:1-Dis}
\end{equation}
Now, from the continuity equation,
\[
h(m)_{t}=h(m)_{x}u_{x}+mh'(m)u_{xx}=(h(m)-mh'(m))_{x}u_{x}+(mh'(m)u_{x})_{x}=-H(m)_{x}u_{x}+(mh'(m)u_{x})_{x}.
\]
Therefore,
\[
(H(m)_{x}u_{x})_{t}=-h(m)_{tt}+(mh'(m)u_{x})_{tx}.
\]
Substituting in (\ref{eq:1-Dis}), we obtain
\begin{multline*}
mh''(m)(mu_{xx}^{2}+f'(m)m_{x}^{2})  =h(m)_{tt}+(H(m)u_{x})_{x}-(mh'(m)u_{x})_{tx}
 \\=h(m)_{tt}-(h(m)_{t}u_{x})_{x}-(mh'(m)u_{xt})_{x}
 =h(m)_{tt}-(h'(m)(mu_{x})_{t})_{x},
\end{multline*}
and (\ref{eq:displacement integral}) then
follows by integrating both sides of this equation in space.
\end{proof}
We now note that the density attains its extremum values at the extremal times.
\begin{cor}
\label{cor: m bound} Assume that $f\in C^{1}(0,\infty)$, and
let $(u,m)$ be a classical solution to (\ref{eq:Per MFG}). Then
\begin{align*}
 \|m\|_{\infty}  & \leq\max(\|m(\cdot, 0) \|_{\infty},\|m(\cdot,T)\|_{\infty}),
\\
\left\Vert m^{-1}\right\Vert _{\infty}  &  \leq\max\left(\left\Vert m(\cdot, 0)^{-1}\right\Vert _{\infty},\left\Vert m(\cdot,T)^{-1}\right\Vert _{\infty}\right).
\end{align*}
\end{cor}

\begin{proof}
One first observes that, for any convex function $h:(0,\infty)\rightarrow \R$, it follows from  (\ref{eq:displacement integral}) that $\int_{\T}h(m)$ is convex in time, which yields
$$
\int_{\T}h(m) \leq \max\left(\int_{\T}h(m(\cdot,0)), \int_{\T}h(m(\cdot,T))\right) \,.
$$
The upper bounds on $m$ and $m^{-1}$, then follow by taking $h(m)=m^{p}$  
and letting $p\rightarrow\pm\infty$, respectively.
\end{proof}
\begin{rem}
    These a priori estimates were proved in \cite{GomesSeneci}. They could have also been derived without the displacement convexity formula, as an immediate consequence of Lemma \ref{lem:minimum principle} below.
\end{rem}
We will now review the Lipschitz estimates which can be established on $u$ following the approach suggested by P.-L. Lions in \cite{L-college}, and developed later in more generality in \cite{Munoz,Porretta}. The following $L^{\infty}$ bounds on $u$ and $m(\cdot,T)$ are well-known
consequences of the maximum principle and the Hopf-Lax formula (see
\cite[Propositions 4.1 and 4.2]{MimikosMunoz}).
\begin{prop}
\label{prop: oscillation bound}  Assume that $f,g \in C^1(0,\infty)$, $f'>0, g'\geq0,$ and $m_0,m_T\in C(\mathbb{T})$. If $(u,m)$ is a classical solution to \eqref{eq:Per MFG}--\rife{couplingT},
then we have, for $(x,t)\in\T\times[0,T]$,
\[
\min m_{0}\leq m(x,T)\leq\max m_{0},
\]
\[
f(\min m_{0})(T-t)+g(\min m_{0})\leq u(x,t)\leq f(\max m_{0})(T-t)+g(\max m_{0}).
\]
Moreover, there exists a constant $C=C(\|f(m_{0})\|_{\infty},\|f(m_{T})\|_{\infty})$
such that, if $(u,m)$ solves \rife{eq:Per MFG}--\rife{per-MFGP}, then
\[
\text{\emph{osc}}\,(u)\leq C(T+R^{2}T^{-1}).
\]
\end{prop}

We may now obtain a gradient estimate which, crucially, is independent of $\min m$. Hereafter, we denote by $Du$ the vector formed by the space and time first derivatives, that is,
$$
Du:= (u_x, u_t)\,.
$$
Moreover, we   denote $\kappa_{0}>0$ to
be a constant such that
\begin{equation}
 \,\, mf'(m),\frac{m|f''(m)|}{f'(m)}\leq \kappa_{0} \qquad \forall \,\,\, m\in(0,2\max(\|m_{0}\|_{\infty},\|m_{T}\|_{\infty})]\,.
\label{eq:polynomial growth f restatement}
\end{equation}
Such a constant exists by virtue of (\ref{eq:polynomial growth}). It is understood that the term with $\|m_{T}\|_{\infty}$ is treated as zero,  in case of conditions \rife{couplingT}.

\begin{prop}
\label{prop:gradient bound} Assume that \eqref{eq:polynomial growth}, \eqref{perMassAssumption}, and \rife{eq:marginals C1,alpha} hold true, and let   $(u,m)$  be  a classical solution to 
\eqref{eq:Per MFG}--\rife{couplingT} or \eqref{eq:Per MFG}--\rife{per-MFGP}.
There exists a constant $C$ 
such that
\begin{align*}
\|Du\|_{\infty} & \leq C.
\end{align*}
where
\[
C= \begin{cases} C(\kappa_{0},R,T,T^{-1},\|f(m_{0})\|_{W^{1,\infty}}, \|g(m_0)\|_\infty) & \text{if } (u,m) \text{ solves \eqref{eq:Per MFG}--\rife{couplingT} },\\
C(\kappa_{0}, R,T,T^{-1},\|f(m_{0})\|_{W^{1,\infty}},\|f(m_{T})\|_{W^{1,\infty}}) & \text{if } (u,m) \text{ solves \eqref{eq:Per MFG}--\rife{per-MFGP}},
\end{cases} \]
\end{prop}

\begin{proof}
By rescaling, it is enough to consider the case $R=1$. By approximation\footnote{See the proof of Theorem \ref{thm:existence} for the details of such an approximation argument.},
we may also assume that $u\in C^{3}(\mathbb{T}\times[0,T])$ . We begin
by noting that, using either Corollary \ref{cor: m bound} or Proposition \ref{prop: oscillation bound}, we have that $\|m\|_\infty$ is controlled by $\|m_0\|_\infty$,  and by  $\|m_T\|_\infty$ in case of \rife{per-MFGP}. Hence we will use \rife{eq:polynomial growth f restatement} for $m(x,t)$ below.
We let $v(x,t)=\frac{1}{2}u_{x}^{2}+\frac{1}{2T}\tilde{u}^{2}$,
where $\tilde{u}$ is defined as
\[
\tilde{u}=u-\min u+T-\frac{(\text{osc}\,(u)+2T)}{T}(T-t),
\]
so that $\tilde{u}(x,0)=u-\max u-T\leq-T$, $\tilde{u}(x,T)=u-\min u+T\geq T$,
and $\|\tilde{u}\|_{\infty}\leq(T+\text{osc}\,(u))$. Let $(x_{0},t_{0})$
be a point in which $v$ achieves its maximum value. We note that,
by the HJ equation, it is enough to bound $v$. If $t_{0}=0$, then
we have
\[
-v_{t}+u_{x}v_{x}=u_{x}f(m_{0})_{x}+\frac{1}{T}\tilde{u}(f(m_{0})+ \frac{1}{2}u_{x}^{2}-\frac{1}{T}(\text{osc}\,(u)+2T)).
\]
Now, we have $v_{x}=0$ and $v_{t}\leq0$, and recall that $\tilde u(x,0)\leq -T$; so either $f(m_{0}) + \frac{1}{2}u_{x}^{2}-\frac{1}{T}(\text{osc}\,(u)+2T)\leq0$,
in which case there is nothing to prove, or we deduce
$$
0\leq u_{x}f(m_{0})_{x}- (f(m_{0})+ \frac{1}{2}u_{x}^{2}-\frac{1}{T}(\text{osc}\,(u)+2T))\,.
$$
This means 
\[
\frac{1}{2}u_{x}^{2} \leq    u_{x}f(m_{0})_{x}- f(m_{0})+  \frac{1}{T}(\text{osc}\,(u) +2T),
\]
and yields the required estimate. The case in which $t_{0}=T$
is similar, so we assume now that $0<t_{0}<T$. We begin by noting
that, since $v_{x}=0,$
\[
u_{x}u_{xx}=-\frac{1}{T}u_{x}\tilde{u},
\]
and, thus, 
\begin{equation}
|u_{xx}|\leq \frac{C}{T}(T+\text{osc}(u)).\label{eq:uxx bd}
\end{equation}
Now, recall the definitions of the linear operators $Q,L$ in \rife{eq:Q definition}, \rife{eq:linearization}. Using \rife{eq:-3} we have
\begin{equation}
L\left(\frac{1}{2}u_{x}^{2}\right) \leq-(u_{xt}-u_{x}u_{xx})^{2}\,.\label{eq:-3new}
\end{equation}
On the other hand, since $Q(\tilde u)=0$,   we have
\[
Q\left(\frac{1}{2}\tilde{u}^{2}\right)=-( -\tilde{u}_{t}+u_{x}^{2})^{2}-mf'(m)\tilde{u}_{x}^{2}\leq-( -\tilde{u}_{t} +u_{x}^{2})^{2}=-\left(f(m)+\frac{1}{2}u_{x}^{2}-\frac{1}{T}(\text{osc}\,(u)+2T)\right)^{2}\,.
\]
%
%
Observe that, by Corollary \ref{cor: m bound} and Proposition \ref{prop: oscillation bound}, $|f(m)|$ is bounded. Therefore, if $f(m)+\frac{1}{4}u_{x}^{2}\leq\frac{1}{T}(\text{osc}\,(u)+2T)$,
there is nothing to prove. Else,
\[
Q\left(\frac{1}{2}\tilde{u}^{2}\right)\leq-\frac{1}{32}u_{x}^{4}-\frac{1}{2}( -\tilde{u}_{t} +u_{x}^{2})^{2},
\]
and, thus, by definition of $L$, we obtain, using  (\ref{eq:polynomial growth f restatement}) and (\ref{eq:uxx bd}),
\begin{align*}
L\left(\frac{1}{2T}\tilde{u}^{2}\right) & \leq-\frac{1}{32T}u_{x}^{4}-\frac{1}{2T}(-\tilde{u}_{t}+u_{x}\tilde{u}_{x})^{2}+\frac{2}{T}(u_{xt}-u_{x}u_{xx})u_{x}\tilde{u}-\frac{1}{T}\left(\frac{mf''(m)}{f'(m)}+1\right)u_{xx}(-\tilde{u}_{t}+u_{x}\tilde{u}_{x})\tilde{u}\\
 & \leq-\frac{1}{32T}u_{x}^{4}+\frac{1}{2T}\left(\kappa_0+1\right)^{2}(u_{xx})^{2}\tilde{u}^2 +\frac{4}{T^{2}}u_{x}^{2}\tilde{u}^{2}+(u_{xt}-u_{x}u_{xx})^{2} \\
 & \leq-\frac{1}{64T}u_{x}^{4}+\frac{C}{T^{3}}(T+\text{osc}\,(u))^{4}+ (u_{xt}-u_{x}u_{xx})^{2} \,.
\end{align*}
Putting toghether the above inequality with \rife{eq:-3new} we get
\[
L(v)\leq-\frac{1}{64T}u_{x}^{4}+\frac{C}{T^{3}}(T+\text{osc\,}(u))^{4}.
\]
Now, since $(x_{0},t_{0})$ is a maximum point for $v$, we have $L(v)\geq0$,
which yields
\[
u_{x}^{4}\leq\frac{C}{T^{2}}(T+\text{osc}\,(u))^{4}.
\]
Recalling that $\text{osc\,}(u)$ is estimated from Proposition \ref{prop: oscillation bound}, we conclude the estimate.
\end{proof}

We   now show that, under the present assumptions, (\ref{eq:Per MFG})--\rife{couplingT} 
and (\ref{eq:Per MFG})--\rife{per-MFGP} may be solved classically.
\begin{thm}
\label{thm:existence}  Assume that conditions \rife{eq:polynomial growth},  \eqref{perMassAssumption}, and \rife{eq:marginals C1,alpha} hold true. 
Then the systems \eqref{eq:Per MFG}--\rife{couplingT}  and \eqref{eq:Per MFG}--\rife{per-MFGP}
have a classical solution $(u,m)\in C^{2,\alpha}(\T\times[0,T])\times C^{1,\alpha}(\T\times[0,T])$,
with $m$ being unique. In the case of (\ref{couplingT}), $u$ is
unique, and in the case of (\ref{per-MFGP}), $u$ is unique up
to a constant. 
\end{thm}

\begin{proof}
The uniqueness is a standard result, proved through duality. We will
do the proof of existence for (\ref{eq:Per MFG})--\rife{per-MFGP};  
the alternative case of (\ref{eq:Per MFG})--\rife{couplingT} requires only minor modifications.
We may approximate  $f$ and the marginals with $f^{\epsilon}\in C^{4}(0,\infty)$, $m_{0}^{\vep},m_{T}^{\vep}\in C^{4}(\T)$, such that $f^{\epsilon}(m_{0}^{\vep}),f^{\epsilon}(m_{T}^{\vep})$
is uniformly bounded in $C^{1,\alpha}(\T)$. Indeed, we may simply
take
\begin{equation}
 f^{\epsilon}=f\ast\eta_{\epsilon,} \quad m_{0}^{\vep}=(f^{\epsilon})^{-1}(f^{\epsilon}(m_{0})\ast\eta_{\vep}+c_{0,\vep}),\quad m_{T}^{\vep}=(f^{\epsilon})^{-1}(f^{\epsilon}(m_{T})\ast\eta_{\vep}+c_{T,\vep}),\label{eq:convolution approx}
\end{equation}
where $\eta^{\vep}$ is
the standard mollifier, and the non-negative constants $c_{0,\vep},c_{T,\vep}$
are adequately chosen such that $c_{0,\vep}\cdot c_{T,\vep}=0$,
$\int_{\T}\me_{0}=\int_{\T}\me_{T}$ and $\lim_{\vep\rightarrow0}c_{0,\vep}=\lim_{\vep\rightarrow0}c_{T,\vep}=0$.
With these regularized data, (\ref{eq:Per MFG})--\rife{per-MFGP} has a unique classical
solution $(\ue,\me)\in C^{3}(\T\times[0,T])\times C^{2}(\T\times[0,T])$
satisfying $\int_{\T}\int_{0}^{T}\ue=0$ (see \cite[Thm 1.1]{MimikosMunoz}).
Moreover, in view of Propositions \ref{prop: oscillation bound} and
\ref{prop:gradient bound}, the solution is bounded in $C^{1}\times C^{0}$,
uniformly in $\vep$. The result will then follow by letting $\vep\rightarrow0$
and applying a version of R. Fiorenza's convergence theorem for elliptic
oblique problems (see, for instance, \cite[Lem. 17.29]{GilbargTrudinger}, \cite[Lem. 2, Cor. 1]{Lieberman solvability}, \cite{Fiorenza, Fiorenza-1}).
For completeness, we sketch the details for this argument, which amounts
to a proof of Fiorenza's result. From Lemma \ref{eq.elluBIS}, the functions $\ue$ solve the oblique
quasilinear elliptic problem

\[
\begin{cases}
-\text{Tr}(A(\ue_{x},\ue_{t})D^{2}\ue)=0 & (x,t)\in\T\times[0,T],\\
-\ue_{t}(x,0)+\frac{1}{2}(\ue_{x})^{2}(x,0)=f(m_{0}^{\vep}(x)) & x\in\T,\\
-\ue_{t}(x,T)+\frac{1}{2}(\ue_{x})^{2}(x,T)=f(m_{T}^{\vep}(x)) & x\in\T,
\end{cases}
\]
where 

\[
A(\ue_{x},\ue_{t})=\begin{pmatrix}(\ue_{x})^{2}+\me f'(\me) & -\ue_{x}\\
-\ue_{x} & 1
\end{pmatrix}=\begin{pmatrix}(\ue_{x})^{2}+f^{-1}(-\ue_{t}+\frac{1}{2}(\ue_{x})^{2})f'(f^{-1}(-\ue_{t}+\frac{1}{2}(\ue_{x})^{2})) & -\ue_{x}\\
-\ue_{x} & 1
\end{pmatrix}.
\]
Since $\|\ue\|_{C^{1}}$, $\|\me\|_{\infty}$ $\|(\me)^{-1}\|_{\infty}$
are uniformly bounded, this equation is uniformly elliptic, uniformly
in $\vep$. Thus, as a result of Lieberman's $C^{1,s}$ estimate
for oblique problems (see \cite[Lem. 2.3]{Lieberman}),
there exists $0<s<1$ and a constant $C>0$ such that 
\[
\|u^{\vep}\|_{C^{1,s}}\leq C.
\]
Therefore if $0<\vep'<1$, the difference $v=u^{\vep}-u^{\vep'}$
solves
\[
\begin{cases}
-\text{Tr}(A^{\vep}D^{2}v)=\text{Tr}((A^{\vep}-A^{\vep'})D^{2}u^{\vep'}) & (x,t)\in\T\times[0,T],\\
-v_{t}(x,0)+\frac{1}{2}(\ue_{x}(x,0)+u_{x}^{\vep'}(x,0))v_{x}(x,0)=f(m_{0}^{\vep}(x))-f(m_{0}^{\vep'}(x)) & x\in\T,\\
-v_{t}(x,T)+\frac{1}{2}(\ue_{x}(x,T)+u_{x}^{\vep'}(x,T))v_{x}(x,0)=f(m_{T}^{\vep}(x))-f(m_{T}^{\vep'}(x)) & x\in\T.
\end{cases}
\]
Let $\alpha'=\min(\alpha,s)$. The standard Schauder estimates for
linear oblique problems (see, for instance, \cite[Lem. 1]{Lieberman solvability})
imply that 
\begin{multline*}
\|v\|_{C^{2,\alpha'}}\leq C(\|v\|_{\infty}+\|u^{\vep'}\|_{C^{2,\alpha'}}\|v\|_{C^{1}}+\|u^{\vep'}\|_{C^{2}}\|v^{\vep}\|_{C^{1,\alpha'}}+\|f(m_{0}^{\vep})-f(m_{0}^{\vep'})\|_{C^{1,\alpha'}}\\
+\|f(m_{T}^{\vep})-f(m_{T}^{\vep'})\|_{C^{1,\alpha'}}+\|v\|_{C^{1}}\|u^{\vep}+u^{\vep'}\|_{C^{2,\alpha'}}).
\end{multline*}
Recalling the interpolation inequality for H\"older spaces, $\| \cdot \|_{C^{2}} \leq \delta \| \cdot  \|_{C^{2,\alpha}} + C_{\delta} \| \cdot \|_{C^{0}}$, we deduce that
\[
\|u^{\vep}-u^{\vep'}\|_{C^{2,\alpha'}}\leq C(o(1)(1+\|u^{\vep}\|_{C^{2,\alpha'}}+\|u^{\vep'}\|_{C^{2,\alpha'}})+\|u^{\vep'}\|_{C^{2}}+1)\leq\frac{1}{4}(\|u^{\vep}\|_{C^{2,\alpha'}}+\|u^{\vep'}\|_{C^{2,\alpha'}}))+C
\]
as $\vep,\vep'\rightarrow0$. Fixing a small $\vep'$,
and letting $\vep\rightarrow0$, we see that $\|u^{\vep}\|_{C^{2,\alpha'}}$
must be bounded. Now repeating the same argument but taking $\alpha'=\alpha$,
we see that $\|u^{\vep}\|_{C^{2,\alpha}}$ must be bounded as
well, which means $u^{\vep}$ converges to a solution $u\in C^{2,\alpha}(\T \times [0,T])$. In turn, the $C^{1,\alpha}$ regularity of $m$ follows from the HJ equation and the $C^2$ regularity of $f$.
\end{proof}

\subsection{Continuity of the density}

In this section, we will prove that the function $f(m)$ satisfies a uniform modulus
of continuity, independent of $\min m$. This estimate is a crucial step in treating the setting of compactly supported solutions, because it will allow us to prove that the density is globally continuous, despite the lack of a positive lower bound. 

\begin{lem}
\label{lem:minimum principle} Under the assumptions of Theorem \ref{thm:existence}, let $(u,m)$ be a classical solution to \eqref{eq:Per MFG}.
Then the function $v=f(m)$ satisfies the maximum
principle and the minimum principle on each compact subset of $\T\times[0,T]$.
\end{lem}

\begin{proof}
By approximation, we may assume that $u$ and $m$ are smooth. Hence $v=f(m)$ satisfies \rife{eq:f(m) equation}.
%
%
Now, from the continuity equation,
we have
\[
mf'(m)u_{xx}=f(m)_{t}-f(m)_{x}u_{x}= v_t-v_x u_x,
\]
and, therefore, substituting in (\ref{eq:f(m) equation}) yields
\begin{equation}
Q(v)-v_{x}^{2}+m(mf''(m)+2f'(m))(mf'(m))^{-2}(v_{t}-v_{x}u_{x})^{2}=0.
\label{eq:--4}
\end{equation}
We notice that $Q$ is a purely second order linear elliptic operator,
and the remaining terms of (\ref{eq:--4}) can  be written as
 first--order terms in $v$. Thus, $v$ satisfies an elliptic
equation with no zero--order terms. This implies that $v$ satisfies
the maximum and the minimum principle on every compact subset of $\T\times[0,T]$,
as wanted.
\end{proof}
We now compute an energy estimate for the function $v$.  Recall that $\kappa_0$ is given by \rife{eq:polynomial growth f restatement}.
\begin{prop}
\label{prop:energy estimate}  Under the assumptions of Theorem \ref{thm:existence}, let $(u,m)$ be a classical solution to \eqref{eq:Per MFG}--\rife{couplingT} or 
 \eqref{eq:Per MFG}--\rife{per-MFGP}, and let $v=f(m)$. Then there exists  $C$
such that
\[
\int_{0}^{T}\int_{\T}|Dv|^{2}\le C
\]
where
\[
C= \begin{cases} C(\kappa_{0},R,T,T^{-1},\|f(m_{0})\|_{W^{1,\infty}}, \|g(m_0)\|_\infty) & \text{if } (u,m) \text{ solves \eqref{eq:Per MFG}--\rife{couplingT} },\\
C(\kappa_{0}, R,T,T^{-1},\|f(m_{0})\|_{W^{1,\infty}},\|f(m_{T})\|_{W^{1,\infty}}) & \text{if } (u,m) \text{ solves \eqref{eq:Per MFG}--\rife{per-MFGP}},
\end{cases} \]
\end{prop}

\begin{proof}
Integrating (\ref{eq:displacement 1}) in space-time yields

\[
\int_{0}^{T}\int_{\T}mf'(m)u_{xx}^{2}+f(m)_{x}^{2}=\int_{\T}f(m)_{x}u_{x}(0)-f(m)_{x}u_{x}(T)\leq C-\int_{\T}f(m)_{x}u_{x}(T),
\]
where we used the gradient bound from Proposition \ref{prop:gradient bound}. If $(u,m)$ solves \rife{per-MFGP}, then we use the bound on $u_x$ and $f(m_T)_x$. If  $(u,m)$ solves (\ref{eq:Per MFG}), we have $-f(m)_{x}u_{x}(T)=-f'(m)g'(m)^{-1}u_{x}(T)^{2}<0$; 
so, in any case, we obtain
\be\label{fmx}
\int_{0}^{T}\int_{\T}mf'(m)u_{xx}^{2}+f(m)_{x}^{2}\leq C\,.
\ee
The bound on $f(m)_{t}$ simply follows from the continuity equation:
\[
f(m)_{t}^{2}=(f(m)_{x}u_{x}+mf'(m)u_{xx})^{2}\leq2(f(m)_{x}^{2}u_{x}^{2}+(mf'(m))^{2}u_{xx}^{2})\leq C(f(m)_{x}^{2}+mf'(m)u_{xx}^{2}),
\]
where the bound on $u_x$ and (\ref{eq:polynomial growth f restatement}) were used in the
last inequality. Integrating and using \rife{fmx}, we get the $L^2$ bound for $f(m)_t$. Finally, we have proved that
\[
\int_{0}^{T}\int_{\T}f(m)_{t}^{2}+f(m)_{x}^{2}\leq C,
\]
where $C$ depends on the same quantities as the bound of $u_x$ in Proposition \ref{prop:gradient bound}.
\end{proof}

The interior continuity now follows from a classical computation, originally attributed
to H. Lebesgue \cite{Lebesgue}, which implies that a sufficient condition for
a $W^{1,2}$ function in two variables to be continuous, is for it to satisfy the maximum and minimum principle.
\begin{prop}[Interior modulus of continuity]
\label{prop:interior modulus}
 Under the assumptions of Theorem \ref{thm:existence}, let
$(u,m)$ be a classical solution to (\ref{eq:Per MFG}). Then 
the function $v=f(m)$ has the following logarithmic modulus of continuity,
valid for all concentric balls $B_{r_{1}}$ and $B_{r_{2}}$, $r_{1}\leq r_{2}$
contained in $Q_{T}$. 
\[
(\emph{\text{osc}}_{B_{r_{1}}}\,(v))^{2}\log\left(\frac{r_{2}}{r_{1}}\right)\leq\pi\iint_{B_{r_{2}}}|Dv|^{2}.
\]
\end{prop}
\begin{proof}
Let $r\in[r_{1},r_{2}]$, and let $\theta_{1},\theta_{2}\in[0,2\pi]$.
Then, using polar coordinates with origin at the center of the balls $B_{r_{i}}$,
\[
v(r,\theta_{2})-v(r,\theta_{1})=\int_{\theta_{1}}^{\theta_{2}}\frac{\partial v}{\partial\theta}d\theta.
\]
Thus, integrating over a half circle and using the Cauchy-Schwarz
inequality,
\[
\text{osc}_{\partial B_{r}}(v)\leq\sqrt{\pi}\sqrt{\int_{0}^{2\pi}\left(\frac{\partial v}{\partial\theta}\right)^{2}d\theta}.
\]
Now, in view of Lemma \ref{lem:minimum principle}, $v$ satisfies
the maximum and minimum principle, so
\[
\text{osc}_{B_{r_{1}}}(v)\leq\text{osc}_{\partial B_{r}}(v).
\]
and, thus,
\[
(\text{osc}_{B_{r_{1}}}(v))^{2}\leq\pi\int_{0}^{2\pi}\left(\frac{\partial v}{\partial\theta}\right)^{2}d\theta.
\]
On the other hand, we have
\[
|Dv|^{2}=\left(\frac{\partial v}{\partial r}\right)^{2}+\frac{1}{r^{2}}\left(\frac{\partial v}{\partial\theta}\right)^{2}\geq\frac{1}{r^{2}}\left(\frac{\partial v}{\partial\theta}\right)^{2},
\]
which then implies that
\[
\frac{1}{r}(\text{osc}_{B_{r_{1}}}(v))^{2}\leq\pi\int_{0}^{2\pi}r|Dv|^{2}d\theta.
\]
Integrating in $r$ from $r_{1}$ to $r_{2}$ yields the result.
\end{proof}
A slight variant of this argument, by integrating over semi--disks
instead, yields the continuity estimate up to the boundary.
\begin{prop}[Boundary modulus of continuity]
\label{prop:boundary modulus}
 Under the assumptions of Theorem \ref{thm:existence}, let  $x_{0}\in\mathbb{T}$ and $D_{r_{1}}$
and $D_{r_{2}}$ be upper semi--disks of radii $r_{1}<r_{2}$, centered
at $(x_{0},0)$. Then we have
\[
(\emph{\text{osc}}_{D_{r_{1}}}(v))^{2}\log\left(\frac{r_{2}}{r_{1}}\right)\leq 2\pi\int\int_{D_{r_{2}}}|Dv|^{2}+4(r_{2}^2-r_{1}^2)\|v_{x}(0)\|_{\infty}^{2}.
\]
Similarly, if $D_{r_{1}}$ and $D_{r_{2}}$ are lower semi--disks
centered at $(x_{0},T)$, one has 
\[
(\emph{\text{osc}}_{D_{r_{1}}}(v))^2\log\left(\frac{r_{2}}{r_{1}}\right)\leq 2\pi\int\int_{D_{r_{2}}}|Dv|^{2}+4(r_{2}^2-r_{1}^2)\|v_{x}(T)\|_{\infty}^{2}.
\]
\end{prop}

\section{Finite speed of propagation and compactly supported solutions}
\label{sec:finitespeed}

This section, which is the main core of the paper, is dedicated to the study of  solutions to problems \rife{eq:MFGR} and \rife{eq:MFGPR} with \emph{compactly supported density}. 

We will first  obtain a  preliminary existence (and uniqueness) result, under fairly general conditions on the coupling functions $f,g$, assuming the initial density $m_0$ (and possibly the terminal density $m_T$) to be compactly supported.  By a suitable choice of approximation of $m_0, m_T$, we will show the property of  finite speed of propagation of the support, and the existence of  a unique   continuous solution with compactly supported density $m$. This will allow us to build a rigorous framework for the study of the free boundary $\partial \{(x,t)\,:\, m(x,t)>0\}$, carried out in Subsection \ref{free-analysis}, and will be tightly connected to the analysis of the flow of optimal trajectories associated to the optimization problem.  
The study of the regularity and geometric properties of the free boundary, as well as of its spreading  speed,   will be analyzed for the model case of $f(m)= m^\theta$, $\theta>0$. We conclude by  establishing space-time H\"older regularity of the pair $(m,Du)$ up to the free boundary.

\subsection{Well-posedness results} \label{subsec:wellposed}

Throughout this subsection,  we will assume that 
\be\label{fg-ass}
\hbox{$f,g\in C^{2}(0,\infty)\cap C[0,\infty)$, $f'>0$, $g'\geq0$, and (\ref{eq:polynomial growth})
holds.}
\ee
For simplicity, and with no loss of generality, we will also
normalize $f$ so that 
\begin{equation}
f(0)=0.\label{eq:f(0)=00003D0}
\end{equation}
As for the initial density (and possibly terminal density, in case of \rife{eq:MFGPR}), we assume that 
$m_{0}$, $m_{T}\in C_{c}(\mathbb{R})$
are compactly supported, non-negative functions, such that
\begin{equation}
\int_{\mathbb{R}}m_{0}=\int_{\mathbb{R}}m_{T} =1,\label{eq:equal mass assumption}
\end{equation}
and
\begin{equation}
m_{0},m_{T}\text{ vanish, respectively, outside the intervals }[a_{0},b_{0}]\text{,}[a_{1},b_1].\label{eq:compact support assumption}
\end{equation}
Rather than (\ref{eq:marginals C1,alpha}), we will instead assume that, for some $\alpha\in (0,1)$,
\begin{equation}
f(m_{0})\in C^{1,\alpha}(\{m_0>0\}),f(m_{T})\in C^{1,\alpha}(\{m_T>0\}).\label{eq:f(marginals) interior holder}
\end{equation}
We remark that, in particular, this assumption implies that $f(m_0),f(m_T)\in W^{1,\infty}(\R)$, but it allows for the possibility that $f(m_0)_x$ or $f(m_T)_x$ might be discontinuous at the boundary of the support, when considered as functions in $\R$.

Our first goal will be to show the   well-posedness of  systems \rife{eq:MFGR} and \rife{eq:MFGPR}. 
%
Since we know, from the model case of Section \ref{sec:self similar} that, in general, the solution is not classical, we must work with a preliminary notion of generalized solution.  We recall that $C_b(\mathbb{R} \times [0,T])$ denotes the space of bounded continuous functions.

\begin{defn}\label{gensol}
We say that $(u,m)\in W^{1,\infty}(\mathbb{R}\times(0,T))\times C_{b}(\mathbb{R}\times[0,T])$
is a  solution to (\ref{eq:MFGR}) (respectively, (\ref{eq:MFGPR})),
if
\end{defn}

\begin{itemize}
\item[(i)] $u$ is a viscosity solution to the HJ equation
\[
-u_{t}+\frac{1}{2}u_{x}^{2}= f(m)\,\,\,\,\,\,(x,t)\in\mathbb{R}\times(0,T),
\]
\item[(ii)] $m$ satisfies the continuity equation
\[
m_{t}-(mu_{x})_{x}=0\,\,\,\,\,\, (x,t)\in\mathbb{R}\times(0,T),
\]
in the distributional sense, with $m(\cdot,0)=m_{0}$.
\item[(iii)] We have $u(\cdot,T)= g(m(\cdot, T))$  (respectively, $m(\cdot,T)=m_{T}$).

\end{itemize}

\begin{rem} Different notions of weak solutions could have been used, alternatively to  Definition \ref{gensol}.  In particular, distributional subsolutions of the HJ equation have been frequently used for MFG problems with a local coupling, both in case of final pay-off and in case of planning conditions (see e.g. \cite{Carda1}, \cite{CMS}, \cite{GMST}, \cite{OPS}, and the survey \cite{CP-CIME}). That approach is tightly related to the concept of relaxed minima of variational problems, and avoids, for instance, any requirement of continuity and boundedness of $m$ and $u_x$. Of course a similar approach would also apply to the present problems. However, since our primary goal in this article is the analysis of the free boundary for compactly supported solutions, it seems more natural to work from the beginning with the stronger notion of continuous solutions. It is also natural,  in that context, to make use of  the standard framework of viscosity solutions for HJ equations.   
\end{rem}

Our goal in this subsection will be to prove the following well-posedness result.  In what follows, $C_c(\R)$ denotes the space of continuous, compactly supported functions on $\R$. Notice that, in particular, the result shows that the unique solution is such that $m$ has compact support. 
 
\begin{thm}
\label{thm:well-posedness theorem}Assume that $f,g$ satisfy \rife{fg-ass}, $m_0,m_T\in C_c(\R)$ satisfy \rife{eq:equal mass assumption}--\rife{eq:f(marginals) interior holder}, and $\kappa_0$ is as in \eqref{eq:polynomial growth f restatement}. 
Then the following holds.
\begin{enumerate}
\item There exists a unique   solution $(u,m)$ to \eqref{eq:MFGR}. Moreover, $(u,f(m))\in C_{\emph{loc}}^{2,\alpha}((\mathbb{\mathbb{R}\times}[0,T))\cap\{m>0\})\times C_{\emph{loc}}^{1,\alpha}((\mathbb{\mathbb{R}\times}[0,T))\cap\{m>0\})$. There exists a constant 
\[
C=C(T,T^{-1}, \kappa_0,\|f(m_{0})\|_{W^{1,\infty}},|\emph{\text{supp}}(m_{0})|,\|g(m_{0})\|_{\infty})
\]
such that
\[
\|u\|_{\infty}\leq\|f(m_{0})\|_{\infty} T+\|g(m_{0})\|_{\infty};\,\,\,\, \qquad \|m\|_\infty\leq\|m_{0}\|_{\infty};\,\,\,\,\qquad  \text{\emph{supp}}(m)  \subset [-C,C] \times [0,T],
\]
\begin{equation}
\|Du\|_{L^{\infty}}\leq C,\label{eq:Du ineq}
\end{equation}
\begin{equation}
\int_{0}^{T}\int_{\mathbb{R}}|D(f(m))|^{2}\leq C,\label{eq:D(f(m)) energy bd}
\end{equation}
and, for $(x,t), (\overline{x},\overline{t})\in \R \times [0,T],$
\begin{equation}
|f(m(x,t))-f(m(\overline{x},\overline{t}))|  \leq\frac{C}{\sqrt{\log(|x-\overline{x}|^{2}+|t-\overline{t}|^{2})_{-}}}. \label{eq:f holder bd, supt bd}
\end{equation}
Finally, if $g'>0$ or $g\equiv0$, then $(u,f(m))\in C_{\emph{\text{loc}}}^{2,\alpha}((\mathbb{\mathbb{R}\times}[0,T])\cap\{m>0\})\times C_{\emph{\text{loc}}}^{1,\alpha}((\mathbb{\mathbb{R}\times}[0,T])\cap\{m>0\})$.
\item There exists a  solution $(u,m)$ to \eqref{eq:MFGPR}. The function
$m$ is unique, $u$ is unique up to a constant on each connected
component of $\{m>0\}$, $(u,f(m))\in C_{\emph{loc}}^{2,\alpha}((\mathbb{\mathbb{R}\times}[0,T])\cap\{m>0\})\times C_{\emph{loc}}^{1,\alpha}((\mathbb{\mathbb{R}\times}[0,T])\cap\{m>0\})$. Moreover,
there exist constants $K,C>0$, with
\[ 
K=K(T,T^{-1}, \kappa_{0},\|f(m_{0})\|_{\infty},\|f(m_{T})\|_{\infty},|\text{\emph{supp}}(m_{0})|,|\text{\emph{supp}}(m_{T})|,\text{\emph{dist}}(\text{\emph{supp}}(m_{0}),\text{\emph{supp}}(m_{T}))), \]
\[C=C(K,\|f(m_{0})\|_{W^{1,\infty}},\|f(m_{T})\|_{W^{1,\infty}}),\]
such that
\[\emph{\text{osc}}(u) \leq K;\,\,\,\, \qquad\|m\|_\infty \leq\max(\|m_{0}\|_{\infty},\|m_{T}\|_{\infty});\,\,\,\, \qquad  \text{\emph{supp}}(m)  \subset [-C,C]\times [0,T],
\]
and
\eqref{eq:Du ineq}, \eqref{eq:D(f(m)) energy bd}, \eqref{eq:f holder bd, supt bd} hold.
\end{enumerate}
\end{thm}

\begin{rem}We note that, due to lack of uniqueness for $u$ in $\{m=0\}$ for the case of \eqref{eq:MFGPR}, the explicit a priori
estimate (\ref{eq:Du ineq}) for the solutions is limited to the support of $m$. However,
the proof of Theorem \ref{thm:well-posedness theorem} will show that
there exists a solution to \eqref{eq:MFGPR} satisfying the global estimate
\[
\|u\|_{W^{1,\infty}(\mathbb{R})}\leq C,
\]
where $C$ depends on the data as described above.
\end{rem}

\subsubsection{From periodic to Neumann boundary conditions} \label{subsubsec:Neumann}
As anticipated in Section \ref{sec: periodic}, we intend to use the estimates for the periodic setting in order to build a solution to \eqref{eq:MFGR} and \eqref{eq:MFGPR}. However, it will be convenient to switch from periodic to Neumann boundary conditions. Indeed, we will see later that this simplifies the analysis of the optimal trajectories, since they do not ``wrap around'' the domain as they do in the periodic case. The following result shows that we may switch to this point of view while preserving all of the estimates for the periodic setting.
\begin{prop}
\label{prop:Neumann theorem} Assume that $f,g$ satisfy \eqref{fg-ass}, and let $0<\alpha<1$. Assume that the positive
functions $m_{0}^{\vep}$, $m_{T}^{\vep}\in C(\mathbb{R})$
are such that $\me_{0},\me_{T}\equiv\vep$ outside of the interval
$[-r,r]$, $\int_{-r}^{r}m_{0}^{\vep}=\int_{-r}^{r}m_{T}^{\vep}$,
and $f(m_{0}^{\vep}),f(m_{T}^{\vep})\in C^{1,\alpha}(\mathbb{R})$.
Assume that $g'>0$, and that (\ref{eq:polynomial growth}) holds,
and let $R>r$. Consider the system
\begin{equation}
\begin{cases}
-\ue_{t}+\frac{1}{2}(\ue_{x})^{2}=f(\me) & (x,t)\in[-R,R]\times(0,T)\\
\me_{t}-(\me\ue_{x})_{x}=0 & (x,t)\in[-R,R]\times(0,T)\\
\me(x,0)=\me_{0}(x), & x\in[-R,R]\\
u_{x}^{\vep}(-R,t)=u_{x}^{\vep}(R,t)=0 & t\in[0,T],
\end{cases}\label{eq:MFG-Neu}
\end{equation}
where either $u^{\vep}(\cdot,T)=g(\me(\cdot,T))$ or $\me(\cdot,T)=\me_{T}$. 
\end{prop}

\begin{itemize}
\item[(i)] There exists a unique solution $ (\ue,m^{\vep})\in C^{2,\alpha}([-R,R]\times[0,T])\times  C^{1,\alpha}([-R,R]\times[0,T])$
to (\ref{eq:MFG-Neu}) satisfying $u^{\vep}(\cdot,T)=g(\me(\cdot,T))$.
Moreover, there exists a constant $C=C(R,T,T^{-1}, \kappa_{0},\|f(m_{0}^{\vep})\|_{W^{1,\infty}},\|g(\me_{0})\|_{\infty})$,
such that
\begin{equation}
\|\ue\|_{\infty}\leq\|f(\me_{0})\|_{\infty}T+\|g(\me_{0})\|_{\infty};\qquad  \,\,\,\min\me_{0}\leq\me\leq\max\me_{0},\label{eq:m epsilon upper bound}
\end{equation}

\begin{equation}
\|Du^{\vep}\|_{\infty}\leq C;\qquad \,\,\,\,\int_{0}^{T}\int_{-R}^{R}|D(f(\me))|^{2}\leq C,\label{eq:-14}
\end{equation}
and, for each $(x,t),(\overline{x},\overline{t})\in[-R,R]\times[0,T]$,
\begin{equation}
|f(\me(x,t))-f(\me(\overline{x},\overline{t}))|\leq C \left(\frac{\int_{0}^{T}\int_{-R}^{R}|D(f(\me))|^{2}}{\log(|x-\overline{x}|^{2}+|t-\overline{t}|^{2})_{-}}\right)^{\frac12}.\label{eq:-15}
\end{equation}
\item[(ii)] Up to an additive constant for $u^\vep$, there exists a unique solution  $(\ue,m^{\vep})\in C^{2,\alpha}([-R,R]\times[0,T])\times C^{1,\alpha}([-R,R]\times[0,T])$ to \eqref{eq:MFG-Neu} 
satisfying $\me(\cdot,T)=\me_{T}$. Moreover, there exist constants $K,C>0$, with
\[ K=K(R,T,T^{-1}, \kappa_{0},\|f(m_{0}^{\vep})\|_{\infty},\|f(\me_{T})\|_{\infty})\,,\, C=C(K,\|f(m_0^{\vep})\|_{W^{1,\infty}},\|f(m_T^{\vep})\|_{W^{1,\infty}})\]
such that 
\begin{equation}
\text{osc(\ensuremath{\ue})}\leq K, \,\,\,\,\min(\min\me_{0},\min\me_{T})\leq\me\leq\max(\max\me_{0},\max\me_{T}),\label{eq:-5}
\end{equation}

and (\ref{eq:-14}), (\ref{eq:-15}) hold. 
\end{itemize}
\begin{proof}
For concreteness, we focus on the planning case, $\me(\cdot,T)=\me_{T}$.
For each $x\in[0,2R]$, define $\tilde{m}_{0}(x)=m_{0}^{\vep}(x-R)$,
and $\tilde{m}_{T}(x)=m_{T}^{\vep}(x-R)$. Then these functions
can be naturally extended to even, periodic functions $\tilde{m_{0}},\tilde{m_{T}}\in4\T$.
With these marginals, the solution $(\tilde{u},\tilde{m})$ to \eqref{eq:Per MFG}--\eqref{per-MFGP}, given by Theorem \ref{thm:existence},
is even, as well as symmetric with respect to $x=\pm2R$. In particular,
we have
\[
\tilde{u}_{x}(0,t)=\tilde{u}_{x}(\pm2R,t)\equiv0.
\]
As a result, the functions $(\ue,\me)$ defined by
\[
\ue(x,t)=\tilde{u}(x+R,t),\,\me(x,t)=\tilde{m}(x+R,t)
\]
are classical solutions to (\ref{eq:MFG-Neu}). The estimates on $(\ue,\me)$
readily follow by applying Corollary \ref{cor: m bound}
and Propositions \ref{prop: oscillation bound}, \ref{prop:gradient bound},
\ref{prop:energy estimate}, \ref{prop:interior modulus}, and \ref{prop:boundary modulus}
to the function $(\tilde{u},\tilde{m})$. A similar discussion yields
the result for the final cost problem, $\ue(\cdot,T)=g(\me(\cdot,T)).$
\end{proof}

\subsubsection{An estimate on the flow of optimal trajectories}

Given a solution $(\ue,\me)$ to (\ref{eq:MFG-Neu}), we may define
the  flow of optimal trajectories \[\gamma^{\vep}:[-R,R] \times [0,T] \rightarrow \mathbb{R}\]   according to \eqref{def.gammaBIS}. We remark that, when $x=\pm R$, the solution is the constant curve
$\gamma^{\vep}(x,\cdot)\equiv x$. Additionally, since $\ue\in C^{2,\alpha}([-R,R]\times[0,T])$, we have
$\gamma^{\vep}\in C^{2}([-R,R]\times[0,T])$, and  $\gamma_{x}^{\vep}>0$. 
We begin by showing that the trajectories starting in the support of $m_0$ remain in a bounded set, independently of $R$.
\begin{prop}[Finite propagation of the support]
\label{prop:bounded support}Under the assumptions of Proposition \ref{prop:Neumann theorem},
let $R\geq1$, $0<\vep<1$, and let $(\ue,\me)$ be a solution
to (\ref{eq:MFG-Neu}), and let $\gamma^{\vep}$ be the flow of
trajectories associated to $(\ue,\me)$. Then there exists a constant
$\overline{r}=\overline{r}(r,T,T^{-1}, \kappa_{0},\|f(m_{0}^{\vep})\|_{W^{1,\infty}},\|f(m_{T}^{\vep})\|_{W^{1,\infty}},\|g(\me_{0})\|_{\infty})$,
such that

\[
\|\gamma^{\vep}\|_{[-r,r]\times[0,T]}\leq\overline{r}.
\]
\end{prop}

\begin{proof}
For simplicity, we write $\gamma^{\vep}=\gamma$. We will treat
separately the planning problem and the final cost problem. First, assume that $\me(\cdot,T)=\me_T$. In view of Lemma
\ref{lem.gamma}, and the facts that $\gamma(-R,\cdot)\equiv-R$,
and $m_0\equiv 0$ outside of $[-r,r]$, we have
\[
\int_{-R}^{y}m_{T}^{\vep}=\int_{-R}^{-r}m_{0}^{\vep}=(R-r)\vep,
\]
where $y=\gamma(-r,T)$. Now, the left hand side is (strictly) increasing
in $y$, and 

\[
\int_{-R}^{-r}m_{T}^{\vep}=(R-r)\vep,
\]
thus $\gamma(-r,T)=y=-r$. Similarly $\gamma(r,T)=r$. Since $\gamma_{x}>0$,
this implies that $\gamma([-r,r],T)\subset[-r,r]$. Now, given $x\in[-r,r]$,
we recall that
\[
\gamma(x,\cdot)=\argmin\limits_{\alpha\in H^1(0,T),\alpha(0)=x}\int_{0}^{T}\frac{1}{2}|\Dot \alpha|^{2}+f(\me(\alpha(t),t))dt+\ue(\alpha(T),T).
\]
Therefore, defining $\alpha:[0,T]\rightarrow[-r,r]$ to be the straight
line segment connecting $x$ and $\gamma(x,T)$, we have
\begin{multline}
\int_{0}^{T}\left(\frac{1}{2}|\gamma_{t}(x,t)|^{2}+f(\me(\gamma(x,t),t))\right)dt+\ue(\gamma(x,T),T)\\  \leq \int_{0}^{T}\left(\frac{1}{2}|\Dot\alpha|^{2}+f(\me(\alpha(x,t),t))\right)dt+\ue(\gamma(x,T),T),\label{eq:-16}
\end{multline}
that is,
\[
\int_{0}^{T}\frac{1}{2}|\gamma_{t}(x,t)|^{2}dt\leq T\left(2\frac{r^{2}}{T^{2}}+f(\max m^{\vep})-f(\min m^{\vep})\right)\leq C.
\]
In particular, given $t\in[0,T]$, we have
\[
|\gamma(x,t)|\leq r+|\gamma(x,t)-\gamma(x,0)|\leq r+\sqrt{2t}\sqrt{\int_{0}^{T}\frac{1}{2}|\gamma_{t}|^{2}}\leq r+\sqrt{C}\leq C.
\]
This proves the result for the planning case. Next, we assume that
$\ue(\cdot,T)=g(\me(\cdot,T))$. Observe that, in this case, while
we do not know that $\gamma([-r,r],T)\subset[-r,r]$,  
we instead observe that $\ue$ is bounded independently of $R$, due to  Proposition \ref{prop:Neumann theorem}.
Therefore,
we have
\[
\int_{0}^{T}\frac{1}{2}|\gamma_{t}(x,t)|^{2}dt=\ue(x,0)-\ue(\gamma(x,T),T)-\int_{0}^{T}f(\me(\gamma(x,t),t)dt\leq C.
\]
\end{proof}
\subsubsection{Compatible approximations and existence of  solutions}
In this section, we will apply Proposition \ref{prop:Neumann theorem} to prove the well-posedness result, Theorem \ref{thm:well-posedness theorem}. For that purpose, we will now build suitable $C^{1,\alpha}(\mathbb{R})$ approximations $\me_0>0$ through a mild regularization procedure. For \eqref{eq:MFGPR}, we must also build $\me_T$, while requiring that the two are suitably compatible, in the sense that they have the same mass. Furthermore, since these approximations will be needed later to prove finer results about the free boundary, we will ensure that they are also compatible in the sense that $[a_0,b_0]$ is mapped bijectively onto $[a_1,b_1]$ by the flow $\gamma^{\vep}(\cdot,T)$.
\begin{defn}
Given $r>0$, we say that $(m_{0}^{\vep},m_{T}^{\vep})$ is
a $r$--compatible approximation of $(m_{0},m_{T})$ if $(m_{0}^{\vep},m_{T}^{\vep})$
satisfies the assumptions of Proposition \ref{prop:Neumann theorem}, $m_{0}^{\vep}\to m_0$, $m_{T}^{\vep}\to m_T$ uniformly,  and,
for every $R>r$,

\begin{equation}
\int_{-R}^{a_{0}}m_{0}^{\vep}=\int_{-R}^{a_{1}}m_{T}^{\vep}\,\text{ and }\,\int_{b_{0}}^{R}m_{0}^{\vep}=\int_{b_{0}}^{R}m_{T}^{\vep}.\label{eq:compatibility condition}
\end{equation}
\end{defn}

\begin{lem}[Construction of $r$--compatible approximations]
\label{lem:lemma compat approx}  Under the assumption of Proposition \ref{prop:Neumann theorem}, let $r>0$ be such that
\begin{equation}
[a_{0}-2,b_{0}+2],[a_{1}-2,b_1+2]\subset[-r,r].\label{eq:r condition}
\end{equation}
There exists a pair of vector-valued functions $\eta_{0},\eta_{T}\in C_{c}^{\infty}(\mathbb{R},[0,\infty)^{3})$
such that, for every $\vep\in(0,1)$, there exist constant vectors
$c_{0,\vep},c_{T,\vep}\in[0,\infty)^{3}$ such that the pair
$(m_{0}^{\vep},m_{T}^{\vep})$ defined by

\begin{equation}
m_{0}^{\vep}=f^{-1}(\sqrt{f(m_{0})^{2}+f(\vep)^{2}}+c_{0,\vep}\cdot\eta_{0}),\,\,m_{T}^{\vep}=f^{-1}(\sqrt{f(m_{T})^{2}+f(\vep)^{2}}+c_{T,\vep}\cdot\eta_{T})\label{eq:compatible approximations}
\end{equation}
is a $r$-compatible approximation of $(m_{0},m_{T})$. Moreover,
we have
\begin{equation}
\lim_{\vep\rightarrow0^{+}}c_{0,\vep}=\lim_{\vep\rightarrow0^{+}}c_{T,\vep}=0,\label{eq:-8}
\end{equation}
\begin{multline}
\eta_{0}\equiv0\:\text{\,in\, }\left[a_{0}-1,\frac{1}{3}(2a_{0}+b_{0})\right]\cup\left[\frac{1}{3}(a_{0}+2b_{0}),b_{0}+1\right], \\
\eta_{T}\equiv0\,\,\text{in }\left[a_1-1,\frac{1}{3}(2a_1+b_1)\right]\cup\left[\frac{1}{3}(a_1+2b_1),b_1+1\right],\label{eq:-9}
\end{multline}
and there exists a constant $C=C(|\emph{\text{supp}}(m_{0})|,|\emph{\text{supp}}(m_{T})|,\emph{\text{dist}}(\emph{\text{supp}}(m_{0}),\emph{\text{supp}}(m_{T})))$
such that
\begin{equation}
\|f(m_{j}^{\vep})\|_{W^{1,\infty}(\mathbb{R})}\leq C\|f(m_{j})\|_{W^{1,\infty}(\mathbb{R})}+f(\vep)+|c_{0,\vep}|\cdot\|\eta_{j}\|_{C^{1}(\mathbb{R},\mathbb{R}^{3})},\,\,j\in\{0,T\},\label{eq:uniform lip}
\end{equation}
\begin{equation}
[f(m_{j}^{\vep})_{x}]_{C^{\alpha}(\mathbb{R})}\leq\frac{C}{f(\vep)}(1+[f(m_{j})_{x}]_{C^{\alpha}(\{m_{j}>0\})})+|c_{0,\vep}|[(\eta_{j})_{x}]_{C^{\alpha}(\mathbb{R},\mathbb{R}^{3})},\,j\in\{0,T\}\label{eq:uniform holder}
\end{equation}
\begin{equation}
[f(m_{j}^{\vep})_{x}]_{C^{\alpha}(\{m_{j}\geq\delta\})}\leq\frac{C}{f(\delta)}(1+[f(m_{j})_{x}]_{C^{\alpha}(\{m_{j}>0\})})+\frac{C}{f(\delta)}|c_{0,\vep}|[(\eta_{j})_{x}]_{C^{\alpha}(\mathbb{R},\mathbb{R}^{3})},\,j\in\{0,T\}\label{eq:interior holder}
\end{equation}
Finally if $j\in\{0,T\}$ and $f(m_{j})$ is semi--convex, then

\begin{equation}
\|f(m_{j}^{\vep})_{xx}^{-}\|_{L^{\infty}(\mathbb{R})}\leq\|f(m_{j})_{xx}^{-}\|_{L^{\infty}(\mathbb{R})}+|c_{0,\vep}|\cdot\|(\eta_{j})_{xx}^{-}\|_{L^{\infty}(\mathbb{R})}.\label{eq:uniform semiconvexity}
\end{equation}
\end{lem}

\begin{proof}
We let $\eta_{0,1},\eta_{0,2},\eta_{0,3}\in C_{c}^{\infty}(\mathbb{R})$
be non--zero, non--negative bump functions supported, respectively,
on 
\[
[a_{0}-2,a_{0}-1],\left[\frac{1}{3}(2a_{0}+b_{0}),\frac{1}{3}(a_{0}+2b_{0})\right],\text{ and }[b_{0}+1,b_{0}+2].
\]
Similarly, we take $\eta_{T,1},\eta_{T,2},\eta_{T,3}\in C_{c}^{\infty}(\mathbb{R})$
to be supported, respectively, on
\[
[a_1-2,a_1-1],\left[\frac{1}{3}(2a_1+b_1),\frac{1}{3}(a_1+2b_1)\right],\text{ and }[b_1+1,b_1+2].
\]
With this, we can define
\[
\eta_{0}=(\eta_{0,1},\eta_{0,2},\eta_{0,3}),\,\,\eta_{T}=(\eta_{T,1},\eta_{T,2},\eta_{T,3}).
\]
We observe first that (\ref{eq:-9}) holds by construction. Next, note that (\ref{eq:r condition}) guarantees that $m_{0}^{\vep},m_{T}^{\vep}\equiv\vep$
outside of $[-r,r]$. For this reason, (\ref{eq:compatibility condition})
will hold for all $R>r$ as long as it holds for $R=r$. Now, we have
\[
\int_{-r}^{a_{0}}m_{0}^{\vep}=\int_{-r}^{a_{0}}f^{-1}(f(\vep)+c_{0,\vep}^{1}\eta_{0,1}),\text{ and }\int_{-r}^{a_1}m_{T}^{\vep}=\int_{-r}^{a_1}f^{-1}(f(\vep)+c_{T,\vep}^{1}\eta_{T,1}).
\]
If, say, $a_0 \leq a_1$,
then, taking $c_{T,\vep}^{1}=0$, there exists a unique choice
of $c_{0,\vep}^{1}\geq0$ that ensures
\[
\int_{-r}^{a_{0}}m_{0}^{\vep}=\int_{-r}^{a_1}m_{T}^{\vep}.
\]
Similarly, if $a_1< a_{0}$, we may take $c_{0,\vep}^{1}=0$
and choose an adequate $c_{T,\vep}^{1}>0$. By the same reasoning,
there exists a choice of $c_{0,\vep}^{3},c_{T,\vep}^{3}$ such
that
\[
\int_{b_{0}}^{r}m_{0}^{\vep}=\int_{b_{0}}^{r}m_{T}^{\vep}.
\]
We must also ensure that $m_{0}^{\vep},m_{T}^{\vep}$ have
equal mass. For this purpose, we observe that
\[
\int_{a_{0}}^{b_{0}}m_{0}^{\vep}=\int_{a_{0}}^{b_{0}}f^{-1}(\sqrt{f(m_{0})^{2}+f(\vep)^{2}}+c_{0,\vep}^{2}\eta_{0,2}),\,\int_{a_1}^{b_1}m_{T}^{\vep}=\int_{a_1}^{b_1}f^{-1}(\sqrt{f(m_{T})^{2}+f(\vep)^{2}}+c_{T,\vep}^{2}\eta_{T,2}).
\]
And, as before, depending on which of the quantities $\int_{a_{0}}^{b_{0}}f^{-1}(\sqrt{f(m_{0})^{2}+f(\vep)^{2}})$,
$\int_{a_1}^{b_1}f^{-1}(\sqrt{f(m_{T})^{2}+f(\vep)^{2}})$
is larger, we may choose one of the numbers $c_{0,\vep}^{2},c_{T,\vep}^{2}$
to be zero, which leaves a unique way to choose the remaining one in such
a way that
\[
\int_{a_{0}}^{b_{0}}m_{0}^{\vep}=\int_{a_1}^{b_1}m_{T}^{\vep}.
\]
This, together with (\ref{eq:compatibility condition}), guarantees
that the two functions have the same integral over $[-R,R]$. Now,
(\ref{eq:-8}) is a straightforward consequence of (\ref{eq:equal mass assumption}),
(\ref{eq:compact support assumption}), and the fact that the vectors
$c_{0,\vep},c_{T,\vep}$ were chosen so that
\[
c_{0,\vep}^{i}c_{T,\vep}^{i}=0\text{ for }i\in\{1,2,3\}.
\]
It remains to show the Lipschitz, $C^{1,\alpha}$, and semi-convexity
estimates on $f(m_{0}^{\vep}),f(m_{T}^{\vep})$. For concreteness, we will only show these for $f(m_{0}^{\vep})$, since the arguments for $f(m_{T}^{\vep})$ are identical. We have
\[
f(m_{0}^{\vep})=\sqrt{f(m_{0})^{2}+f(\vep)^{2}}+c_{0,\vep}\cdot\eta_{0}\,; \qquad \,\,\,f(m_{0}^{\vep})_{x}=\frac{f(m_{0})f(m_{0})_{x}}{\sqrt{f(m_{0})^{2}+f(\vep)^{2}}}+c_{0,\vep}\cdot(\eta_{0})_{x},
\]
which readily yields (\ref{eq:uniform lip}). Now, given $x,y\in[-r,r]$, we
may write
\begin{multline}
f(m_{0}^{\vep})_{x}(x)-f(m_{0}^{\vep})_{x}(y)=c_{0,\vep}\cdot((\eta_{0})_{x}(x)-(\eta_{0})_{x}(y))+\frac{(f(m_{0})(x)-f(m_{0})(y))f(m_{0})_{x}(x)}{\sqrt{f(m_{0})^{2}(x)+f(\vep)^{2}}}\\
+f(m_{0})(y)\frac{f(m_{0})_{x}(x)-f(m_{0})_{x}(y)}{\sqrt{f(m_{0})(x)^{2}+f(\vep)^{2}}}+f(m_{0})(y)f(m_{0})_{x}(y)\left(\frac{1}{\sqrt{f(m_{0})(x)^{2}+f(\vep)^{2}}}-\frac{1}{\sqrt{f(m_{0})(y)^{2}+f(\vep)^{2}}}\right)\\
=c_{0,\vep}\cdot((\eta_{0})_{x}(x)-(\eta_{0})_{x}(y))+\frac{(f(m_{0})(x)-f(m_{0})(y))f(m_{0})_{x}(x)}{\sqrt{f(m_{0})^{2}(x)+f(\vep)^{2}}}+f(m_{0})(y)\frac{f(m_{0})_{x}(x)-f(m_{0})_{x}(y)}{\sqrt{f(m_{0})(x)^{2}+f(\vep)^{2}}}\\
+f(m_{0})(y)f(m_{0})_{x}(y)\left(\frac{(f(m_{0})(y)-f(m_{0})(x))(f(m_{0})(y)+f(m_{0})(x))}{\sqrt{f(m_{0})(x)^{2}+f(\vep)^{2}}\sqrt{f(m_{0})(y)^{2}+f(\vep)^{2}}(\sqrt{f(m_{0})(y)^{2}+f(\vep)^{2}}+\sqrt{f(m_{0})(x)^{2}+f(\vep)^{2}})}\right).\label{eq:approx holder computation}
\end{multline}
Thus, we obtain
\begin{multline*}
|f(m_{0}^{\vep})_{x}(x)-f(m_{0}^{\vep})_{x}(y)|\leq|c_{0,\vep}|[\eta_{0}]_{\alpha}|x-y|^{\alpha}+\frac{C}{f(\vep)}\|f(m_{0})_{x}\|_{\infty}^{2}|x-y|^{\alpha}+\frac{1}{f(\vep)}\|f(m_{0})\|_{\infty}[f(m_{0})_{x}]_{C^{\alpha}(\{m_{0}>0\})}|x-y|^{\alpha}\\
+\frac{C}{f(\vep)}\|f(m_{0})_{x}\|_{\infty}^{2}|x-y|^{\alpha},
\end{multline*}
which yields (\ref{eq:uniform holder}). Similarly, if $x,y\in\{m_{0}\geq\delta\}$, 
we obtain  (\ref{eq:interior holder}). Finally, if
$m_{0}$ is smooth, then
\begin{multline*}
f(m_{0}^{\vep})_{xx}=\frac{f(m_{0})f(m_{0})_{xx}}{\sqrt{f(m_{0})^{2}+f(\vep)^{2}}}+\frac{f(m_{0})_{x}^{2}}{\sqrt{f(m_{0})^{2}+f(\vep)^{2}}}-\frac{f(m_{0})^{2}f(m_{0})_{x}^{2}}{(f(m_{0})^{2}+f(\vep)^{2})^{\frac{3}{2}}}+c_{0,\vep}\cdot(\eta_{0})_{xx}\\
\geq-\|f(m_{0})_{xx}^{-}\|_{\infty}+\frac{f(m_{0})_{x}^{2}f(\vep)^{2}}{(f(m_{0})^{2}+f(\vep)^{2})^{\frac{3}{2}}}-|c_{0,\vep}\||(\eta_{0})_{xx}\|_{\infty}\\
\geq-\|f(m_{0})_{xx}^{-}\|_{\infty}-|c_{0,\vep}\||(\eta_{0})_{xx}\|_{\infty},
\end{multline*}
which shows (\ref{eq:uniform semiconvexity}). For a non-smooth $m_{0}$,
the result then follows by standard approximation.
\end{proof}
Having constructed the necessary approximations, we are now ready to prove the well-posedness theorem.
\begin{proof}[Proof of Theorem \ref{thm:well-posedness theorem}]
We begin with the case of (\ref{eq:MFGPR}). We first prove existence
of a solution satisfying the estimates. For this purpose, we choose the $r$-compatible approximations
$(m_{0}^{\vep},m_{T}^{\vep})$ as defined by (\ref{eq:compatible approximations}), where $r>0$ is fixed and chosen large enough to satisfy (\ref{eq:r condition}), and take
$R>2\overline{r}$, where $\overline{r}$ is the constant of Proposition
\ref{prop:bounded support}. In view of (\ref{eq:uniform lip}) from
Lemma \ref{lem:lemma compat approx}, the constant $C$ of Proposition
\ref{prop:Neumann theorem} may be chosen independently of $\vep$.
In particular, from (\ref{eq:-5}), we may choose the solutions $(u^{\vep},m^{\vep})$
of Proposition \ref{prop:Neumann theorem} in such a way that
$\|\ue\|_{\infty}\leq C.$
Therefore, due to the bounds \rife{eq:-14} and \rife{eq:-15}, 
 the family $\{(u^{\vep},m^{\vep})\}_{\vep\in(0,1)}$ is bounded and equicontinuous, and we may extract a subsequence and
conclude that $(u^{\vep},m^{\vep})\rightarrow(u,m)\in W^{1,\infty}([-R,R]\times[0,T])\times C([-R,R]\times[0,T])$,
with $(u,m)$ satisfying the required estimates.

Moreover, if $(x,t)\in\{m>0\}$,
by the equicontinuity of $\{m^{\vep}\}_{\vep\in(0,1)}$, there
exists an open set $V$ containing $(x,t)$ such that $f'(m^\vep)m^\vep>0$, so that  $u^{\vep}$
satisfies a uniformly elliptic equation in $V$ (recall \eqref{eq.elluBIS}). As a result, if $0<t<T$,
the standard interior $C^{1,\alpha}$ estimates followed by the Schauder
estimates (see \cite[Thm. 13.6]{GilbargTrudinger} and \cite[Thm. 6.2]{GilbargTrudinger}, respectively)  imply that $(u,m)\in C^{2,\alpha}\times C^{1,\alpha}$ in a neighborhood of $(x,t)$.
On the other hand, if $t\in\{0,T\}$, by proceeding as in the proof of Theorem \ref{thm:existence}, through the $C^{1,\alpha}$ estimates
for oblique problems, followed by Fiorenza's convergence argument (as detailed in the proof of Theorem \ref{thm:existence}), we see that $(u,m)\in C^{2,\alpha} \times C^{1,\alpha}$ in a neighborhood of
$(x,t)$. In particular, this shows that $u$ solves the HJ equation
in the pointwise sense at $(x,t)$.

Now,  we claim that $m\in C_{c}(\mathbb{R}\times[0,T])$, with $\text{supp}(m)\subset [-\overline{r},\overline{r}] \times [0,T]$.
 Indeed, this follows from the fact that, for each $t\in [0,T],$ $m(\cdot,t)$ has mass $1$ and, by Lemma \ref{lem.gamma} and Proposition \ref{prop:bounded support},
\[\int_{-\overline{r}}^{\overline{r}} m(\cdot,t)= \lim_{\epsilon \downarrow 0}\int_{-\overline{r}}^{\overline{r}} \me(\cdot,t)\geq \lim_{\epsilon \downarrow 0} \int_{a_0}^{b_0} \me_0 = \int_{a_0}^{b_0} m_0=1.\]
On the other hand, we can extend $u$ to $W^{1,\infty}(\mathbb{R}\times(0,T))$
in the following way:
\[
u(x,t)=\begin{cases}
u(-R,t) & \text{if }x<-R\\
u(R,t) & \text{if }x>R.
\end{cases}
\]
Notice that, since $\ue_{x}(\pm R,\cdot)\equiv0$,
\[
-u_{t}^{\vep}(\pm R,t)=f(\me(\pm R,t))\rightarrow f(0)\text{ uniformly in \ensuremath{t}},
\]
and, thus,
\[
u^{\vep}(\pm R,t)\rightarrow f(0)(T-t)+u(\pm R,T)\text{ in }C^{1,1}([0,T]).
\]
In particular, we see that, for $|x|>R$,
\[-u_{t}+\frac{1}{2}u_{x}^{2}=f(0)=f(m).
\]
It is then straightforward to verify, by using the basic stability property of viscosity solutions under uniform convergence (see \cite{UsersGuide}), that this extension $(u,m)$
is a solution to (\ref{eq:MFGPR}), in the sense of Definition \ref{gensol},  which satisfies all the necessary
estimates. 

Now, to prove uniqueness, we assume that $(\tilde{u},\tilde{m})$
is another solution. 
Since the function $\tilde{u}$ is a viscosity solution to the HJ equation, and it is almost everywhere differentiable, it must also satisfy the equation pointwise almost everywhere. Therefore, it also solves the equation in the distributional sense.  Noting that the pairs $(u,m), (\tilde{u},\tilde{m})\in W^{1,\infty}(\R \times (0,T)) \times C_b(\R \times [0,T])$ are sufficiently regular to serve as test functions, we may apply the standard Lasry-Lions computation to the pair  $(u,m)$ and $(\tilde{u},\tilde{m})$,
obtaining
 \begin{equation*}
\int_{0}^{T}\int_{\R}\frac{1}{2}(m+\tilde{m})|u_{x}-\tilde{u}_{x}|^{2}+(m-\tilde{m})(f(m)-f(\tilde{m}))=0.
\end{equation*}Therefore, since the left hand side is
non-negative, we conclude that $m=\tilde{m}$, and $u_{x}=\tilde{u}_{x}$
in $\{m>0\}$. Finally, since $m$ is continuous, $\{m>0\}$ is an
open set, and thus $u$ and $\tilde{u}$ at most differ by a constant
on each connected component of $\{m>0\}$. For the case of (\ref{eq:MFGR}),
the proof is completely analogous, noting that in applying Proposition
\ref{prop:Neumann theorem}, we take $g^{\vep}(m)=g(m)+\vep m$,
to satisfy the strict monotonicity assumption. We note that, unless
$g'>0$, the Lieberman and oblique Schauder estimates may not be applied
at $t=T$, hence the weaker conclusion $(u,f(m))\in C_{\text{loc}}^{2,\alpha}((\mathbb{\mathbb{R}\times}[0,T))\cap\{m>0\})\times C_{\text{loc}}^{1,\alpha}((\mathbb{\mathbb{R}\times}[0,T))\cap\{m>0\})$.
However, in the special case $g\equiv0$, $u$ satisfies a smooth
Dirichlet boundary condition at $t=T$, and thus we may apply the Ladyzhenskaya-Uralt'seva
and Schauder estimates for the Dirichlet problem (see \cite[Thm 13.7, Thm 6.19]{GilbargTrudinger}) to the limiting function
$u$ to still obtain the $C_{\text{loc}}^{2,\alpha}((\mathbb{\mathbb{R}\times}[0,T])\cap\{m>0\})$
regularity. Note that global uniqueness 
follows because, since $m$ is unique, the terminal value $u(\cdot,T)=g(m(\cdot,T))$
is uniquely determined.
\end{proof}

\subsection{Analysis of the free boundary}\label{free-analysis}

Having established the existence and uniqueness of solutions, with compactly supported density, we now study the set $\partial(\{m>0\})$,
where $(u,m)$ is the solution to (\ref{eq:MFGR}) or (\ref{eq:MFGPR}).  Henceforth, we restrict our analysis to the case that, for some given constant $\theta>0$, we have
\be\label{eq: assumption f=00003Dpower}
f(m)= m^\theta\,,\qquad \theta>0\,.
\ee
We will focus on the setting in which the initial distribution $m_0$ (and possibly $m_T$, in case of problem \rife{eq:MFGPR}) are each supported
exactly on an interval. That is, we will assume that
\begin{equation}
\{m_{0}>0\}=(a_{0},b_{0})\text{ and }\{m_{T}>0\}=(a_1,b_1),\label{eq:support assumption}
\end{equation}
for some finite $a_0, b_0, a_1, b_1$. 
Moreover, we will assume that $m_{0}$ and $m_{T}$ decay like powers
near the endpoints of these intervals. In other words, there exist
numbers $\alpha_0,\alpha_1>0$ such that
\begin{equation}
\frac{1}{C_{0}}\text{dist}(x,\{a_{0},b_{0}\})^{\alpha_0}\leq m_{0}(x)\leq C_{0}\text{dist}(x,\{a_{0},b_{0}\})^{\alpha_0}, \qquad x\in [a_0,b_0],\label{eq:decay restatement m0}
\end{equation}
\begin{equation}
\frac{1}{C_{1}}\text{dist}(x,\{a_1,b_1\})^{\alpha_1}\leq m_{T}(x)\leq C_{1}\text{dist}(x,\{a_1,b_1\})^{\alpha_1}, \qquad x\in [a_1,b_1].\label{eq:decay restatement mT}
\end{equation}
For concreteness, we will also assume here that 
\begin{equation}
\alpha_0\geq\alpha_1,\label{eq:alpha beta assumption}
\end{equation}
but we remark that the opposite case can be studied, with analogous
results, by simply considering the time-reversed functions $(-u(x,T-t),m(x,T-t))$,
which solve (\ref{eq:MFGPR}), but with the roles of $m_{0}$ and
$m_{T}$ reversed. 

Correspondingly, in case of problem \rife{eq:MFGR}, we assume that the  coupling function $g$ at final time satisfies, for some given constants
$ \theta_1>0$, $c_1\geq0$, with $\theta_1\geq\theta$,
\begin{equation}
 g(m)=c_1Tm^{\theta_1}.\label{eq: assumption g=00003Dpower}
\end{equation}
\begin{rem}The factor of $T$ in the definition of $g$ is made explicit in order to adequately state the sharp long time behavior result, Theorem \ref{thm:long time}. This is the natural scaling for the final pay-off which is consistent with the behavior of the self-similar solution, see Section 2.  However we also note that our assumptions include, in particular, the case of terminal condition $u(\cdot,T)=0$.
\end{rem} 

\subsubsection{Characterization of the free boundary through the equation satisfied by the flow}

In this subsection, we properly establish the existence of the free boundary curves together with their basic characterization.  Our main tool will be the elliptic equation satisfied by the flow $\gamma$ of optimal characteristics (see \eqref{eq.ell.gammaBIS}).
 We recall from Section \ref{sec:self similar} that $\oal\in (0,1)$ is defined by
\[\oal=\frac{2}{2+\theta}.\]

\begin{thm}[Characterization of the free boundary and the flow]
 \label{thm:characteriz free boundary} Let $f$ be given by \rife{eq: assumption f=00003Dpower}.  Assume that  \rife{eq:equal mass assumption}--\rife{eq:f(marginals) interior holder} and \eqref{eq:support assumption}--\eqref{eq: assumption g=00003Dpower} hold, and that $f(m_{0})$
is semi--convex. Let $(u,m)$ be a solution to \eqref{eq:MFGR}
or \eqref{eq:MFGPR}. Then there exist two functions $\gamma_{L}<\gamma_{R}\in W^{1,\infty}(0,T)$,
such that 
\begin{equation}
\{m>0\}=\{(x,t)\in\mathbb{R}\times[0,T]:\gamma_{L}(t)<x<\gamma_{R}(t)\}.\label{eq:free bdary charact}
\end{equation}
Moreover, the flow $\gamma$ of optimal trajectories is well defined
on $(a_{0},b_{0})\times[0,T]$, we have
\[
\gamma\in W^{1,\infty}((a_{0},b_{0})\times(0,T))\cap C_{\emph{\text{loc}}}^{2,\alpha}((a_{0},b_{0})\times[0,T]),\,\,\,\gamma_{x}>0,\,\,\,\gamma_{L}(t)=\gamma(a_{0},t),\,\gamma_{R}(t)=\gamma(b_{0},t),
\]
 and $\gamma$ is a classical solution in $(a_{0},b_{0})\times(0,T)$
to the elliptic equation

\begin{equation}
\gamma_{tt}+\frac{\theta f(m_{0})}{(\gamma_{x})^{2+\theta}}\gamma_{xx}=\frac{f(m_{0})_{x}}{(\gamma_{x})^{1+\theta}}.\label{eq:flow equation}
\end{equation}
Moreover, there exists a constant
\[
C= \begin{cases} C\left(T,T^{-1},C_{0},\|f(m_{0})\|_{W^{1,\infty}},\|f(m_{0})_{xx}^{-}\|_{\infty},\theta, \frac{\theta_1}{\theta},c_1\right) & \text{if } (u,m) \text{ solves \eqref{eq:MFGR}},\\
C\left(T,T^{-1},C_{0},C_1,\|f(m_{0})\|_{W^{1,\infty}},\|f(m_{T})\|_{W^{1,\infty}},\|f(m_{0})_{xx}^{-}\|_{\infty}, \theta\right) & \text{if } (u,m) \text{ solves \eqref{eq:MFGPR}},
\end{cases} \]
such that 

\begin{equation}
\|\gamma\|_{W^{1,\infty}([a_{0},b_{0}]\times[0,T])}\leq C,\label{eq:-11}
\end{equation}
and, for each $(x,t)\in(a_{0},b_{0})\times[0,T]$,

\begin{equation}
\frac{1}{C}m_{0}(x)\leq m(\gamma(x,t),t).\label{eq:-12}
\end{equation}
\end{thm}

\begin{proof}
Throughout this proof, as usual, the constant $C>0$ may increase
at each step. We will first treat the case in which $(u,m)$ solves
(\ref{eq:MFGPR}). As in the proof of Theorem \ref{thm:well-posedness theorem},
we let $(\ue,\me)$ be the solution to (\ref{eq:MFG-Neu}) given by
Proposition \ref{prop:Neumann theorem}, corresponding to the choices of
$m_{0}^{\vep},m_{T}^{\vep}$ given by (\ref{eq:compatible approximations}).
We may normalize the solution to satisfy $\int_{\T}\ue(\cdot,T)=0$.
We also fix $R$ large enough that $(-R,R)$ contains both $[a_{0}-2,b_{0}+2]$
and $[a_1-2,b_1+2]$. With this choice, we see from (\ref{eq:compatible approximations})
and (\ref{eq:-9}) that
\begin{equation}
\int_{-R}^{a_{0}}m_{0}^{\vep}=\int_{-R}^{a_1}m_{T}^{\vep},\text{ and }\int_{b_{0}}^{R}m_{0}^{\vep}=\int_{b_1}^{R}m_{T}^{\vep},\label{eq:compatible epsilon approx}
\end{equation}
which guarantees that $\gamma^{\vep}(a_{0},T)=a_1$, and $\gamma^{\vep}(b_{0},T)=b_1$,
or, equivalently, that $\gamma^{\vep}(\cdot,T)$ is a bijection
between $\text{supp }m_{0}$ and $\text{supp }m_{T}$. We are now
interested in a Lipschitz bound for $\gamma^{\vep}$. Recall  that
$
\gamma_{t}^{\vep}=-\ue_{x}(\gamma^{\vep}(x,t),t)$, 
so $\|\gamma_{t}^{\vep}\|_{L^{\infty}((a_{0},b_{0})\times[0,T])}$
is bounded due to Proposition \ref{prop:Neumann theorem} (we emphasize
that our choice of $R$ is already fixed).  Moreover, we have $\gamma^{\vep}\in C^{2}([-R,R]\times[0,T])$ and, by Lemma \ref{lem.gamma} 
we know that it satisfies   \rife{eq.ell.gammaBIS}; for the case $f(m)=m^\theta$, this equation may be rewritten as
%
%
%
%
\begin{equation}
\gamma_{tt}^{\vep}+\frac{\theta f(\me_{0})}{(\gamma_{x}^{\vep})^{2+\theta}}\gamma_{xx}^{\vep}=\frac{f(\me_{0})_{x}}{(\gamma_{x}^{\vep})^{1+\theta}}.\label{eq:flow equation epsilon}
\end{equation}
For $K>0$, we set 
\[v=\gamma_{x}^{\vep}, \quad \,w=v-Kt^{\oal}, \,\,\text{ and }\, \, k_1=\|f(m_{0})_{xx}^{-}\|_{\infty}.\]
We now want to show that, if $K$ is chosen sufficiently large,
\begin{equation}
\max_{[a_{0},b_{0}]\times[0,T]}w\leq \max_{\partial([a_{0},b_{0}]\times[0,T])}w.\label{eq:max principle on gamma_x}
\end{equation}
For this purpose we assume first that $\me_{0},\me_{T}$ are smooth.
Then $\gamma^{\vep}$ is smooth as well, and differentiating (\ref{eq:flow equation epsilon})
with respect to $x$, we get, for some function $b(x,t)$, and sufficiently small $\vep>0$,
\begin{equation}
v_{tt}+\frac{\theta f(\me_{0})}{v^{2+\theta}}v_{xx}+b(x,t)v_{x}=\frac{f(\me_{0})_{xx}}{v^{1+\theta}}\geq \frac{-2k_1}{v^{1+\theta}},\label{eq:-17}
\end{equation}
 where equation \eqref{eq:uniform semiconvexity} from Lemma \ref{lem:lemma compat approx} was used in the last inequality. Recalling the definition of $w$, this yields
\begin{equation}\label{eq:v subsol-1}
w_{tt}+\frac{\theta f(\me_{0})}{v^{2+\theta}}w_{xx}+b(x,t)w_{x}\geq\frac{-2k_1}{v^{1+\theta}}+  K \oal(1-\oal)t^{\oal-2}.
\end{equation}
Let $(x_{0},t_{0})$ be an interior maximum point of $w$. Then the right hand side of \eqref{eq:v subsol-1} must be non-positive, which may be rewritten as
\be\label{absK}
v(x_0,t_0)\leq \left(\frac{2k_1}{ K \oal(1-\oal)}\right)^{\frac{1}{1+\theta}}t_0^{\frac{2-\oal}{1+\theta}}=\left(\frac{2k_1}{ K \oal(1-\oal)}\right)^{\frac{1}{1+\theta}}t_0^{\oal}.
\ee
Since $\max_{[a_{0},b_{0}]\times[0,T]}w \geq 0$, we have $v(x_0, t_0)\geq Kt_0^{\oal}$. Hence inequality \rife{absK} is impossible if we choose  $K$ sufficiently large, namely if  $K>\left(\frac{2k_1}{\oal(1-\oal)}\right)^{\frac{1+\theta}{2+\theta}}$. This shows (\ref{eq:max principle on gamma_x}).
To remove the assumption that $\me_{0},\me_{T}$ are sufficiently
smooth to perform the above computations, we approximate them first
with convolutions, as was done in (\ref{eq:convolution approx}) (but
without any need to regularize $f$), and we simply note that (\ref{eq:max principle on gamma_x})
is stable under such an approximation.

In view of (\ref{eq:max principle on gamma_x}), it now suffices to
bound $w$ at the boundary points. When $t=0$, we have $w=v\equiv1$,
and when $x=a_{0}$ or $x=b_{0}$, $m_{0}^{\vep}(x)=\vep=\min\me$,
so that, recalling \eqref{masspreservationBIS},
\[
w(x,t)\leq v(x,t)=\frac{\me_{0}(x)}{\me(\gamma^{\vep}(x,t),t)}\leq1.
\]
Therefore, we are only left with estimating $v(x,T)$, for $x\in(a_{0},b_{0})$.
 We will assume that $x\in(a_{0},\frac{1}{2}(a_{0}+b_{0}))$, since
the converse case is completely symmetric. We first observe that the explicit form of the approximations \rife{eq:compatible approximations}, where $f(s)=s^\theta$, together with  
(\ref{eq:decay restatement m0}) and (\ref{eq:decay restatement mT}), imply, for some constants $c_0, c_1$: 
\begin{equation}
\frac{1}{c_{0}}( \text{dist}(x,\{a_{0},b_{0}\})^{\alpha_0}+ \vep )\leq \me_{0}(x)\leq c_{0}(\text{dist}(x,\{a_{0},b_{0}\})^{\alpha_0}+ \vep), \qquad x\in (a_0,b_0),\label{powersvep0}
\end{equation}
and
\begin{equation}
\frac{1}{c_{1}}(\text{dist}(x,\{a_1,b_1\})^{\alpha_1}+\vep)\leq \me_{T}(x)\leq c_{1}(\text{dist}(x,\{a_1,b_1\})^{\alpha_1}+\vep), \qquad x\in (a_1,b_1).
\label{powersvepT}
\end{equation}
Let us set $\overline{x}=\gamma^{\vep}(x,T)$; since we have
\begin{equation}
\int_{a_{0}}^{x}\me_{0}=\int_{a_1}^{\ox}m_{T}^{\vep}\label{eq:equal int}
\end{equation}
we deduce from \rife{powersvep0}--\rife{powersvepT} that $x\to a_0$ if and only if $\ox\to a_1$. Hence we can assume that $\overline{x}\leq\frac{a_1+b_1}{2}$ and we estimate 
\[
\frac{1}{c_0}((x-a_{0})^{\alpha_0+1}+\vep(x-a_{0}))\leq\int_{a_{0}}^{x}m_{0}^{\vep}=\int_{a_1}^{\overline{x}}m_{T}^{\vep}\leq c_1((\overline{x}-a_1)^{\alpha_1+1}+\vep(\overline{x}-a_1)).
\]
In particular, this implies that
\[
\frac{1}{c_0}(x-a_{0})^{\alpha_0+1}\leq c_1(\overline{x}-a_1)^{\alpha_1+1}\text{ or }\,\,\frac{1}{c_0}\vep(x-a_{0})\leq c_1\vep(\overline{x}-a_1),
\]
and since $\alpha_0\geq \alpha_1$ we deduce, for some constant $C$, that
\[
(x-a_{0})^{\alpha_0}\leq C\, (\overline{x}-a_1)^{\alpha_1}\,.
\]
As a result, using \rife{powersvep0}--\rife{powersvepT}, 
we have, for $\vep$ small enough,
\[
v(x,T) =\frac{\me_{0}(x)}{\me(\gamma^{\vep}(x,T),T)}\leq \frac{c_0(x-a_{0})^{\alpha_0}+\vep}{\frac{1}{c_1}(\ox-a_{1})^{\alpha_1}+\vep}\leq C.
\]
We thus conclude that $w\leq C$, hence
\begin{equation} \label{eq:gamma_x long time bound 1}
\gamma^{\vep}_x(x,t)\leq C(1+t^{\oal}).\end{equation}
This finally establishes that
\begin{equation}
\|\gamma^{\vep}\|_{W^{1,\infty}([a_{0},b_{0}]\times[0,T])}\leq C,\label{eq:-13}
\end{equation}
and, in particular,
\begin{equation}
\frac{1}{C}\me_{0}(x)\leq m^{\vep}(\gamma^{\vep}(x,t),t).\label{eq:-6}
\end{equation}
Now, recalling the proof of Theorem \ref{thm:well-posedness theorem},
we see that $m^{\vep}$ converges uniformly to the unique weak
solution $m$ of (\ref{eq:MFGPR}). Moreover, $\gamma^{\vep}\rightarrow\gamma$
uniformly in $[a_{0},b_{0}]\times[0,T]$, and (\ref{eq:-11}), (\ref{eq:-12})
follow. Now, assume that $x_{0}<\gamma(a_{0},t_{0})$. We have, for
$\vep$ small enough, $x_{0}<\gamma^{\vep}(a_{0},t_{0})$,
so Lemma \ref{lem.gamma} yields
\[
\int_{-R}^{x_{0}}\me\leq\int_{-R}^{a_{0}}m_{0}^{\vep},
\implies \int_{-R}^{x_{0}}m\leq\int_{-R}^{a_{0}}m_{0}=0, \]
and, thus, $m(x_{0},t_{0})=0$. Similarly, $m(x_{0},t_{0})=0$ whenever
$x_{0}>\gamma(b_{0},t_{0}).$ On the other hand, if $\gamma(a_{0},t_{0})<x_{0}<\gamma(b_{0},t_{0})$,
we have, by continuity, $\gamma(a_{0},t_{0})=\gamma(c_{0},t_{0})$
for some $a_{0}<c_{0}<b_{0}.$ As a result, (\ref{eq:-12}) yields
\[
m(x_{0},t_{0})=m(\gamma(c_{0},t_{0}),t_{0})\geq\frac{1}{C}m_{0}(c_{0})>0,
\]
and this proves (\ref{eq:free bdary charact}). We now recall from
the proof of Theorem \ref{thm:well-posedness theorem} that $(\ue,f(\me))$
is bounded in $C_{\text{loc}}^{2,\alpha}(([-R,R]\times[0,T])\cap\{m>0\})\times C_{\text{loc}}^{1,\alpha}(([-R,R]\times[0,T])\cap\{m>0\})$,
independently of $\vep$. Thus, for $(x,t)\in(a_{0},b_{0})\times[0,T]$,
letting $\vep\rightarrow0$ in the relations
\[
\gamma_{t}^{\vep}=-u_{x}^{\vep}(\gamma^{\vep}(x,t),t),\,\,\gamma_{x}^{\vep}=\frac{m_{0}^{\vep}(x)}{m^{\vep}(\gamma^{\vep}(x,t),t)}
\]
shows that $\gamma^{\vep}\rightarrow\gamma$ in $C_{\text{loc}}^{2}((a_{0},b_{0})\times[0,T])$,
\begin{equation} 
\gamma_{t}=-u_{x}(\gamma(x,t),t),\,\,\gamma_{x}=\frac{m_{0}(x)}{m(\gamma(x,t),t)},\label{eq:derivatives of flow}
\end{equation}
and, in particular, $\gamma\in C_{\text{loc}}^{2,\alpha}((a_{0},b_{0})\times[0,T])$.
Finally, letting $\vep\rightarrow0$ in (\ref{eq:flow equation epsilon})
yields (\ref{eq:flow equation}).

We now explain the necessary changes in the above proof to deal with
the case that $(u,m)$ solves (\ref{eq:MFGR}). We initially assume that $c_1>0$. The first modification
lies in the proof of (\ref{eq:gamma_x long time bound 1}), since our previous argument
to estimate $\gamma_{x}^{\vep}(\cdot,T)$ does not apply if $\me(\cdot,T)\neq\me_{T}$.
Assume then that the function $w$ achieves its maximum value at a point $(x_{0},T)$, with
$a_{0}<x_{0}<b_{0}$. 
We first observe that
\[
\ue(\cdot,T)=g(\me(\cdot,T)),
\]
so that 
\[
-\gamma_{t}^{\vep}(\cdot,T)=\ue_{x}(\gamma^\vep(\cdot,T), T) = g'(\gamma^\vep(\cdot,T)) \me_{x}(\gamma^\vep(\cdot,T), T)\,.
\]
Since $\me_x= \frac1{f'} f(\me)_x= \frac1{f'}\gamma_{tt}^\vep$, we get
\[ 
-\gamma_{t}^{\vep}(\cdot,T)= \frac{g'}{f'}\gamma_{tt}^{\vep}(\cdot,T)=Tc_1\frac{\theta_1}{\theta}(\me)^{\theta_1-\theta}\gamma_{tt}^{\vep}.
\]
Differentiating this relation with respect to $x$ once more yields
\[
T\, c_1\frac{\theta_1}{\theta}(\me)^{\theta_1-\theta}v_{tt}=-v_{t}-T\, c_1\frac{\theta_1}{\theta}(\theta_1-\theta)(\me)^{\theta_1-\theta-1}(\me_{x})^{2}\gamma_{x}f'(\me)\leq-v_{t}.
\]
Now, evaluating this at $x_{0}$, and using the fact that $w_{t}\geq0$
at $(x_{0},T)$, we get 
\[
Tc_1\frac{\theta_1}{\theta}(\me)^{\theta_1-\theta}v_{tt}\leq -K\oal T^{\oal-1}.
\]
 On the other hand, using $w_{x}=v_{x}=0$ and $w_{xx}=v_{xx}\leq0$, (\ref{eq:-17}) yields $v_{tt}\geq -\frac{2k_1}{v^{1+\theta}}$. Hence we get
\[
-\frac{2k_1}{v^{1+\theta}}\frac{\theta_1}{\theta}(\me)^{\theta_1-\theta} \leq -K \frac{\oal T^{\oal-2}}{c_1} \,. 
\]
Using the uniform bound on $m$ (see \eqref{eq:m epsilon upper bound}) we obtain, for $\vep$ sufficiently small,
\begin{equation} \label{eq:gamma_x(T) bound}
v\leq \left( \frac{2k_1 c_1\theta_{1}}{\oal\theta} \|m_0\|^{\theta_1-\theta}_{\infty}\right)^{\frac{1}{1+\theta}}\frac{T^{\frac{2-\oal}{1+\theta}}}{K^{\frac{1}{1+\theta}}} := C \frac{T^{\frac{2-\oal}{1+\theta}}}{K^{\frac{1}{1+\theta}}}=C \frac{T^{\oal}}{K^{\frac{1}{1+\theta}}}.
\end{equation}

If we choose $K>C^{\frac{1+\theta}{2+\theta}}$, then the right hand side of \eqref{eq:gamma_x(T) bound} is bounded above by $KT^{\oal}$, which yields $w(x,T)\leq 0$, completing the proof of (\ref{eq:gamma_x long time bound 1}).
We also observe that $g=c_1Tm^{\theta}$ satisfies $g'>0$, which, as seen in the proof of Theorem \ref{thm:well-posedness theorem},
is still sufficient to obtain uniform bounds of $(\ue,f(\me))$
in 
\begin{equation} \label{eq:3 1 1}
C_{\text{loc}}^{2,\alpha}(([-R,R]\times[0,T])\cap\{m>0\})\times C_{\text{loc}}^{1,\alpha}(([-R,R]\times[0,T])\cap\{m>0\}).
\end{equation}
When $c_1=0$, we can repeat the present proof for $g^\vep(m)=\vep m^{\theta}$. The functions $(\ue,f(\me))$ may still be estimated in $C^{2,\alpha}_{\text{loc}}\times C^{1,\alpha}_{\text{loc}}$ away from $t=T$, and we conclude as before, except that we are only able to obtain that $\gamma \in C_{\text{loc}}^{2,\alpha}\left((a_0,b_0)\times[0,T)\right)$. However, since we know from Theorem \ref{thm:well-posedness theorem} that $u\in C^{2,\alpha}_{\text{loc}}((\mathbb{R}\times [0,T])\cap \{m>0\}),$ the regularity of $\gamma$ up to $t=T$ follows from \eqref{eq:derivatives of flow}.
\end{proof}
We now obtain the optimal upper bound for the time evolution of the quantity $\gamma_x$, which is attained by the self-similar solutions of Section \ref{sec:self similar}.
\begin{cor}[Upper bound on $\gamma_x$]\label{cor: gamma_x upper bound} Under the assumptions of Theorem \ref{thm:characteriz free boundary}, let $(u,m)$ be a  solution to \eqref{eq:MFGR} or \eqref{eq:MFGPR}, let $\gamma$ be the flow of optimal trajectories for $(u,m)$, and define
\begin{equation} \label{eq:d(t) defi}
\mathscr{d}(t) = 
\begin{cases} 
t & \text{if } u \text{ solves \eqref{eq:MFGR}}, \\
\emph{\text{dist}}(t, \{0, T\}) & \text{if } u \text{ solves \eqref{eq:MFGPR}.}
\end{cases}
\end{equation}
Then there exists a constant $C>0$, with
\begin{equation} \label{eq:gamma_x upper bound C dependence}
C = 
\begin{cases} 
C\left(C_{0},\|f(m_{0})_{xx}^{-}\|_{L^{\infty}(\{m_{0}>0\})},\frac{\theta_1}{\theta},c_1\right) & \text{if } u \text{ solves \eqref{eq:MFGR}}, \\
C\left(C_{0},C_1,\|f(m_{0})_{xx}^{-}\|_{L^{\infty}(\{m_{0}>0\})}\right) & \text{if } u \text{ solves \eqref{eq:MFGPR}},
\end{cases}
\end{equation}
independent of $T$, such that
\[\gamma_x(x,t) \leq C(1+\mathscr{d}(t)^{\oal}).\]

\end{cor}
\begin{proof} In the proof of Theorem \ref{thm:characteriz free boundary}, we showed \eqref{eq:gamma_x long time bound 1}. In fact, by simply following the proof, one readily sees that the constant $C$ in \eqref{eq:gamma_x long time bound 1} may be chosen independently of $T$, and depending only on the quantities specified in \eqref{eq:gamma_x upper bound C dependence}. This observation applies for both \eqref{eq:MFGR} and \eqref{eq:MFGPR}. Thus, for \eqref{eq:MFGR} there is nothing left to prove. As for \eqref{eq:MFGPR}, repeating exactly the same argument for the function $w=v-K(T-t)^{\oal}$ yields
\[\gamma_x(x,t)\leq C(1+(T-t)^{\oal}).\]
Thus, combining this with \eqref{eq:gamma_x long time bound 1}, we conclude that, for $(u,m)$ solving \eqref{eq:MFGPR}
\[\gamma_x(x,t)\leq C(1+\min(t^{\oal},(T-t)^{\oal}))=C(1+\mathscr{d}(t)^{\oal}).\]
\end{proof}
\subsubsection{$C^{1,1}$ regularity, strict convexity, strict monotonicity, and long time behavior}

In this subsection we obtain, under adequate compatibility
and non-degeneracy assumptions on the data, uniform $W^{2,\infty}(0,T)$
estimates for the free boundary.  Additionally, we obtain strict convexity and
strict concavity for the left and right free boundary curves, respectively, and prove that for the terminal cost problem, \eqref{eq:MFGR}, the boundary is spreading outward. Finally, we quantify the exact rate of propagation of the support and the exact rate of decay in time for the density, which are the ones exhibited by the self-similar solutions of Proposition \ref{self-building}.

To obtain these extra properties on the free boundary, we strengthen the assumptions of the previous subsection. In particular, we will require  the following compatibility condition between terminal and initial data, namely that   
\begin{equation}
\alpha_0=\alpha_1,\,\,\theta_1=\theta.\label{eq:compatibility assumption}
\end{equation}
where $\alpha_0, \alpha_1, \theta_1$ are defined in \rife{eq:decay restatement m0}--\rife{eq: assumption g=00003Dpower}. This kind of assumption will guarantee 
that the function $\gamma_{x}$ is well-behaved at $t=T$. We will also strengthen  the nondegeneracy assumption on $m_0$ by requiring 
\begin{equation}
f(m_{0})_{xx}\leq0\text{ in }\{x\in (a_0,b_0) :\text{dist}(x,\{a_{0},b_{0}\})<\delta\}\text{ for some }\text{\ensuremath{\delta>0}}\,.\label{eq:concavity assumption}
\end{equation}
We observe that,
since $f(m_{0})$ is Lipschitz, (\ref{eq:concavity assumption}) necessarily
implies that   $\alpha_0=\frac{1}{\theta}$ in (\ref{eq:decay restatement m0}),
and
\begin{equation}
f(m_{0})_{x}(b_{0}^{-})<0<f(m_{0})_{x}(a_{0}^{+}).\label{eq:nondegeneracy}
\end{equation}
We begin by obtaining a uniform lower bound on $\gamma_{x}(\cdot,T)$
for solutions to (\ref{eq:MFGPR}). 
\begin{lem}
\label{lem:lem gamma_x(T) lower bd}Under the assumptions of Theorem
\ref{thm:characteriz free boundary}, let $(u,m)$ be a  solution
to \eqref{eq:MFGPR}, and assume that $\alpha_0=\alpha_1$. Then there
exists a constant  $C=C(C_0,C_1)$,
such that, for $x\in(a_{0},b_{0})$,
\[
\gamma_{x}(x,T)\geq\frac{1}{C}.
\]
\end{lem}

\begin{proof}
With no loss of generality, we assume that $x\in(a_{0},\frac{1}{2}\left(a_{0}+b_{0})\right)$,
and that $x$ is close enough to $a_{0}$ to guarantee that $\gamma(x,T)\in(a_1,\frac{1}{2}(a_1+b_1))$.
We have, by conservation of mass,

\[
\int_{a_{0}}^{x}m_{0}=\int_{a_1}^{\gamma(x,T)}m_{T}(\gamma(\cdot,T)).
\]
Thus, in view of (\ref{eq:decay restatement m0}) and (\ref{eq:decay restatement mT}),
\[
C_{0}(x-a_{0})^{\alpha_0+1}\geq \frac{1}{C_{1}}(\gamma(x,T)-a_1)^{\alpha_0+1} \implies \frac{(\gamma(x,T)-a_1)^{\alpha_0}}{(x-a_{0})^{\alpha_0}}\leq (C_{0}C_1)^{\frac{\alpha_0}{\alpha_0+1}}.
\]
This, combined with (\ref{eq:decay restatement m0}) and (\ref{eq:decay restatement mT})
once more, implies that
\[
\gamma_{x}(x,T)=\frac{m_{0}(x)}{m_{T}(\gamma(x,T),T)}\geq (C_{0}C_1)^{-1}\frac{(x-a_{0})^{\alpha_0}}{(\gamma(x,T)-a_1)^{\alpha_0}}\geq (C_{0}C_1)^{-(1+\frac{\alpha_0}{\alpha_0+1})}.
\]
\end{proof}
Next we obtain a global lower bound on $\gamma_{x}$. We also show that when \eqref{eq:concavity assumption} is strengthened to be both strict and global, as occurs in the self-similar solution, this bound can be improved to yield the optimal rate at which $\gamma_x$ may grow in time, complementing the upper bound of Corollary \ref{cor: gamma_x upper bound}.   

\begin{prop}
\label{prop:lower bound gamma x}(Lower bounds on $\gamma_x$) Under the assumptions of Theorem
\ref{thm:characteriz free boundary}, let $(u,m)$ be a solution
to \eqref{eq:MFGR} or \eqref{eq:MFGPR}. Assume that \eqref{eq:compatibility assumption} and \rife{eq:concavity assumption} hold. Then we have:
\begin{itemize}
\item[(i)] There exists a constant
$C>0$ such that, for $(x,t)\in(a_{0},b_{0})\times[0,T]$, $m(\gamma(x,t),t)\leq Cm_{0}(x)$, that is, 
\begin{equation}
\label{eq:-1} \gamma_x(x,t)\geq \frac{1}{C},
\end{equation}
and 
\begin{equation}\label{primeira}
C=
\begin{cases} 
C(C_{0},\delta^{-1})& \text{if } u \text{ solves \eqref{eq:MFGR}}, \\
C(C_{0},C_1,\delta^{-1}) & \text{if } u \text{ solves \eqref{eq:MFGPR}}.
\end{cases}
\end{equation}

\item[(ii)] Assume, in addition,  that $f(m_0)_{xx} \leq -\frac{1}{K} <0$ in $(a_0,b_0)$, and that $c_1>0$ in \rife{eq: assumption g=00003Dpower},  and let $\mathscr{d}$ be defined as in \eqref{eq:d(t) defi}. Then there exists a constant $C>0$ such that the sharp estimate
\begin{equation} \label{eq:-1 strict}\gamma_x(x,t)\geq  \frac{1}{C}(1+\mathscr{d}(t)^{\oal})\end{equation}
holds, and
\begin{equation}\label{segunda}
C=
\begin{cases} 
C\left(C_{0},c_1^{-1}
\right)& \text{if } u \text{ solves \eqref{eq:MFGR}}, \\
C\left(C_{0}, C_1,K\right) & \text{if } u \text{ solves \eqref{eq:MFGPR}}.
\end{cases}
\end{equation}

\end{itemize}
\end{prop}

\begin{proof}
We first treat the case in which $(u,m)$ solves (\ref{eq:MFGPR}).
Observe that,  since $f\in C^{\infty}(0,\infty)$, the interior Schauder estimates applied to \eqref{eq.elluBIS} imply that the solution $(u,m)$ is $C^{\infty}$ in the set $\{m>0\}\cap\{0<t<T\}$,
and so by Lemma \ref{eqfmgam} the function \[
v(x,t)=f(m(\gamma(x,t),t))
\]
solves, for $(x,t)\in[a_{0},b_{0}]\times[0,T]$,
\begin{equation}
-v_{tt}-\frac{\theta v}{\gamma_{x}^{2}}v_{xx}+v_{x}\frac{\theta v}{\gamma_{x}^{3}}\gamma_{xx}+\frac{\theta+1}{\theta}v^{-1}v_{t}^{2}=0.\label{eq:mg equation}\end{equation}
Since $v= \frac{f(m_0)}{\gamma_x^\theta}$ (by \rife{masspreservationBIS}), using (\ref{eq:flow equation}) and $\gamma_{tt}=  -\frac{v_x}{\gamma_x}$ (by \rife{eq.EulerBIS})
we deduce
\begin{equation}
\begin{split}
-v_{tt}-\frac{\theta v}{\gamma_{x}^{2}}v_{xx} &  =-\frac{v_{x}}{\gamma_{x}^{2}}\left(\frac{\theta f(m_0)\gamma_{xx}}{\gamma_x^{1+\theta}}\right)-\frac{\theta+1}{\theta}v^{-1}v_{t}^{2} =-\frac{v_{x}}{\gamma_{x}^{2}}\left(\frac{f(m_{0})_{x}}{\gamma_{x}^{\theta}}-\gamma_x\, \gamma_{tt} \right)-\frac{\theta+1}{\theta}v^{-1}v_{t}^{2}
\\ & 
\leq-\frac{v_{x}}{\gamma_{x}^{2}}\left(\frac{f(m_{0})_{x}}{\gamma_{x}^{\theta}}-v_{x}\right).\label{eq:mg subsol}
\end{split}
\end{equation}

Now, given $0<\vep<1$, for each $(x,t)\in[a_{0},b_{0}]\times[0,T]$,
we let
\[
w(x,t)=v(x,t)-Cf(m_{0})(x)-\vep t,
\]
where $C\geq1$ is a constant large enough to guarantee that $w(x,T)\leq0$,
and $w\leq0$ on the set $\{(x,t):\text{dist}(x,\{a_{0},b_{0}\})\geq\delta\}$.
Such a constant exists because of Lemma \ref{lem:lem gamma_x(T) lower bd},
and because, by (\ref{eq:decay restatement m0}), $m_{0}\geq\frac{1}{C_{0}}\delta^{\alpha_0}$
on $\{x:\text{dist}(x,\{a_{0},b_{0}\})\geq\delta\}$. Observe that, as a result of \eqref{eq:concavity assumption} and (\ref{eq:nondegeneracy}), we must have
\begin{equation}
f(m_{0})_{x}\neq0\text{ on }\{x:0<\text{dist}(x,\{a_{0},b_{0}\})<\delta\}.\label{eq:-23} \end{equation}
Let $(x_{0},t_{0})$ be an interior maximum of $w$ on the set 
\[
S:=\{w>-\vep t\}\cap\{\text{dist}(x,\{a_{0},b_{0}\})<\delta\}.
\]
We note that, in $S$, $v(x,t)>Cf(m_{0})$. Moreover, at $(x_{0},t_{0})$, $v_x=Cf(m_{0})_{x}$ and $D^{2}w\leq0$.
Therefore, in view of \eqref{eq:concavity assumption}, (\ref{eq:mg subsol}), and (\ref{eq:-23}), we
infer that
\begin{equation} \label{eq:gamma_x lower bound eq1}
0\leq-w_{tt}-\frac{\theta v}{\gamma_{x}^{2}}w_{xx}\leq\frac{-1}{\gamma_{x}^{2}}C\frac{f(m_{0})_{x}^{2}}{f(m_0)}\left(v-Cf(m_0)\right)+\frac{C\theta v}{\gamma_{x}^{2}}f(m_{0})_{xx}<0.
\end{equation}
This is a contradiction, so the maximum of $w$ in $S$ is achieved
at $\partial S$. At such a point, we have either $w=-\vep t$, $t=0$,
$t=T$, or $\text{dist}(x,\{a_{0},b_{0}\})=\delta$. By our choice
of $C$, in each of these cases we have $w\leq0$. Therefore, $w\leq0$
on all of $[a_{0},b_{0}]\times[0,T]$, that is,
\[
f(m(\gamma(x,t),t))\leq Cf(m_{0}(x))+\vep t,
\]
and the result follows by letting $\vep \rightarrow 0$.

Now, to modify this for the terminal cost problem $u(\cdot,T)=c_1Tf(m(\cdot,T))$
(recall that (\ref{eq:compatibility assumption}) holds), the only
issue we must address is that we do not know a priori that $w\leq0$
at $t=T$. It is therefore enough to prove that $w$ cannot achieve
a maximum in $S$ at some $(x_{0},T)$, where $x_{0}\in(a_{0},b_{0})$.
Assume otherwise. On one hand, we have, by \eqref{def.gammaBIS} and the continuity equation,
\begin{multline}\label{asdfasad}v_t(x,T)=f'(m)(m_t+m_x\gamma_t)(\gamma(x,T),T)=f'(m)(m_t-m_xu_x)(\gamma(x,T))\\
=f'(m)mu_{xx}(\gamma(x,T),T)=\theta v(x,T)u_{xx}(\gamma(x,T),T).
\end{multline}
On the other hand, by the definition of $v$ and the chain rule,
\begin{equation}\label{asdfasa123d} f(m)_{xx}(\gamma(x,T),T)=\left(\frac{v_x}{\gamma_x}\right)_x(x,T)\frac{1}{\gamma_x(x,T)}.\end{equation}
Differentiating the terminal condition implies that
$u_{xx}(\cdot,T)=c_1Tf(m)_{xx}(\cdot,T)$, so that, in view of \eqref{asdfasad} and \eqref{asdfasa123d}, we obtain
\begin{equation} \label{eq: gamma_x low bd final cost eq1}
v_{t}=\theta vu_{xx}(\gamma(x_{0},T),T)=\theta c_1Tvf(m)_{xx}(\gamma(x_{0},T),T)=\theta c_1Tv\left(\frac{1}{\gamma_{x}^{2}}v_{xx}-v_{x}\frac{1}{\gamma_{x}^{3}}\gamma_{xx}\right).
\end{equation}

Now, notice that $w_{xx}(x_0,T)=v_{xx}-Cf(m_{0})_{xx}\leq0$
and $f(m_{0})_{xx}(x_0)\leq0$, so $v_{xx}\leq 0$. Hence, in view of \eqref{eq:flow equation} and \eqref{eq.EulerBIS}, we obtain
\begin{equation}
v_{t}\leq-\theta c_1Tv_{x}\frac{v}{\gamma_{x}^{3}}\gamma_{xx}=-c_1Tv_{x}\left(\frac{f(m_{0})_{x}}{\gamma_{x}^{2+\theta}}-\frac{v_{x}}{\gamma_{x}^{2}}\right).\end{equation}
As a result, since $w_{t}\geq0$ and $w_{x}=0,$ that is, $v_{t}\geq \vep$ and $v_{x}=Cf(m_{0})_{x},$
recalling (\ref{eq:-23}) and the definition of $S$, we get
\[
\vep \leq v_t \leq -c_1TC\frac{1}{\gamma_{x}^{2}}\frac{f(m_{0})_{x}^{2}}{f(m_0)}\left(v-Cf(m_0)\right)\leq0.
\]
This is a contradiction, which proves \eqref{eq:-1}. \\
To prove part (ii), we now assume that $f(m_0)_{xx}\leq -\frac{1}{K}$, where $K>0$, and again we focus first on the case where $(u,m)$ solves \eqref{eq:MFGPR}. We repeat the above argument, but with a different choice for the function $w$, namely
\begin{equation} \label{eq:w defi strict}
    w(x,t)=v(x,t)-\zeta(t)f(m_0)(x) ,\,\,\text{ where }\,\,\zeta(t)=C\left(\frac{1}{t^{\oal\theta}}+\frac{1}{(T-t)^{\oal\theta}}.\right)
\end{equation}
Since $f(m_0)$ is now globally concave, we may also redefine $S$ to be simply
\[S:=\{w>0\}.\]
Then, instead of \eqref{eq:gamma_x lower bound eq1}, we obtain  
\begin{equation} \label{eq:gamma_x lower bound strict eq1}
\begin{split}
0\leq-w_{tt}-\frac{\theta v}{\gamma_{x}^{2}}w_{xx}& \leq\frac{-1}{\gamma_{x}^{2}}\zeta(t_0)\frac{f(m_{0})_{x}^{2}}{f(m_0)}w+\zeta''(t_0)f(m_0)+\frac{\zeta(t_0)\theta v}{\gamma_{x}^{2}}f(m_{0})_{xx} \\
& \leq f(m_0)^{-\frac2\theta}\left(f(m_0)^{1+\frac2\theta}\zeta''(t_0)+\zeta(t_0)\theta f(m_{0})_{xx}v^{1+\frac2\theta}\right)
\end{split}
\end{equation}
where we used that $\gamma_x= \left(\frac{f(m_0)}v\right)^{\frac1\theta}$ due to \rife{masspreservationBIS}. Since $f(m_0)_{xx} \leq -\frac{1}{K}$, and $v>\zeta(t)f(m_0)$ because $w>0$, we estimate the right hand side of  \rife{eq:gamma_x lower bound strict eq1}  obtaining
$$  
0 \leq -f(m_0)\left( - \zeta''(t_0)+\theta \frac{1}{K} \zeta(t_0)^{2+\frac2\theta}\right),
$$
It is straightforward to check that,  choosing $C$ sufficiently large, independently of $T$,  $\zeta$ satisfies
\[-\zeta''(t)+\theta \frac{1}{K}\zeta(t)^{2+\frac2\theta}>0.\]
Therefore, with this choice of $\zeta$, we obtain a contradiction. Moreover, since $w(\cdot,0)\equiv w(\cdot,T)\equiv -\infty$, we conclude that the set $S$ must be empty. That is,
\[v(x,t)\leq \zeta(t)f(m_0)(x),\,\,\,\, (x,t)\in (a_0,b_0)\times (0,T),\]
which, combined with \eqref{eq:-1}, readily implies \eqref{eq:-1 strict}. Finally, we prove \eqref{eq:-1 strict} for the case in which $(u,m)$ solves \eqref{eq:MFGR}. We again define $w$ according to \eqref{eq:w defi strict}, but this time $\zeta$ is defined by
\[\zeta(t)=\frac{C}{t^{\oal \theta}},\,\,\,\, t\in (0,T).\]
We may now follow the same proof as for \eqref{eq:MFGPR}, with the only issue being that we no longer have $w(\cdot,T)\equiv -\infty$, and, thus, we must consider the case of a maximum point $(x_0,T)$ of $w$.  We begin by noticing that, since $w_{xx}=v_{xx}-\zeta(T)f(m_0)_{xx}\leq0$, $w_x=0$, and $w_t\geq 0$, \eqref{eq: gamma_x low bd final cost eq1} implies 
\[ \zeta'(T)f(m_0)\leq v_t \leq -\frac{\theta c_1 T v}{\gamma_x^2} \frac{1}{K}\zeta(T) -c_1TC\frac{1}{\gamma_{x}^{2}}\frac{f(m_{0})_{x}^{2}}{f(m_0)}w \leq -\frac{\theta c_1 T v}{\gamma_x^2} \frac{1}{K}\zeta(T).  \]
This may be rearranged as
\[v(x_0,T)\leq\left( \frac{\oal K}{c_1}\right)^\frac{\theta}{2+\theta} T^{-\oal \theta}f(m_0)(x_0):=C_2 T^{-\oal \theta}f(m_0)(x_0).  \]
Therefore, if we choose $C>C_2$, we conclude that $v(x_0,T)\leq \zeta(T)f(m_0)(x_0),$ as wanted.
\end{proof}
We can now establish our main regularity result for the free boundary.
\begin{thm}[Regularity and convexity of the free boundary]\label{thm:free boundary regularity}
Under the assumptions of Theorem \ref{thm:characteriz free boundary},
let $(u,m)$ be a solution to \eqref{eq:MFGR} or \eqref{eq:MFGPR}.
Assume that (\ref{eq:compatibility assumption}) and (\ref{eq:concavity assumption})
hold. Let $\gamma_{L}=\gamma(a_{0},\cdot),\,\,\gamma_{R}=\gamma(b_{0},\cdot)$
be, respectively, the left and right free boundary curves. Then $\gamma_{L},\gamma_{R}\in W^{2,\infty}(0,T)$,
and there exist constants $K_1,K_2$, with $$K_1=K_1(C_{0},\|\gamma_{x}\|_{\infty},|f(m_{0})_x(a_0^+)|^{-1},|f(m_{0})_x(b_0^-)|^{-1})\, \text { and }\,K_2=K_2(C_{0},\|\gamma_{x}^{-1}\|_{\infty},|f(m_{0})_x(a_0^+)|,|f(m_{0})_x(b_0^-)|)$$
such that, for a.e. $t\in [0,T]$,
\begin{equation} \label{convexity of free boundary}
\frac{1}{K_1}\leq \Ddot\gamma_{L}(t)\leq K_2,\,\text{ and }\,-K_2\leq \Ddot\gamma_{R}(t)\leq-\frac{1}{K_1}.
\end{equation}
Moreover, when $(u,m)$ solves \eqref{eq:MFGR}, we have, for $t\in [0,T],$
\begin{equation}\label{monotonicity of free boundary} -K_2(c_1T+(T-t))\leq \dot \gamma_L(t)\leq -\frac{1}{K_1}(c_1T+(T-t)),\,\text{ and }\,\frac{1}{K_1}(c_1T+(T-t))\leq \dot \gamma_R(t)\leq K_2(c_1T+(T-t)).\end{equation}
\end{thm}
\begin{proof}
By symmetry, it is enough to show the estimates for $\gamma_{L}$.
Let $t\in(0,T)$, and let $h>0$ be such that $(t-h,t+h)\subset(0,T)$.
We begin by noting that (\ref{eq:flow equation}) may be written as
\[
\gamma_{tt}=\frac{1}{\gamma_{x}}\left(\frac{f(m_{0})}{\gamma_{x}^{\theta}}\right)_{x}.
\]
We therefore have, for $(x_{0},\tau)\in(a_{0},\frac{a_{0}+b_{0}}{2})\times[0,T],$

\begin{equation}\label{sabau}
 \int_{a_{0}}^{x_{0}}\gamma_{tt}(x,\tau)dx
=  \int_{a_{0}}^{x_{0}}\frac{1}{\gamma_{x}}\left(\frac{f(m_{0})}{\gamma_{x}^{\theta}}\right)_{x}dx
= \frac{f(m_{0})}{\gamma_{x}^{1+\theta}}(x_{0})-\int_{a_{0}}^{x_{0}}\left(\frac{1}{\gamma_{x}}\right)_{x}\frac{f(m_{0})}{\gamma_{x}^{\theta}}dx,
\end{equation}
where in the last step we integrated by parts and used the fact that
$\gamma_{x}$ is bounded below and $f(m_{0})(a_{0})=0$. Using
the identity
\[
\left(\frac{1}{\gamma_{x}}\right)_{x}\frac{f(m_{0})}{\gamma_{x}^{\theta}}=\frac{1}{\theta}\left(\frac{1}{\gamma_{x}}\left(\frac{f(m_{0})}{\gamma_{x}^{\theta}}\right)_{x}-\frac{f(m_{0})_{x}}{\gamma_{x}^{1+\theta}}\right),
\]
we infer from \rife{sabau} that
$$
\int_{a_{0}}^{x_{0}}\gamma_{tt}(x,\tau)dx  =  \left(1-\frac1\theta\right) \frac{f(m_{0})}{\gamma_{x}^{1+\theta}}(x_{0})+ \frac1\theta \int_{a_{0}}^{x_{0}} \left(\frac{1}{\gamma_{x}}\right)_{x}\frac{f(m_{0})}{\gamma_{x}^{\theta}}dx + \frac1\theta \int_{a_{0}}^{x_{0}} \frac{f(m_{0})_{x}}{\gamma_{x}^{1+\theta}}dx .
$$
Multiplying this equality by $\theta$ and adding to \rife{sabau} yields
\[
\int_{a_{0}}^{x_{0}}\gamma_{tt}(x,\tau)dx=\frac{\theta}{\theta+1}\frac{f(m_{0})}{\gamma_{x}^{1+\theta}}(x_{0})+\frac{1}{\theta+1}\int_{a_{0}}^{x_{0}}\frac{f(m_{0})_{x}}{\gamma_{x}^{1+\theta}}dx.
\]
Recalling that $f(m_0)$ is Lipschitz, with $f(m_0)(a_0)=0$, and that $\gamma_{x}$ is bounded above and below, we conclude that, for $x_{0}$ sufficiently
close to $a_{0}$, 
\begin{equation}
\frac{1}{C_1}(x_{0}-a_{0})\leq\int_{a_{0}}^{x_{0}}\gamma_{tt}(x,\tau)dx\leq C_2(x_{0}-a_{0}).\label{eq:-21}
\end{equation}
Next, observe that, for $x\in(a_{0},x_{0})$,
\[
\gamma(x,t+h)+\gamma(x,t-h)-2\gamma(x,t)=\int_{0}^{h}\int_{t-s}^{t+s}\gamma_{tt}(x,\tau)d\tau ds,
\]
and, therefore, integrating both sides,
\begin{align*}
\int_{a_{0}}^{x_{0}}\gamma(x,t+h)+\gamma(x,t-h)-2\gamma(x,t)dx= & \int_{0}^{h}\int_{t-s}^{t+s}\int_{a_{0}}^{x_{0}}\gamma_{tt}(x,\tau)dxd\tau ds.
\end{align*}
Using (\ref{eq:-21}), we see that this yields
\[
\frac{1}{C}(x_{0}-a_{0})h^{2}\leq\int_{a_{0}}^{x_{0}}\gamma(x,t+h)+\gamma(x,t-h)-2\gamma(x,t)dx\leq C(x_{0}-a_{0})h^{2},
\]
so, dividing by $(x_{0}-a_{0})$ and letting $x_{0}\rightarrow a_{0}^{+}$,
we obtain
\[
\frac{1}{C}h^{2}\leq\gamma_{L}(t+h)+\gamma_{L}(t-h)-2\gamma_{L}(t)\leq Ch^{2},
\]
which yields \eqref{convexity of free boundary}. 

Now assume that $(u,m)$ solves \eqref{eq:MFGR}. Recall that, since $u(\cdot,T)=c_1Tf(m(\cdot,T))$, we have
\begin{equation}\label{eq: gamma_t Robin bc} \gamma_t(\cdot,T)=-c_1T \gamma_{tt}(\cdot,T). \end{equation}
We observe that, by Taylor's theorem, for $x\in (a_0,b_0)$ and small $h>0$, we have
\[\gamma(x,T-h)=\gamma(x,T)-h\gamma_t(x,T)+\int_{T-h}^{T}(s-(T-h))\gamma_{tt}(x,s)ds. \]
Thus, integrating from $a_0$ to $x_0$ and using \eqref{eq: gamma_t Robin bc}, we obtain
\[\int_{a_0}^{x_0}\gamma(x,T-h)-\gamma(x,T)dx=hc_1T \int_{a_0}^{x_0}\gamma_{tt}(x,T)dx+\int_{T-h}^{T}\int_{a_0}^{x_0}(s-(T-h))\gamma_{tt}(x,s)dsdx, \]
so we infer from \eqref{eq:-21} that, for $x_0$ sufficiently close to $a_0$,
\[\frac{c_1T}{C}h(x_0-a_0) +\frac{1}{2C}h^2(x_0-a_0)\leq\int_{a_0}^{x_0}\gamma(x,T-h)-\gamma(x,T) dx \leq Cc_1Th(x_0-a_0)+\frac{1}{2}Ch^2(x_0-a_0).\]
Dividing by $(x_0-a_0)$ and letting $x_0 \rightarrow a_0^+$, we see that
\[ \frac{c_1T}{C}h + \frac{1}{2C}h^2\leq \gamma_L(T-h)-\gamma_L(T) \leq Cc_1Th +\frac{1}{2}Ch^2.\]
Finally, dividing by $h$ and letting $h\rightarrow 0^+$ yields
\[-Cc_1T\leq\dot \gamma_L(T)\leq \frac{-c_1T}{C}.\]
Thus, in view of \eqref{convexity of free boundary}, and noting that $\dot \gamma_L(t)= \dot \gamma_L(T) - \int_{t}^{T}\Ddot \gamma_L(s)ds$, we obtain \eqref{monotonicity of free boundary}. \end{proof}
Finally, we show that the support grows with algebraic rate $\oal=\frac{2}{2+\theta}$, and the density decays to $0$ with algebraic rate $-\oal$, as is expected from the model case of Section \ref{sec:self similar}.
\begin{thm}[Optimal rate of propagation and long time decay] \label{thm:long time}Under the assumptions of Theorem \ref{thm:characteriz free boundary},
let $(u,m)$ be a solution to \eqref{eq:MFGR} or \eqref{eq:MFGPR}, let $\gamma$ be the associated flow of optimal trajectories, and let $\mathscr{d}:[0,T]\rightarrow [0,\infty)$ be defined by \eqref{eq:d(t) defi}. Assume also that $-K \leq f(m_0)_{xx}\leq -\frac{1}{K}$ in $(a_0,b_0)$ for some $K>0$. Then there exists a constant $C>0$, with
\begin{equation}
C=
\begin{cases} 
C\left(C_{0}, c_1, c_1^{-1},|a_0|, |b_0|,  K\right)& \text{if } u \text{ solves \eqref{eq:MFGR}}, \\
C\left(C_{0},C_1,|a_0|,|b_0|, |a_1|, |b_1|, K\right) & \text{if } u \text{ solves \eqref{eq:MFGPR}},
\end{cases}
\end{equation}
such that, for every $(x,t)\in [a_0,b_0] \times [0,T]$,
\begin{equation} \label{eq: gamma long time}  \frac{1}{C} (1+\mathscr{d}(t)^{\oal})\leq |\emph{supp}(m(\cdot,t))|\leq C(1+\mathscr{d}(t)^{\oal}),\quad|\gamma(x,t)|\leq C(1 +\mathscr{d}(t)^{\oal}),
\end{equation}
\begin{equation} \label{eq: m long time}  \frac{1}{C} \frac{m_0(x)}{(1+\mathscr{d}(t)^{\oal})}\leq m(\gamma(x,t),t)\leq C\frac{m_0(x)}{(1+\mathscr{d}(t)^{\oal})}.    
\end{equation}
\end{thm}
\begin{proof}
Recalling \eqref{masspreservationBIS}, we observe that \eqref{eq: m long time} is simply obtained from combining the upper bound on $\gamma_x$ from Corollary \ref{cor: gamma_x upper bound} and the lower bound on $\gamma_x$ from Proposition \ref{prop:lower bound gamma x}, 
\begin{equation}\label{eq: gamma_x two sided sharp}\frac{1}{C}(1+\mathscr{d}(t)^{\oal})\leq \gamma_x(x,t) \leq C(1+\mathscr{d}(t)^{\oal}).\end{equation}
Now, integrating \eqref{eq: gamma_x two sided sharp} between $a_0$ and $b_0$ immediately yields the first equation of \eqref{eq: gamma long time}. Moreover, for the second equation of \eqref{eq: gamma long time}, it suffices to show that, for each $t_0\in [0,T]$, there exists at least one $x_0\in [a_0,b_0]$ such that $|\gamma(x_0,t_0)|\leq C$. When $(u,m)$ solves \eqref{eq:MFGR}, this follows from the fact that, by Theorem \ref{thm:free boundary regularity},  $\text{supp}(m(\cdot,t))$ is expanding, and, thus, in particular, $\gamma(x_0,t_0)=a_0$ for some $x_0\in [a_0,b_0]$. On the other hand, if $(u,m)$ solves \eqref{eq:MFGPR}, Theorem \ref{thm:free boundary regularity} implies that $\text{supp}(m)$ is a convex set. Therefore, we may choose $\gamma(x_0,t_0)$ to be the first coordinate of the intersection between $\mathbb{R}\times \{t_0\}$ and the line segment $\{(1-s)(a_0,0)+s(a_1,T): 0\leq s \leq 1\}$.

\end{proof}

\subsection{Regularity of the solution up to the free boundary}\label{subsec:reguH}

\subsubsection{Intrinsic scaling and H\"older continuity of $m$}

In this subsection, we show how the bounds on $\gamma_{x}$ and the intrinsic scaling of the problem allow us to improve the logarithmic modulus of continuity for $m$ to a H\"{o}lder one. Throughout the section,  we continue to assume that $f$ is given by \eqref{eq: assumption f=00003Dpower} , and we will assume the conditions of Theorem \ref{thm:free boundary regularity}, namely \eqref{eq:support assumption}--\eqref{eq: assumption g=00003Dpower} together with \rife{eq:compatibility assumption} and \rife{eq:concavity assumption}. We will focus on  obtaining H\"older regularity estimates for the function 
\be\label{eq: v defi}
v(x,t):=f(m(\gamma(x,t),t)), \qquad (x,t)\in [a_0,b_0] \times [0,T].
\ee 
This is equivalent to obtaining H\"older estimates for $m(x,t)$, in view of \rife{eq:-11} and \rife{eq:-1}.  

Our first result is a simple corollary of the bounds on $\gamma_x$, stated in the form of a Harnack inequality (cf. \cite[Thm 11.1]{DiB2}). 
\begin{prop}[Harnack inequality]
\label{prop:Harnack} Let the assumptions of Theorem \ref{thm:free boundary regularity} be in place, and let $(u,m)$ be a solution to \eqref{eq:MFGR} or \eqref{eq:MFGPR}, and let $v$ be defined by \eqref{eq: v defi}. There exists a constant 
\[
C=C(C_{0},\|\gamma_{x}\|_{\infty},\|\gamma_{x}^{-1}\|_{\infty},\theta)
\]
such that the following alternative holds. Let $x_{0}\in(a_{0},b_{0})$,
and let $\rho>0$ be such that $(x_{0}-\rho,x_{0}+\rho)\subset (a_{0},b_{0})$.
Then either $\rho\geq\text{\ensuremath{\frac{1}{2}}\emph{dist}}(x_{0},\{a_{0},b_{0}\})$
and $\sup_{(x_{0}-\rho,x_{0}+\rho)\times[0,T]} v\leq C\rho$, or
$\rho<\text{\ensuremath{\frac{1}{2}}\emph{dist}}(x_{0},\partial(a_{0},b_{0}))$
and

\begin{equation} \label{eq:DG Harnack}
\sup_{(x_{0}-\rho,x_{0}+\rho)\times[0,T]}v\leq C\inf_{(x_{0}-\rho,x_{0}+\rho)\times[0,T]}v.
\end{equation}
\end{prop}

\begin{proof}
For simplicity, we normalize $a_{0}=0$, and by symmetry we may assume
that $x_{0}<\frac{b_{0}}{2}.$ Then, in view of (\ref{eq:decay restatement m0}) (where, we recall, $\alpha_0=\frac1\theta$ under the current assumptions), and the fact that $\gamma_x=m_0v^{-\frac{1}{\theta}}$ is bounded above and below, we have, for some constant $C$,
and every $(x,t)\in(x_{0}-\rho,x_{0}+\rho)\times[0,T]$, 
\begin{equation}
\frac{1}{C}x\leq v(x,t)\leq Cx.\label{eq:-20}
\end{equation}
\textbf{Case 1:} $\rho<\frac{1}{2}x_{0}$. Then (\ref{eq:-20}) implies
\begin{equation*}
\sup_{(x_{0}-\rho,x_{0}+\rho)\times[0,T]}v  \leq C\left(\frac{3x_{0}}{2}\right) =3C^2 \left(\frac{1}{C}\left(\frac{1}{2}x_{0}\right)\right)
 \leq 3C^{2}\inf_{(x_{0}-\rho,x_{0}+\rho)\times[0,T]}v.
\end{equation*}
\textbf{Case 2: $\frac{1}{2}x_{0}\leq\rho$. }Then, in view of \eqref{eq:-20}
\[
\sup_{(x_{0}-\rho,x_{0}+\rho)\times[0,T]}v\leq C(x_0+\rho) \leq 3C\rho.
\]
\end{proof}
We note that, unlike in the linear theory, the above Harnack inequality is not yet sufficient to obtain H\"older regularity, because the result does not hold for, say, translations of the solution. Instead, we will proceed by obtaining analogues of the Caccioppoli inequality and De Giorgi type lemmas, adapted to the scaling of the equation satisfied by $v$ (recall \eqref{eq.v=fmgammamtheta}), which is much more diffusive in time than in space near the free boundary.
\begin{lem}[Intrinsic Caccioppoli inequality]
\label{lem:Caccioppoli}Let the assumptions of Theorem \ref{thm:free boundary regularity} hold, and let $(u,m)$ be a solution to \eqref{eq:MFGR} or \eqref{eq:MFGPR}. There exists a positive constant $C=C(\|\gamma_{x}\|_{\infty},\|\gamma_{x}^{-1}\|_{\infty})$
such that, for each $\zeta\in C_{c}^{\infty}((a_{0},b_{0})\times(0,T))$,
and each $k\geq0$, we have  
\begin{equation}
\int_{a_{0}}^{b_{0}}\int_{0}^{T}\zeta^{2}((\psi^{\pm}(v))_{x}^{2}+(\theta v)^{-1}(\psi^{\pm}(v))_{t}^{2})dtdx\leq C\int_{a_{0}}^{b_{0}}\int_{0}^{T} \psi^{\pm}(v)^{2}(\zeta_{x}^{2}+(\theta v)^{-1}\zeta_{t}^{2})dtdx,\label{eq:Caccioppoli}
\end{equation}
where $\psi^{\pm}(v)=(v-k)^{\pm}$.
\end{lem}
\begin{proof} 
In view of Lemma \ref{eqfmgam}, and recalling that $v$ is smooth in $(a_0,b_0)\times (0,T)$, we know that $v$ satisfies 
\begin{equation} \label{eq:DG v eq} -\left(\gamma_x^{-1}v_x\right)_x  - \left(\gamma_x(\theta v)^{-1}v_t\right)_t= 0.\end{equation}
Testing this equation against the function $\psi(v)^{\pm}\zeta^{2}$ yields
\begin{equation}
 \int_{a_{0}}^{b_{0}}\int_{0}^{T}\zeta^{2}(\gamma_x^{-1}v_x^2 +\gamma_x (\theta v)^{-1} v_t ^2) (\psi^{\pm})'(v) 
=-\int_{a_{0}}^{b_{0}}\int_{0}^{T}2\zeta\psi^{\pm}(v)( \gamma_x^{-1}v_x\zeta_x + \gamma_x (\theta v)^{-1}v_t\zeta_t).
\label{eq:caccio1}
\end{equation}
We may estimate the right hand side as follows:
\begin{multline}
\left| \int_{a_{0}}^{b_{0}}\int_{0}^{T} 2\zeta\psi^{\pm}(v)( \gamma_x^{-1}v_x\zeta_x + \gamma_x (\theta v)^{-1}v_t\zeta_t)\right|\leq \frac12\int_{a_{0}}^{b_{0}}\int_{0}^{T}\zeta^{2}(\gamma_x^{-1}v_x^2 +\gamma_x (\theta v)^{-1} v_t ^2)  \\
+2\int_{a_{0}}^{b_{0}}\int_{0}^{T}\psi^{\pm}(v)^2(\gamma_{x}^{-1}\zeta_{x}^{2}+\gamma_{x}(\theta v)^{-1}\zeta_{t}^{2}).\label{eq:caccio2}
\end{multline}
We also notice that $(\psi^{\pm})'(v)= \pm \chi_{\pm(v-k)\geq0}$, so that, as a result of \eqref{eq:caccio1} and \eqref{eq:caccio2},
\[\int_{a_{0}}^{b_{0}}\int_{0}^{T}\zeta^{2}(\gamma_x^{-1}(\psi^{\pm}(v))_x^2 +\gamma_x (\theta v)^{-1} (\psi^{\pm}(v))_t ^2)  \leq 4\int_{a_{0}}^{b_{0}}\int_{0}^{T}\psi^{\pm}(v)^2(\gamma_{x}^{-1}\zeta_{x}^{2}+\gamma_{x}(\theta v)^{-1}\zeta_{t}^{2}).\]
The result now follows from the fact that $\gamma_x$ is bounded above and below by positive constants.

\end{proof}

Our De Giorgi type lemma will be proved for the following special domains, adapted to the scaling of the equation \eqref{eq:DG v eq}.
\begin{defn}
Given $(x_{0},t_{0})\in(a_{0},b_{0})\times(0,T)$ and $\rho>0$, we
define the intrinsic rectangle $R_{\rho}(x_{0},t_{0})$ of radius
$\rho$ centered at $(x_{0},t_{0})$ by
\[
R_{\rho}(x_{0},t_{0}):=(x_{0}-\rho,x_{0}+\rho)\times(t_{0}-(\theta v(x_{0},t_{0}))^{-\frac{1}{2}}\rho,t_{0}+(\theta v(x_{0},t_{0}))^{-\frac{1}{2}}\rho).
\]
\end{defn}
\begin{lem}[De Giorgi type lemma on intrinsic rectangles]
\label{lem:De Giorgi} Let the assumptions of Theorem \ref{thm:free boundary regularity} hold, and let $(u,m)$ be a solution to \eqref{eq:MFGR} or \eqref{eq:MFGPR}.  There exists a positive constant $\nu$, with
\[
\nu^{-1}  =\nu^{-1}(C_0,\|\gamma_{x}\|_{\infty},\|\gamma_{x}^{-1}\|_{\infty}),\]
such that the following holds. Let $(x_{0},t_{0})\in(a_{0},b_{0})\times(0,T)$, $0<r<\frac{1}{4}$,
\[
\,\mu^{-}=\min_{R_{4\rho}(x_{0},t_{0})}\mg,\qquad \,\,\mu^{+}=\max_{R_{4\rho}(x_{0},t_{0})}\mg,\qquad \,\,\omega=\mu^+-\mu^-=\text{\emph{osc}}_{R_{4\rho}(x_0,t_0)}(v),
\]
and 
\begin{equation}
\rho\leq\frac{1}{8}\min\left(\text{\emph{dist}}(x_{0},\{a_{0},b_{0}\}),   (\theta v(x_{0},t_{0}))^{\frac{1}{2}}\text{\emph{dist}}(t_{0},\{0,T\})\right).\label{eq:distance assumption}
\end{equation}
Then 
\[
|\{\pm(\mg-(\mu^{\pm}\mp 2r\omega))>0\}\cap R_{2\rho}(x_{0},t_{0})|\leq\nu|R_{2\rho}(x_{0},t_{0})|\text{ implies that }\pm \mg\leq \pm(\mu^{\pm}\mp r\omega)\text{ in }R_{\rho}(x_{0},t_{0}).
\]

\end{lem}
\begin{proof}
We begin by noting that, in view of (\ref{eq:distance assumption})  and Proposition
\ref{prop:Harnack},
$R_{4\rho}(x_{0},t_{0})\subset(a_{0},b_{0})\times(0,T)$ and, for some constant $k_1>0$,
\begin{equation}
\max_{R_{4\rho}(x_{0},t_{0})}\mg\leq k_{1}\min_{R_{4\rho}(x_{0},t_{0})}\mg.\label{eq:degio 2}
\end{equation}
We now define, for $n\geq0$,
\[
r_{n}=r+\frac{1}{2^{n}}r,\qquad \,\,\rho_{n}=\rho+\frac{1}{2^{n}}\rho,\qquad \,\,k_{n}^{\pm}=\mu^{\pm}\mp r_{n}\omega.
\]
and we choose non-negative functions $\zeta_{n}\in C_{c}^{\infty}((a_{0},b_{0})\times[0,T])$
such that $\zeta_{n}\equiv1$ on $R_{\rho_{n+1}}(x_{0},t_{0})$, $\zeta_{n}\equiv0$
outside of $R_{\rho_{n}}(x_{0},t_{0})$, and
\[
|(\zeta_{n})_{x}|\leq\frac{C2^{n}}{\rho},\qquad |(\zeta_{n})_{t}|\leq\frac{C2^{n}(\theta v(x_{0},t_{0}))^{\frac{1}{2}}}{\rho}.
\]
As a result of Lemma \ref{lem:Caccioppoli}, (\ref{eq:Caccioppoli})
holds when taking $\zeta=\zeta_{n}$ and 
\[
\psi(\mg)=\psi_n^{\pm}(v):=(\mg-k_{n}^{\pm})_{\pm}.
\]
On $(x,t)\in Q_{2\rho}:=(x_{0}-2\rho,x_{0}+2\rho)\times(t_{0}-2\rho,t_{0}+2\rho)$, we now define the rescaled functions $w_n,\overline{\zeta_{n}}$:
by
\[
w_n^{\pm}(x,t)=\psi_{n}^{\pm}(\mg(x,t_{0}+(\theta \mg(x_{0},t_{0}))^{-\frac{1}{2}}(t-t_{0}))),\qquad \overline{\zeta}_{n}(x,t)=\zeta_{n}(x,t_{0}+(\theta \mg(x_{0},t_{0}))^{-\frac{1}{2}}(t-t_{0})).
\]
We see that (\ref{eq:Caccioppoli}) may be written as

\[
\int_{Q_{2\rho}}\overline{\zeta}_{n}^{2}( (w_n^{\pm})_{x}^{2}+ \mg(x_{0},t_{0})\mg^{-1}(w_n^{\pm})_{t}^{2})\leq C\int_{Q_{2\rho}}(w_n^{\pm})^{2} ((\overline{\zeta}_{n})_{x}^{2}+\mg(x_{0},t_{0})\mg^{-1}(\overline{\zeta}_{n})_{t}^{2}).
\]

In view of \eqref{eq:degio 2}, up to increasing the constant $C$, we thus obtain
\begin{equation}
\int_{Q_{2\rho}}\overline{\zeta}_{n}^{2}|Dw_n^{\pm}|^{2}\leq C\int_{Q_{2\rho}}|D\overline{\zeta}_{n}|^{2}(w_n^{\pm})^{2}\leq\frac{C4^{n}}{\rho^{2}}\int_{Q_{2\rho}}(w_n^{\pm})^{2}.\label{eq:caccio rescaled}
\end{equation}
This is now the usual (rather than intrinsic) Caccioppoli inequality,
and thus by the standard De Giorgi iteration argument (e.g. see \cite[Lem. 5]{Vasseur}), writing $w_{\infty}^{\pm}=\lim_{n\rightarrow\infty}w_n^{\pm}$,
we see that there exists $\nu>0$ such that
\begin{equation}
|\{w_{0}^{\pm}>0\}\cap Q_{2\rho}|\leq\nu|Q_{2\rho}|\text{ implies that }w_{\infty}^{\pm}=0\text{ in }Q_{\rho}.\label{eq:degio rescaled}
\end{equation}
Now, writing 
\[
\psi_{\infty}^{\pm}(\mg)=(\mg-(\mu^{\pm}\mp r\omega))_{\pm},
\]
 we have $w_{\infty}^{\pm}=\psi_{\infty}^{\pm}(v(x,t_{0}+(\theta \mg(x_{0},t_{0}))^{-\frac{1}{2}}(t-t_{0})) )$. Scaling back the time variable we have  $|Q_{2\rho}|=(\theta v(x_{0},t_{0}))^{\frac{1}{2}}|R_{2\rho}(x_{0},t_{0})|$ and
\[
|\{w_{0}^{\pm}>0\}\cap Q_{2\rho}|=(\theta v(x_{0},t_{0}))^{\frac{1}{2}}|\{\psi_{0}^{\pm}(\mg)>0\}\cap R_{2\rho}(x_{0},t_{0})|.
\]
Thus, we conclude the proof by noticing that (\ref{eq:degio rescaled}) may be equivalently written as:
\[
|\{\psi_{0}^{\pm}(\mg)>0\}\cap R_{2\rho}(x_{0},t_{0})|\text{\ensuremath{\leq\nu|R_{2\rho}(x_{0},t_{0})|} implies that }\psi_{\infty}^{\pm}(\mg)=0\text{ in }R_{\rho}(x_{0},t_{0}).
\]
\end{proof}

\begin{cor}[Reduction of oscillation]
\label{cor:reduction of osci} Let the assumptions of Theorem \ref{thm:free boundary regularity} be in place. Assume that the intrinsic rectangle
$R_{\rho}(x_{0},t_{0})$ satisfies (\ref{eq:distance assumption}).
There exists a constant $0<\sigma<1$, independent of the choice of
$x_0, t_0$ and $\rho$, such that
\begin{equation} \label{eq:osc decr}
\text{\emph{osc}}_{R_{\rho}(x_{0},t_{0})}v\leq\sigma\, \text{\emph{osc}}_{R_{4\rho}(x_{0},t_{0})} v.
\end{equation}
\end{cor}
\begin{proof}
The proof of this corollary, included for the reader's convenience, will be a standard application of the classical arguments that yield interior H\"older continuity for functions that satisfy the Caccioppoli inequality (originally due to E. De Giorgi \cite{degiorgi}. See also, for instance, \cite{Vasseur}). Let $\nu, r,\mu^+, \mu^-,$ and $\omega$ be as in Lemma \ref{lem:De Giorgi}. To simplify notation, we write $R_{\rho}:=R_{\rho}(x_0,t_0).$

We begin by defining $z_0:[-2,2]\times [-2,2] \rightarrow [-1, 1]$ as
$$z_0(x,t):=2\omega^{-1}\left(v(x_0+\rho (x-x_0), t_0 + (\theta v(x_0,t_0))^{-\frac{1}{2}}\rho (t-t_0))- \mu^-\right)-1.$$
Notice that, since $16 = |[-2,2]\times[-2,2] |$, we must have
\begin{equation} \label{eq:DG alternative}|z_0 \leq 0| \geq 8 \text{ or } |z_0 \geq 0|\geq 8.\end{equation}
We assume the former, and remark that the proof in the alternative case is completely analogous. We define, for $n\geq 1$, $z_n:[-2,2]\times [-2,2] \rightarrow (-\infty, 1]$ by
\begin{equation}
    z_n(x,t):=\frac{1}{4r^n}(z_0-(1-4r^{n})).
\end{equation}
Notice that $z_n$ is nonincreasing, and we have
$$z_n \geq 0 \Longleftrightarrow v \geq \mu^+ - 2r^n\omega,$$
and, thus, by a change of variables, 
\begin{multline}\fint_{[-2,2]\times [-2,2]}|Dz_n^+|^2=
\fint_{R_{2\rho}  \cap \{v \geq \mu^+ - 2r^n\omega \}}\frac{\rho^2}{4r^{2n}\omega^2}( v_x^2 +(\theta v(x_0,t_0))^{-1} v_t^2)\\
=  \frac{\rho^2}{4 r^{2n}\omega^2}\fint_{R_{2\rho}}( \psi(v)_x^2 +(\theta v(x_0,t_0))^{-1} \psi(v)_t^2)
,\end{multline}
where $\psi(v)=(v-\mu^++2r^n\omega)^+.$ By Proposition \ref{prop:Harnack}, we thus infer that
\begin{equation}
\int_{[-2,2]\times [-2,2]}|Dz_n^+|^2\leq  \frac{ C \rho^2}{ r^{2n}\omega^2}\fint_{R_{2\rho}}( \psi(v)_x^2 +(\theta v)^{-1} \psi(v)_t^2)
\end{equation}
 where, as usual, $C$ is a generic constant that could be increased line by line.
On the other hand, Proposition \ref{prop:Harnack} and Lemma \ref{lem:Caccioppoli} imply that 
\begin{equation}
    \int_{R_{2\rho}}(\psi(v)_x^2 + (\theta v)^{-1}\psi(v)_t^2)\leq \frac{C}{\rho^2} \int_{R_{4\rho}} \psi(v)^2 \leq C \omega^2 r ^{2n} (\theta v(x_0,t_0))^{-1/2}.
\end{equation}
Thus, we deduce from the last two inequalities that 
\begin{equation} \label{eq:DG energy}    
\int_{[-2,2]\times [-2,2]}|Dz_n^+|^2 \leq C,\end{equation}
where $C$ is independent of $n$.
Now, assume that, for some $\delta>0$ and some $n\geq 1$, we have 
\begin{equation}
    \label{eq:DG inductive}
 |z_{n-1} \leq 0| \geq 8 + (n-1)\delta .\end{equation}
We will show that, as long as
\begin{equation}
    \label{eq:DG inductive 2}
 |z_n > 0|\geq\nu,\end{equation} then \eqref{eq:DG inductive} must also hold when $n$ is replaced by $n+1$.  Observe that, since we have
 \be\label{znnn}
 z_n= \frac{z_{n-1}-(1-r)}r\,,
 \ee
 then \eqref{eq:DG inductive 2} may be rewritten as
\begin{equation} \label{eq:DG iso 1} |z_{n-1} >(1-r)|\geq \nu.\end{equation}
On the other hand, since the first alternative in \eqref{eq:DG alternative} was assumed to hold, we have
\begin{equation} \label{eq:DG iso 2}|z_{n-1} \leq 0|\geq  |z_0 \leq 0|\geq 8.\end{equation}
We recall that $r<\frac{1}{4}$. A straightforward compactness argument\footnote{The explicit form of this estimate is known as the De Giorgi isoperimetric inequality (see \cite[Lem. 10]{Vasseur}).} then shows that, in view of \eqref{eq:DG energy}, \eqref{eq:DG iso 1}, and \eqref{eq:DG iso 2}, if $\delta>0$ is chosen sufficiently small, depending only on $\nu^{-1}$ and $C$, then
\[|0 <z_{n-1}  <1-r| \geq \delta.\]
Thus, we deduce from \eqref{eq:DG inductive}  and \rife{znnn} that
\[|z_{n} \leq 0| \geq |z_{n-1} \leq 0|+ |0 <z_{n-1}  <1-r| \geq 8 + n\delta. \]
Since the left hand side is bounded above by $16$, this recursive process must fail after finitely many steps, and, therefore, there exists $n\geq1$, depending only on $\nu^{-1}$ and $C$, such that \eqref{eq:DG inductive 2} does not hold.
Rewritten in terms of $v$, this means that
\[|\{v > \mu ^+ -2r^n \omega\}\cap R_{2\rho}| < \nu |R_{2\rho}|.\]
Thus, by Lemma \ref{lem:De Giorgi}, $v\leq \mu^+ -r^n \omega$ in $R_{\rho}$. In particular,
$$\text{osc}_{R_{\rho}}(v)\leq \mu^+ -r^n \omega - \mu^- = (1-r^n)\omega=(1-r^n)\text{osc}_{R_{4\rho}}(v).$$

\end{proof}
We may now prove our main regularity result for the density.
\begin{thm}[H\"older continuity of the density up to the free boundary] \label{thm:f(m) holder} Let the assumptions of Theorem \ref{thm:free boundary regularity} be in place, and let $(u,m)$ be a solution to \eqref{eq:MFGR} or \eqref{eq:MFGPR}.
Let $0<\delta_0<\min\left(1,\frac{1}{4}T\right)$, and let $\sigma$ be the constant of Corollary \ref{cor:reduction of osci}. There exist constants $C>0$, $0<s<1$, with
\[
C=C(\delta_0^{-1},C_{0},\|\gamma_{x}\|_{\infty},\|\gamma_{x}^{-1}\|_{\infty},\theta,(1-\sigma)^{-1}),\,\,\,\,s^{-1}=s^{-1}(C_{0},\|\gamma_{x}\|_{\infty},\|\gamma_{x}^{-1}\|_{\infty},(1-\sigma)^{-1}),
\]
such that
\[
\| f(m(\gamma(\cdot,\cdot),\cdot))\|_{C^{s}([a_{0},b_{0}]\times[\delta_0,T-\delta_0])}\leq C.
\]
 In particular, $m$ is H\"older continuous on $\R \times [\delta_0,T-\delta_0]$.
\end{thm}

\begin{proof}
Let $(x_{0},t_{0}),(x_1,t_1)\in(a_{0},b_{0})\times[\delta_0,T-\delta_0]$, where $x_1<x_0$. As usual, we will consider intrinsic rectangles centered at $(x_0,t_0)$, so we abbreviate $R_{\rho}:=R_{\rho}(x_0,t_0)$. As in
the proof of Proposition \ref{prop:Harnack}, we may assume that
$a_{0}=0$ and $x_{0}<\frac{1}{2}\min\left(b_{0},1\right)$, and write, for some constant $k_1>1$, and each $(x,t)\in R_{4 \rho},$
\begin{equation} \label{eq:DG A}
\frac{1}{k_1}x\leq v(x,t)\leq k_1x.
\end{equation}
Therefore,
$x_{0}\leq \sqrt{x_0} \leq \sqrt{k_1v(x_{0},t_{0})}$.
Hence, letting
\[
\rho_{0}=\frac{1}{8\sqrt{k_1}(1+\sqrt{\theta})}\delta_0 x_{0},\,\,a=\max(|x_1-x_{0}|,\sqrt{\theta v(x_{0},t_{0})}|t_1-t_{0}|),
\]
we see that, in particular, $\rho_{0}$ satisfies (\ref{eq:distance assumption}). Setting now
$$
a=\max(|x_1-x_{0}|,\sqrt{\theta v(x_{0},t_{0})}|t_1-t_{0}|)
$$
we distinguish two alternative cases.

\textbf{Case 1.} $(x_1,t_1)\in R_{4\rho_{0}}$. Equivalently,
we have $a<4\rho_{0}$. Let $n\geq0$ be the unique integer such that
\[
\frac{1}{4^{n}}\rho_{0}\leq a<\frac{1}{4^{n-1}}\rho_{0}.
\]
Iterating Corollary \ref{cor:reduction of osci}, we see that, in view of \eqref{eq:DG A},
\[
\text{osc}_{R_{4^{-(n-1)}\rho_{0}}}(v)\leq\sigma^{n-1}\text{osc}_{R_{4\rho_{0}}}(v)\leq \frac{1}{\sigma}\sigma^{n}k_1(x_{0}+4\rho_0)\leq C\sigma^{n}\rho_{0}\delta_0^{-1}.
\]
Moreover, by increasing the value of $\sigma$ if necessary, we may assume that $\sigma>\frac{1}{4}$ , so that $s=-\log\sigma(\log4)^{-1}$ satisfies $0<s<1$. Thus, observing that $(x_1,t_1)\in R_{4^{-(n-1)}\rho_0}$ and $n\geq-\log\left(\rho_{0}^{-1}a\right)(\log4)^{-1}$,
 we have
\[
|\mg(x_1,t_1)-\mg(x_{0},t_{0})|\leq C(\rho_{0}^{-1}a)^{s}\rho_{0}\delta_0^{-1}\leq C\rho_{0}^{1-s}\delta_0^{-1}(|x_1-x_{0}|^{s}+(\theta\mg(x_{0},t_{0}))^{\frac{s}{2}}|t_1-t_{0}|^{s}).
\]

\textbf{Case 2.} $(x_1,t_1)\notin R_{4\rho_{0}}.$ Then $a\geq4\rho_{0}$,
and so, since $x_1<x_{0}$, appealing to the lower bound on $\gamma_x=\left(f(m_0(x))v(x,t)^{-1}\right)^{\frac{1}{\theta}}$, and allowing the constant $C$ to increase at each step, we have
\begin{align*}
|\mg(x_1,t_1)-\mg(x_{0},t_{0})| & \leq C(f(m_{0})(x_1)+f(m_{0})(x_{0}))\leq  2CC_0^{\theta}x_0 \leq C\delta_0^{-1}\rho_{0}\\
 & \leq  C\delta_0^{-1}a \leq C\delta_0^{-1}(|x_{1}-x_{0}|+\sqrt{\theta \mg(x_{0},t_{0})}|t_{1}-t_{0}|).
\end{align*}
 Finally, to see that $m$ itself is also H\"older continuous, we simply observe that, by \eqref{eq:-11} and Proposition \ref{prop:lower bound gamma x}, the inverse of the map $(x,t) \mapsto \Gamma(x,t):=(\gamma(x,t),t)$ is Lipschitz. Therefore, since $f^{-1}:[0,\infty)\rightarrow [0,\infty)$ is H\"older continuous on bounded sets,
\[ m=f^{-1} \circ v \circ \Gamma^{-1}: \text{supp}(m) \cap (\mathbb{R}\times [\delta_0,T-\delta_0]) \rightarrow \mathbb{R}\] is the composition of H\"older continuous functions. 
\end{proof}
\begin{rem} The result of Theorem \ref{thm:f(m) holder} may be improved to obtain H\"older continuity up to the initial time $t=0$, by working with one-sided (in time) analogues of the intrinsic rectangles $R_{\rho}$. Moreover, in the case of \eqref{eq:MFGPR}, the H\"older regularity may be established up to $t=T$ as well, by simply imposing on $m_T$ the same assumptions as those of $m_0$, and applying the continuity result at $t=0$ to the reflected functions $(-u(x,T-t),m(x,T-t))$.

\end{rem}

\subsubsection{H\"older continuity of $Du$} 

We now address the regularity of $u$ in view of the H\"older regularity of $m$. For \eqref{eq:MFGPR}, since the solution is not unique outside of $\{m>0\}$, we choose to work with the specific $u$ that was constructed in the proof of Theorem \ref{thm:well-posedness theorem}. With this choice, for both \eqref{eq:MFGR} and \eqref{eq:MFGPR}, we observe that $u$ is the unique $\text{BUC}(\R \times [0,T])$ viscosity solution to the HJ equation with terminal condition $u(\cdot,T)$. This implies that $u$ satisfies the representation formula
$$
u(x,t)= \inf_{\gamma \in H^1((t,T)), \; \gamma(t)=x} \int_t^T \frac12 |\dot \gamma(s)|^2 +f(m(\gamma(s),s)) \ ds +u(\gamma(T),T). 
$$
We also note that, from the proof of Theorem \ref{thm:well-posedness theorem}, the function $-u(x,T-t)$ is also the unique viscosity solution to the HJ equation with terminal condition $-u(\cdot,0)$, and, therefore, also satisfies the representation formula.
\begin{thm} \label{thm: Du Holder}Assume that $f(m)$ is in {$C^\beta_{\emph{\text{loc}}}$.} Then the map $u$ is { $C^{1,\frac{\beta}{2}}_{\text{\emph{loc}}}$} in $\R\times (0,T)$. 
\end{thm} 

\begin{proof} In view of the proof of Theorem \ref{thm:well-posedness theorem}, we know that $u_x$ is globally bounded. Let us first check that $u$ is locally semiconcave with a semiconcavity modulus of the form  $\omega(r)= {Cr^{\beta/2}}$. The argument is known and goes back to \cite{CaSo}. Fix $(x,t)\in \R\times {I}${, where $I$ is a fixed compact sub-interval of $(0,T)$}. Let $\gamma$ be optimal for $u(x,t)$ and $0<|h|,h'<\tau$ small. We set 
$$
\gamma_{\pm}(s):= \gamma(\theta_1^\pm s+\theta_2^\pm) + \theta^\pm_3 (s-(t+\tau)) \qquad \forall s\in [t\pm h', t+\tau], 
$$
where $\theta_1^\pm = \frac{\tau}{\tau-\pm h'}$, $\theta_2^\pm= t- \theta_1^\pm(t\pm h')$, $\theta^\pm_3 =- \frac{\pm h}{\tau-\pm h'}$. Note that 
$$
\theta_1^\pm (t\pm h')+\theta_2^\pm = t, \; \theta_1^\pm (t+\tau)+\theta_2^\pm= t+\tau, \;
\gamma_{\pm}(t\pm h') = x\pm h, \; \gamma_{\pm}(t+\tau)= \gamma(t+\tau), 
$$
so that 
\begin{align*}
& u(x+h,t+h')+u(x-h,t-h')-2u(x,t) \\
& \leq \int_{t+h'}^{t+\tau} \frac12 |\dot \gamma_+|^2 + f(m(\gamma_+,s)) ds 
+ \int_{t-h'}^{t+\tau} \frac12 |\dot \gamma_-|^2 + f(m(\gamma_-,s)) ds
- 2  \int_{t}^{t+\tau} \frac12 |\dot \gamma|^2 + f(m(\gamma,s)) ds  \\ 
& \leq \int_{t}^{t+\tau} \frac{\theta_1^+}{2} |\dot \gamma +{\theta^+_3/\theta^+_1}|^2 +\frac{1}{\theta_1^+} f\left(m\left(\gamma(s)+\theta^+_3\left((s-\theta_2^+)/\theta_1^+-(t+\tau)\right),(s-\theta_2^+)/\theta_1^+\right)\right) ds \\
& \qquad + \int_{t}^{t+\tau} \frac{\theta_1^-}{2} |\dot \gamma +\theta^-_3/\theta^-_1|^2 +\frac{1}{\theta_1^-} f(m(\gamma(s)+\theta^-_3((s-\theta_2^-)/\theta_1^--(t+\tau)),(s-\theta_2^-)/\theta_1^-)) ds \\
& \qquad - 2  \int_{t}^{t+\tau} \frac12 |\dot \gamma|^2 + f(m(\gamma(s),s)) ds .
\end{align*}
Using the H\"older regularity of $f(m)$ and the fact that $u$ is bounded we find
\begin{align*}
& u(x+h,t+h')+u(x-h,t-h')-2u(x,t) \\
& \leq \int_{t}^{t+\tau} \frac12(\theta_1^+ +\theta_1^--2) |\dot \gamma|^2 + \dot \gamma( {\theta_3^++ \theta_3^-}) +\frac12 ( (\theta_3^+)^2/\theta_1^++ (\theta_3^-)^2/\theta_1^-)  ds \\
& \qquad + \int_{t}^{t+\tau} (1/\theta_1^++ 1/\theta_1^--2) f(m(\gamma(s), s)) ds\\
& \qquad  +{C_I } \int_{t}^{t+\tau}   \frac{1}{\theta_1^+} (|\theta^+_3((s-\theta_2^+)/\theta_1^+-(t+\tau){)}|{^\beta} + |{(}(s-\theta_2^+)/\theta_1^+)
-s|^{\beta}) ds\\
& \qquad  +{C_I}  \int_{t}^{t+\tau} \frac{1}{\theta_1^-} (|\theta^-_3((s-\theta_2^-)/\theta_1^--(t+\tau){)}|{^\beta} + |{(}(s-\theta_2^-)/\theta_1^-)-s|{^\beta}) ds \\
& \leq C{  \left(\frac{\tau (h')^2}{\tau^2-(h')^2} +\frac{\tau (h)^2}{\tau^2-(h')^2} + C_I( |h|^\beta \tau+|h'|^{\beta} \tau)\right) }
\end{align*}
where $C= C( \|u_x\|{_{\infty}})$, since $|\dot\gamma|= |u_x|$.
We choose $\tau=({|h|}+h')^{\delta}$ with ${\delta} = {1-\beta/2}$, which leads to 
$$
u(x+h,t+h')+u(x-h,t-h')-2u(x,t) \leq C({|h|}+h')^{{1+\frac{\beta}{2}}}. 
$$
This inequality implies the local semiconcavity of $u$ with modulus $\omega(r)=Cr^{{\beta/2}}$  \cite[Thm. 2.1.10]{CaSiBook}. 

We now set $w(x,t)= -u(x,T-t)$. Then, as above, $w$ is locally semiconcave with a modulus $\omega(r)=C{r^{\beta/2}}$ since the semiconcavity property does not rely on the regularity of the terminal value. Hence $u$ is semiconcave and semiconvex with modulus $\omega$. Following  \cite[Thm. 3.3.7]{CaSiBook}, the derivatives of $u$ are therefore locally ${\beta/2}-$H\"older continuous. \end{proof}

\section{Infinite speed of propagation: the case of entropic coupling}\label{subsec:infinitespeed}

This section is devoted to the special case of so-called entropic coupling, namely when $f(m)=\log m$. 
We will show that, in this case, the evolution of $m$ has the property of  infinite speed of propagation; marginals with compact support evolve into positive, smooth densities.  

\subsection{The periodic case}\label{subsec.periodic}

As before, we start by considering  periodic solutions  defined on the torus $\T$, for arbitrarily large $R$:
\begin{equation}\label{log-per}
\begin{cases}
-u_{t}+\frac{1}{2}u_{x}^{2}=\log(m) & (x,t)\in\T\times(0,T),\\
m_{t}-(mu_{x})_{x}=0 & (x,t)\in\T\times(0,T)\,,
\end{cases} 
\end{equation}
complemented  either with prescribed marginals
\be\label{plan-log}
m(x,0)=m_{0}(x),\,\,m(x,T)=m_{T}(x)\,, \,\,  x\in\T.
\ee
or with final pay-off condition
\be\label{mfg-log}
u(x,T)= g(m(x,T))\,, \,\,  x\in\T.
\end{equation}
We define solutions which are smooth in $(0,T)$, with  traces at $t=0,t=T$ in the space of measures. For this purpose, we denote by $\cP(\T)$ the set of Borel probability measure on $\T$, endowed with the weak--$*$ convergence.

\begin{defn}\label{sollog} We say that $(u,m)$ is a (classical) solution of \rife{log-per}  if $(u,m)  \in C^2(\T\times (0,T))\times C^1(\T\times (0,T))$, $m>0$  in $\T\times (0,T)$ and  the equations are satisfied in a classical sense. For  $m_0, m_T\in \cP(\T)$, we say that  \rife{plan-log}    is satisfied if    $m\in C([0,T];\cP(\T))$ and $m(0)=m_0, m(T)=m_T$. Respectively, we say that   \rife{mfg-log} is satisfied, if $m\in C([0,T];\cP(\T))$, $m(T)\in L^1(\T)$ and $\lim_{t\to T^-} u(x,t)= g(m(x,T))$ for every $x\in \T$.
\end{defn}

The main result of this subsection is the following theorem.

\begin{thm}\label{ex-log-torus} Assume that $m_0, m_T\in C_c(\T)$. Then the following holds:

\begin{enumerate}
\item There exists a unique (up to addition of a constant to $u$) solution $(u,m)$ of \rife{log-per}--\rife{plan-log} such that $m \in L^\infty(\T\times (0,T))$. In addition, $u,m\in C^{\infty}(\T\times (0,T))$.

\item Assume that $g(s)= c_T \log (s)$, for some $c_T\geq 0$. Then there exists a unique  solution $(u,m)$ of \rife{log-per}--\rife{mfg-log} such that $m \in L^\infty(\T\times (0,T))$. In addition, $u,m\in C^{\infty}(\T\times (0,T])$ and $m>0$ in $(0,T]$.
\end{enumerate}
\end{thm}

We recall (see Theorem \ref{thm:existence}) that, if  $m_0, m_T\in C^{1,\alpha}(\T)$ and are strictly positive, then the problems
 \rife{log-per}--\rife{plan-log} and   \rife{log-per}--\rife{mfg-log} admit  a unique classical solution $(u,m)$ (up to addition of  a constant to $u$, in case of planning conditions \rife{plan-log}).   Therefore, for the proof of Theorem \ref{ex-log-torus}, we will proceed by approximating $m_0, m_T$ with strictly positive smooth measures.   
 
 As a first step, we derive estimates which are independent of lower bounds of $m_0, m_T$ as well as independent of $R$. We denote by $W_2(\mu,\nu)$ the $2$--Wasserstein distance between measures $\mu,\nu\in \cP(\T)$, and, for a single measure $\mu\in L^1(\T)$, we denote the entropy of $\mu$ by $\cE(\mu)=\int_{\T} \mu\log\mu\,dx$  and by $M_2(\mu)=\int_{\T} |x|^2\,d\mu$ its second order moment.



\begin{prop}
\label{energy-log}
Let $(u,m)$ be a solution of    \rife{log-per},   and assume that $u,m$ are continuous in $\T\times [0,T]$. 

\begin{enumerate}

\item If $(u,m)$ solves \rife{log-per}--\rife{plan-log},  there exists a constant $K$, only depending on $T$,   $\cE(m_0), \cE(m_T), M_2(m_0), M_2(m_T)$ (and independent of $R$) such that
\be\label{en-bound}
\int_0^T \intt m\, |u_x|^2\, dxdt + \int_0^T \intt |m \log(m)|\, dxdt  \leq   K \,,
\ee
\be\label{momest}
\sup_{t\in (0,T)}\,\,   \intt m(t)|x|^2\, dx  \leq K,  
\ee
and
\be\label{pointt}
 \sup_{t\in (0,T)}\,\,  \intt m(t)u_x^2\, dx +   \sup_{t\in (0,T)}\,\, \left | \intt m(t)\, \log m(t)\right|   \leq K \,.
\ee
As a consequence, it also holds that
\be\label{wass-unif}
W_2(m(t),m(s)) \leq K \, |t-s|^{\frac12} \qquad \forall t,s\in (0,T)\,.
\ee
\item If  $(u,m)$ solves \rife{log-per}--\rife{mfg-log} (with $g(m)=c_T\log(m)$),  then the estimates \rife{en-bound}, \rife{momest}, \rife{wass-unif} hold, for some $K$ only depending on $T, \cE(m_0), M_2(m_0)$. If $c_T>0$, then \rife{pointt}   holds true as well,  with $K$ depending also on $(c_T)^{-1}$. 
\end{enumerate}
\end{prop}

\begin{proof} 
We first consider the case of problem \rife{log-per}--\rife{plan-log}.  Using the equation of $m$, we compute
\begin{align*}
\frac{d}{dt} \frac12 \intt m\, |x|^2\, dx & = \frac12 \intt m_t\, |x|^2\, dx = - \intt x\, m\, u_x
\\
& \leq \frac12 \intt m\, |x|^2\, dx + \frac 12 \intt m\, |u_x|^2\, dx\,.
\end{align*}
Hence, Gronwall's lemma implies that, for some constant depending only on $T$,
\be\label{Gro}
\intt m(t)\, |x|^2\, dx \leq C_T\, \int_0^T \intt m\, |u_x|^2\, dx+ \intt m_0\, |x|^2\, dx\,.
\ee
On the other hand, we know that $(m,u_x)$ is the optimizer of the functional 
$$
{\mathcal B}(m,v):= \int_0^T\!\! \!\intt \frac 12\, |v|^2 dm+  \int_0^T\!\! \!\intt m\log(m)\, dxdt\,,
\quad \text{subject to } \begin{cases} m_t - (vm)_x=0 \,\,  \hbox{in $\T\times (0,T)$} & \\ m(0)=m_0\,, 
m(T)=m_T \,.& \end{cases}
$$

This means that
$$
\int_0^T \intt m\, |u_x|^2\, dxdt + \int_0^T \intt m \log(m)\, dxdt 
\leq \int_0^T\!\! \!\int_{\T} \frac 12\, |v|^2 d\mu+  \int_0^T\!\! \!\int_{\T} \mu\log(\mu)\, dxdt
$$
for any $(\mu, v)$ such that $\mu_t = (\mu v)_x$, with $\mu(0)=m_0, \mu(T)=m_T$.  Let us take $\mu$  as the Wasserstein geodesic connecting $m_0, m_T$. By McCann's classical displacement convexity result for Wasserstein geodesics \cite{McCann}, we know that $\int_{\T} \mu\log(\mu)\, dx$ is convex, hence
$$
\int_{\T} \mu\log(\mu)\, dx \leq \max\left( \cE(m_0)\,, \cE(m_T)\right)\,.
$$
We deduce that
$$
\int_0^T \intt m\, |u_x|^2\, dxdt + \int_0^T \intt m \log(m)\, dxdt  \leq C_T\, W_2(m_0, m_T)+ \max\left( \cE(m_0)\,, \cE(m_T)\right)\,.
$$
This yields, for some constant $C$ independent of $R$,
\be\label{posneg}
\int_0^T \intt m\, |u_x|^2\, dxdt + \int_0^T \intt (m \log(m))_+\, dxdt \leq C + \int_0^T \intt (m\log m)_-dxdt\,.
\ee
Here and below, the constants will be independent of $R$, although they may depend on $T$, $\cE(m_0)$, $\cE(m_T)$, $M_2(m_0)$, $M_2(m_T)$.
Since $(s\log s)_-\leq c\,\sqrt s$ we have
\be\label{posneg1}
\begin{split}
\int_0^T \intt (m \log m)_-dxdt  & \leq c\, \int_0^T \intt \sqrt m\, dxdt 
\\ & \leq C+ \vep  \int_0^T \intt  m\,(|x|^2+1) dxdt + C_\vep \int_0^T \intt \frac1{1+ |x|^2}\, dxdt
\end{split}
\ee
by Young's inequality. Then \rife{posneg} yields
\be\label{posneg2}
\int_0^T \intt m\, |u_x|^2\, dxdt + \int_0^T \intt (m \log(m))_+\, dxdt   \leq C+ \vep  \int_0^T \intt  m\,(|x|^2+1) dxdt + C_\vep \int_0^T \intt \frac1{1+ |x|^2}\, dxdt\,.
\ee
From \rife{Gro}, after integration  we deduce
\begin{align*}
\int_0^T \intt  m\,(|x|^2+1) dxdt  & \leq T\, C_T \int_0^T \intt m\, |u_x|^2\, dxdt  + C_T \, M_2(m_0)
\\ & \leq \vep\, T\, C_T\int_0^T \intt  m\,(|x|^2+1) dxdt + C(\vep, T, m_0)\,.
\end{align*}
Choosing $\vep$ suitably small, we get an estimate for $\int_0^T \intt  m\,(|x|^2+1) dxdt $. In turn, from \rife{posneg1} and \rife{posneg2},  we deduce \rife{en-bound}.
Now, the right-hand side in  \rife{Gro} is controlled, and we get the estimate \rife{momest}. Moreover,  from the equation satisfied by $m$ and the estimate \rife{en-bound}, we immediately deduce \rife{wass-unif}.

We are left with the pointwise estimate of $\intt mu_x^2\, dx$. For this purpose, we observe that 
\be\label{preserve}
\frac{d}{dt} \left[\intt mu_x^2\, dx- \intt m\log m \, dx  \right]=0.
\ee
Therefore,
$$
\intt mu_x^2\, dx- \intt m\log m \, dx  = \frac1T \int_0^T \left\{\intt mu_x^2\, dx- \intt m\log m \, dx\right\} dt \,.
$$
We deduce that
\be\label{balance}
\left| \intt mu_x^2\, dx- \intt m\log m \, dx \right| \leq \frac1T \left\{\int_0^T \intt m\, |u_x|^2\, dxdt + \int_0^T \intt |m \log(m)|\, dxdt \right\} \leq C\,.
\ee
Since, by the displacement convexity formula \rife{eq:displacement integral}, we have 
$$
\intt m\log m \, dx \leq \max\left( \cE(m_0)\,, \cE(m_T)\right),
$$
we conclude by \rife{balance} that $\left|\intt m\log m \, dx\right|$ is bounded, so $ \intt mu_x^2\, dx$ is bounded above, uniformly in $t$.  This yields \rife{pointt}.

In case of problem \rife{log-per}--\rife{mfg-log}, the only difference is in the first estimate, \eqref{en-bound}.  By the optimality condition, using that $g(s)= c_T\log(s)$, we know that  
\begin{align*}
 \intt u(x,0)\,m_0\,dx   \leq   \frac12 \int_0^T \intt \mu |v|^2\, dxds  +  \int_0^T  \intt \mu\, \log \mu\, dxds  +  c_T \intt \mu_T\log(\mu_T)\, dx 
\end{align*} 
for any curve $\mu(t)$ such that $\mu(0)=m_0$ and $\mu(T)=\mu_T$, with $\mu_t= (\mu\,v)_x$.\\
It is enough to choose some $\mu_T$ with finite entropy to obtain a global estimate of the right-hand side. 
%
Since we also have, by duality,
$$
\int_0^T \intt m\, |u_x|^2\, dxdt + \int_0^T \intt m \log(m)\, dxdt + c_T \intt m(T)\log(m(T))dx = \intt u(x,0)\,m_0\,dx 
$$
we get immediately \rife{posneg} if $c_T=0$.  If $c_T>0$, we estimate the term at $t=T$ in a similar way as in \rife{posneg1}, \rife{posneg2}, and we end up with 
\begin{align*}
\int_0^T \intt m\, |u_x|^2\, dxdt  & + \int_0^T \intt |m \log(m)|\, dxdt  + c_T \intt |m(T) \log(m(T))|\, dx \leq 
\\ & 
C+ \vep  \int_0^T \intt  m\,(|x|^2+1) dxdt + \vep \intt m(T) (|x|^2+1)dx
 + C_\vep (1+ T)  \intt \frac1{1+ |x|^2}\, dxdt \,.
\end{align*}
Using \rife{Gro} and choosing $\vep$ suitably small,   the  estimates \rife{en-bound}, \rife{momest} and \rife{wass-unif} follow as before. Notice that, if $c_T>0$, we also estimate $\cE(m(T))$ and then we are back to the previous case, obtaining \rife{pointt} as well. 
\end{proof}

\vskip1em
Now we show a local bound  on the value function $u$, which is independent  of the period $R$.

\begin{prop}\label{bounduloc} Let $(u,m)$ be a solution of    \rife{log-per},   and assume that $u,m$ are continuous in $\T\times [0,T]$. 

\begin{enumerate}

\item If $(u,m)$ solves \rife{log-per}--\rife{plan-log},  there exists a constant  $C>0$, depending   on $\|m_{0}\|_{\infty},\|m_{T}\|_{\infty}$, $M_2(m_0), M_2(m_T)$, but independent of $R$, such that, if we normalize $u$ such that $\int u(T)m_{T}=0$, then we have
\be\label{stimuloc}
- \frac C t (1+ |x|^2) \leq u(x,t) \leq  \frac{C}{T-t} (1+ |x|^2) \qquad \forall t\in (0,T)\,, x\in \R\,.
\ee
\item If  $(u,m)$ solves \rife{log-per}--\rife{mfg-log} (with $g(m)=c_T\log(m)$),  then the estimate \rife{stimuloc} holds  for some $C$ only depending on $\|m_{0}\|_{\infty}, M_2(m_0), c_T$.  
\end{enumerate}
\end{prop}

\begin{proof}  First we consider the case of \rife{log-per}--\rife{plan-log}, and we adapt a similar proof given in  \cite[Lemma 4.2]{Porretta} for bounded domains.  

We observe that, from the standard duality between $u,m$ we have
$$
\int_0^T \intt m\, |u_x|^2\, dxdt + \int_0^T \intt m \log(m)\, dxdt 
= \intt u(x,0)m_0 -  \intt u(x,T)m_T
$$
which implies, using Proposition \ref{energy-log} and  $\int u(T)m_{T}=0$, that
\be\label{uosot}
\intt u(x,0)m_0 \geq - K\,.
\ee
Notice that the constant depends on $m_0, m_T$ through the entropy, by Proposition \ref{energy-log}; however, for bounded functions,  $\cE(m)$ is itself estimated in terms of $M_2(m)$ and $\|m\|_\infty$.

Now, let us consider the Wasserstein geodesic $\mu(\cdot)$ which connects, in time $(0,t)$,  $m_0$ with  any measure $\lambda$, supposed to be compactly supported in $(-R/2, R/2)$. This means that $\mu (t)=\lambda$, $ \mu(0)=m_0$ and $\mu_s= (\mu v)$ in $ (0,t)$, for some $v$ such that $\int_0^t \intt |v|^2d\mu<\infty$.  By duality with the equation satisfied by $u$, we have
\begin{align*}
  -\intt u(x,t)d\lambda + \intt u(x,0)\,m_0\,dx  &  +  \frac12\int_0^t  \intt |u_x|^2\,\mu\, dxds   =\int_0^t  \intt \mu\, \log m\, dxds + \int_0^t   \intt \mu\, v\, u_x\, dxds
\\  \qquad &  \leq  \frac12\int_t^T \intt |u_x|^2\,\mu\, dxds+ \frac12 \int_0^t \intt \mu |v|^2\, dxds 
+ T \log(\|m\|_\infty \vee 1)
\end{align*} 
which yields
\begin{align*}
\intt u(x,t)d\lambda  & \geq  \intt u(x,0)\,m_0\,dx - \frac12 \int_0^t \intt \mu |v|^2\, dxds -  T \log(\|m\|_\infty \vee 1)
\\ &  \geq  -K - \frac c{t} \, W_2(\lambda, m_0)^2 -  T \log(\|m\|_\infty \vee 1) \,,
\end{align*} 
where we used  \rife{uosot} and the scaling  of Wasserstein geodesic. 
If we let $\lambda $ converge (in the weak-$*$ topology) towards a Dirac mass $\de_{x_0}$, we get
$$
u(x_0,t) \geq  - K - \frac c{t} \, \intr |x_0-y|^2 dm_0(y) - T \log(\|m\|_\infty \vee 1)\,.
$$
We recall that, by Corollary \ref{cor: m bound}, $\|m(t)\|_\infty$ is controlled by the initial-terminal values. 
Hence, there exists a constant $C$, depending on $\|m_0\|_\infty, \|m_T\|_\infty, T$ and $M_2(m_0)$, $M_2(m_T)$, such that
$$
u(x_0,t) \geq  -\frac C{t} (1+ |x_0|^2)\,.
$$
A similar argument (instead of \rife{uosot} we simply use that $\intt u(x,T)m_T=0$) shows  the upper bound $u(x_0,t) \leq   \frac C{T-t} (1+ |x_0|^2)$ and concludes the proof of \rife{stimuloc}.

In the case of \rife{mfg-log}, the duality equality takes the form
$$
\int_0^T \intt m\, |u_x|^2\, dxdt + \int_0^T \intt m \log(m)\, dxdt 
= \intt u(x,0)m_0 -  c_T\intt m(T) \log(m(T))dx 
$$
and  \rife{uosot}   follows again from   Proposition \ref{energy-log}.  We also recall that the $L^\infty$ bound on $m$ is given by Proposition \ref{prop: oscillation bound}. Then we obtain as before the lower estimate of $u$. For the upper estimate, we just observe that $\intt u(T)m(T)dx= c_T \intt m(T)\log(m(T)) dx$  and this is bounded above (uniformly with respect to $R$) if either $c_T=0$ or $c_T>0$ (from Proposition \ref{energy-log}). Hence we repeat the argument above using any geodesic connecting $m(T)$ with a Dirac mass.
\end{proof}

In the next step we aim at showing that, if $(u,m)$ is a solution of \rife{log-per},  then $m(t)$ becomes positive for $t>0$ even if starting from a compactly supported initial measure.  For this purpose, we use in a key way the displacement convexity estimates.
However, we warn the reader that, while the  estimates of Proposition \ref{energy-log} and \ref{bounduloc} were all independent of the period $R$, this will not be the case for the following bounds on $\log (m)$.

\begin{lem}
Let $(u,m)$ be a solution to \rife{log-per}, where $m_0, m_T>0$. Set $K=\max(\|m_{0}\|_{\infty},\|m_{T}\|_{\infty})$. Then, for each
integer $p\geq1$,

\begin{equation} 
\frac{d^{2}}{dt^{2}}\int_{\T}\left|\log\left(\frac{m}{K}\right)\right|^{p}\geq0.\label{logmK}
\end{equation}
Moreover, there exists a constant $C >0$, depending   on $\|m_{0}\|_{\infty},\|m_{T}\|_{\infty}$,
such that, for each $t\in[0,T]$,
\begin{equation}
\left\Vert \log\left(\frac{m(t)}{K}\right)\right\Vert _{\infty}^{2}\leq R\, \frac{d^{2}}{dt^{2}}\int_{\T}\left|\log\left(\frac{m(t)}{K}\right)\right|+C\label{logmK2}
\end{equation}
\end{lem}

\begin{proof}
Letting $h(m)=\log(\frac{m}{K})^{p}$, we obtain
\[
h'(m)=p\log\left(\frac{m}{K}\right)^{p-1}\frac{1}{m},\,\,\,h''(m)=p\left(p-1-\log\left(\frac{m}{K}\right)\right)\log\left(\frac{m}{K}\right)^{p-2}\frac{1}{m^{2}}.
\]
We observe that,  by Corollary \ref{cor: m bound} and by definition of $K$, we have  $\frac{m}{K}\leq1$. 
Hence, in the range of $m$, $h$ is positive and convex when $p$
is even, and negative and concave when $p$ is odd. The displacement
convexity formula \eqref{eq:displacement integral} yields
\[
\frac{d^{2}}{dt^{2}}\int_{\T}h(m)=\int_{\T}h''(m)(m^{2}u_{xx}^{2}+m_{x}^{2}),
\]
which, in particular, shows \eqref{logmK}, and, setting $p=1$, we obtain
\begin{align*}
\int_{\T} \left(\log\left( \frac{m}{K}\right)_{x}\right)^2\leq \frac{d^2}{dt^2}\int_{\T}\left| \log\left(\frac{m}{K}\right) \right|.
\end{align*}
On the other hand, by the fundamental theorem of calculus, and the
fact that $m$ is a density, we have
\[
\left\| \log\left(\frac{m(t)}{K}\right)-\log\left(\frac{1}{K}\right) \right\|_{\infty}^{2}\leq R\int_{\T}\left(\log\left( \frac{m}{K}\right)_{x}\right)^{2},
\]
and \eqref{logmK2} follows.
\end{proof}

We can now state the (local in time) uniform bound from below, which is independent of $\|m_{0}^{-1}\|_{\infty},\|m_{T}^{-1}\|_{\infty}$.

\begin{prop}\label{stimalog-inside}
Let $(u,m)$ be a solution to \rife{log-per}, where $m_0, m_T>0$.  There exists a constant $C_R>0$,
only depending on  $\|m_{0}\|_{\infty},\|m_{T}\|_{\infty}$ and $R$, 
such that, for each $t\in(0,T),$
\[
\|\log m(t)\|_{\infty}\leq C_R\left(\frac{1}{t^{2}}+\frac{1}{(T-t)^{2}}\right).
\]
\end{prop}

\begin{proof}
Let $t_{0}\in[0,\frac{T}{2}]$. In view of (\ref{logmK}), we have,
for integers $p\geq1$ and $s\in(0,\frac{t_{0}}{2})$,

\begin{multline*}
\left(\frac{1}{(T-2t_{0})}\int_{t_{0}}^{T-t_{0}}\int_{\T}\left|\log\left(\frac{m}{K}\right)\right|^{2p}\right)^{\frac{1}{p}}\leq\left(\int_{\T}\left|\log\left(\frac{m(t_{0}-s)}{K}\right)\right|^{2p}+\int_{\T}\left|\log\left(\frac{m(T-t_{0}+s)}{K}\right)\right|^{2p}\right)^{\frac{1}{p}}\\
\leq \left(\int_{\T}\left|\log\left(\frac{m(t_{0}-s)}{K}\right)\right|^{2p}\right)^{\frac{1}{p}}+\left(\int_{\T}\left|\log\left(\frac{m(T-t_{0}+s)}{K}\right)\right|^{2p}\right)^{\frac{1}{p}}.
\end{multline*}
Thus, integrating in $s$, we infer that

\[
\left(\int_{t_{0}}^{T-t_{0}}\int_{\T}\left|\log\left(\frac{m}{K}\right)\right|^{2p}\right)^{\frac{1}{p}}\leq\frac {2(T-2t_{0})^{\frac{1}{p}}}{t_{0}}\int_{\frac{t_{0}}{2}}^{T-\frac{t_{0}}{2}}\left(\int_{\T}\left|\log\left(\frac{m}{K}\right)\right|^{2p}\right)^{\frac{1}{p}},
\]
and letting $p\rightarrow\infty$ yields
\[
\left\Vert \log\left(\frac{m}{K}\right)\right\Vert _{L^{\infty}(\T \times [t_{0},T-t_{0}])}^{2}\leq\frac{2}{t_{0}}\int_{\frac{t_{0}}{2}}^{T-\frac{t_{0}}{2}}\left\Vert \log\left(\frac{m(t)}{K}\right)\right\Vert _{\infty}^{2}dt.
\]
Now, integrating (\ref{logmK2}) against a test function $\zeta$ supported
in $[\frac{t_{0}}{3},T-\frac{t_{0}}{3}]$, satisfying $0\leq\zeta\leq1$,
$\zeta\equiv1$ in $[\frac{t_{0}}{2},T-\frac{t_{0}}{2}]$, and $|\zeta''|\leq\frac{C}{t_{0}^{2}}$,
we get
\[
\left\Vert \log\left(\frac{m}{K}\right)\right\Vert _{L^{\infty}(\T \times [t_{0},T-t_{0}])}^2\leq\frac{C\, R}{t_{0}^{3}}  \int_{\frac{t_{0}}{3}}^{T-\frac{t_{0}}{3}}\int_{\T}\left|\log\left(\frac{m}{K}\right)\right|dt+\frac{C}{t_{0}}.
\]
Finally, by Proposition  \ref{bounduloc}, $u$ is bounded by $\frac{C(1+ R^2)}{t_{0}}$ on $[t_0,T-t_0]$,
and, hence, by the HJ equation, we estimate 
$$
\int_{\frac{t_{0}}{3}}^{T-\frac{t_{0}}{3}}\int_{\T} \left|\log\left(\frac{m}{K}\right)\right|\leq \frac{C  R^3}{t_{0}}\,.
$$
This yields
\[
\left\Vert \log\left(\frac{m}{K}\right)\right\Vert _{L^{\infty}(\T \times [t_{0},T-t_{0}])}\leq\frac{C(1+R^2)}{t_{0}^{2}},
\]
which implies the result.
\end{proof}

Finally, we have all the ingredients for the proof of Theorem \ref{ex-log-torus}.

{\bf Proof of Theorem \ref{ex-log-torus}.} \quad 
We start with the case of problem \rife{log-per}--\rife{plan-log}. Let $m_0^\vep, m_T^\vep$ be two sequences of functions such that $m_0^\vep, m_T^\vep \in C^{1,\alpha}(\T)$, $m_0^\vep, m_T^\vep>0$ in $\T \times [0,T]$ and $m_0^\vep, m_T^\vep$  converge uniformly to  $m_0, m_T$ respectively.  Such an approximation can be readily built by convolution, for instance. By Theorem \ref{thm:existence}, there exists a smooth positive solution $(u^\vep, m^\vep)$ of \rife{log-per}, where we normalize $u^\vep$ such that
$$
\intt u^\vep(T)\,m_T^\vep dx=0\,.
$$
By Corollary \ref{cor: m bound}, we know that $m^\vep$ is uniformly bounded. Then, from Proposition \ref{bounduloc}, we deduce that $u^\vep$ is locally bounded in $(0,T)$. It also follows from Proposition \ref{stimalog-inside} that $\log(m^\vep)$ is locally bounded in $(0,T)$ (i.e. $m^\vep$ is locally uniformly bounded below). In turn, this implies that $u_x$ is locally bounded in $(0,T)$; one can use for example  \cite[Thm 6.5]{Porretta} which shows\footnote{The proof in \cite{Porretta} is given for Neumann boundary conditions, but applies identically to periodic solutions} that $\frac{|u_x^\vep|^2}4+ \log(m^\vep) \leq C_\de$ for every $x\in \T, t\in (a+ \de, b-\de)$, where $C_\de$ only depends on $\de$ and  the bound on $u^\vep$   in $(a,b)$. Since $u^\vep$ satisfies 
\be\label{eq}
-u_{tt}+2u_{x}u_{xt}-(u_{x}^{2}+ 1)u_{xx}=0 \,,
\ee
once $(u^\vep)_x$ is locally bounded  then the above equation becomes a quasilinear equation which has    bounded uniformly elliptic coefficients in any compact subset of $\T \times (0,T)$. By a standard bootstrap regularity from Schauder's estimates, we deduce that $u^\vep$ is locally bounded in $C^{k,\alpha}$ (for every $k\in N, \alpha\in (0,1)$). Hence, $u^\vep$ converges in $\T \times (0,T)$ towards  a function $u$ which is $C^\infty$. Since $m^\vep= \exp( -(u^\vep)_t + |(u^\vep)_x|^2/2)$, we also have $m^\vep$ converging to some $m\in C^\infty (\T \times (0,T))$. But the global estimates also imply that $m\in L^\infty(Q_T)$. In addition, by \rife{wass-unif}, we also have that 
$m^\vep(t)$ is equi-continuous in the Wasserstein space of measures, hence it uniformly converges in $[0,T]$. We deduce that $m\in C^0([0,T]; \cP(\T))$ and $m(0)=m_0, m(T)=m_T$. This concludes the proof that $(u,m)$ is a solution of \rife{log-per}, which is classical inside $(0,T)$.

In case of problem \rife{log-per}--\rife{mfg-log}, the proof is similar, except that we only approximate $m_0$. The $L^\infty$ bound on $m^\vep$ follows from Proposition \ref{prop: oscillation bound}, then we argue as before to deduce that $m^\vep$ is locally uniformly bounded below, and $u^\vep, m^\vep $ are locally bounded in $C^{k,\alpha}$. Applying Propositions  \ref{prop: oscillation bound} and \ref{prop:gradient bound} to  problem \rife{log-per}--\rife{mfg-log}   for $t\in (T/2,T)$, we conclude that $u_\vep, Du^\vep $ are uniformly bounded up to $t=T$, and $m^\vep$ is bounded below up to $t=T$.  By regularity of equation \rife{eq} up to the boundary $t=T$, we conclude that $u,m$ are smooth up to $t=T$ and $u(T)= c_T \log (m(T))$.

For the uniqueness of solutions, we use some argument which was already developed   for much weaker notions of  solutions (\cite{Carda1}, \cite{CP-CIME}, \cite{OPS}). Let $(u,m)$ be a solution which is classical inside, as in Definition \ref{sollog}, and such that $m  \in L^\infty(\T \times (0,T))$. First we notice that, from \rife{preserve}, the bound on $m$ implies that $mu_x^2\in L^1(\T\times (0,T))$. Next we  observe that $u$ satisfies 
$-u_{t}+\frac{1}{2}u_{x}^{2} \leq \log(\|m\|_\infty)$,  so (up to a time translation) we can assume that $u(\cdot, x)$ is nondecreasing.  this implies  that $u(t)$ admits one-sided traces at $t=0$, $t=T$, namely two measurable functions (not necessarily finite) defined as    $u(x,0^+)= \lim\limits_{t\downarrow 0} u(x,t)\in \R\cup\{-\infty\}$, $u(x,T^-)= \lim\limits_{t\uparrow 0} u(x,t)\in \R\cup\{+\infty\}$. We first show that $u(0^+)\in L^1(dm_0)$; 
indeed, since $(u,m)$ is smooth in $\T \times (0,T)$, we know that 
\be\label{t0t1}
\intt u(t_0)m(t_0) - \intt u(t_1)m(t_1) = \int_{t_0}^{t_1}  \intt mu_x^2 +\int_{t_0}^{t_1}\intt m\log(m), \qquad\forall \, 0<t_0<t_1<T\,.
\ee
Choosing, for instance, $t_1= \frac T2$, it follows that $\intt u(t_0)m(t_0)$ is bounded below, for every arbitrarily small $t_0$.  Since $u$ is nondecreasing, we deduce  $\intt u(s)m(t_0) \geq - C$ for any $s>t_0$; letting  $t_0\to 0$ (and using that $m\in C([0,T]; \cP(\T))$) yields  $\intt u(s)m_0 \geq - C$. Then by monotone convergence we deduce, as $s\to 0^+$,  that $\intt u(0^+)m_0 \geq - C$. Since the opposite inequality is clear by monotonicity and Proposition \ref{bounduloc}, we find that $u(0^+)\in L^1(dm_0)$. 
Similarly we reason  to show that $u(T^-)\in L^1(dm_T)$ (in case of  problem \rife{plan-log}) or that $m(T)\log(m(T))\in L^1(\T)$ (in case of \rife{mfg-log} with $c_H>0$).  

Now, with a  truncation argument, from the equality  \rife{t0t1} we will show that $(u,m)$ satisfies
\be\label{ide}
 \intt u(0^+)dm_0 - \intt u(T^-)dm(T) = \int_0^T \intt mu_x^2 + \int_0^T\intt m\log(m) \,.
 \ee
Indeed, we  first  replace $u$ by truncations $u_k:= \min(k, \max(u,-k))$,  we multiply the HJ equation by $m$ and  we integrate in $(t_0, t_1)$; next we can let $t_0\to 0^+, t_1 \to T^-$ (using the weak-$^*$ convergence of $m$  and the strong $L^1$ convergence of $u_k$ at $t=0,t=T$). Then we finally let $k\to \infty$ (thanks to $u(0^+)\in L^1(dm_0), u(T^-) \in L^1(dm(T))$) and we obtain \rife{ide}. 

In a similar way, one can prove that, for any couple of  solutions $(u,m)$, $(\tilde u, \tilde m)$, it holds
 \be\label{ideh}
 \intt u(0^+)dm_0 - \intt u(T^-)d\tilde m(T) \leq  \int_0^T \intt \tilde m\left(   \tilde u_x \, u_x-\frac12 u_x^2\right)  + \int_0^T\intt \tilde m\log(m)\,.
\ee
The proof of \rife{ideh} can be done, as before, replacing first $u$ with its truncation $u_k$ and integrating in $(t_0, t_1)$:
$$
\intt u_k(t_0)\tilde m(t_0) - \intt u_k(t_1)\tilde m(t_1) \leq  \int_{t_0}^{t_1} 
\intt \tilde m \left((u_k)_x\tilde u_x -\frac12(u_k)_x^2\right)+ \int_{t_0}^{t_1} \intt \tilde m\log(m)\mathop{1}\nolimits_{\{|u|<k\}}\,.
$$
The right-hand side integrand is easily dominated from above. Once more, we can let first $t_0 \to 0, t_1\to T$ and then $k\to \infty$, in order to get \rife{ideh}. 

From \rife{ide}, \rife{ideh}, the uniqueness follows as in the classical Lasry-Lions monotonicity argument.  For problem \rife{log-per}--\rife{plan-log}, we take   two solutions normalized such that $ \intt u(T^-)dm_T=\intt \tilde u(T^-)dm_T=0$. Using \rife{ide}, \rife{ideh} for both couples we obtain that
\begin{align*}
  & \int_0^T \intt \tilde m \left(\frac12  u_x^2 - \frac12 \tilde u_x^2 - (u_x - \tilde u_x ) \, \tilde u_x \right) + \int_0^T \intt   m \left(\frac12  \tilde u_x^2 - \frac12   u_x^2 - (\tilde u_x -  u_x ) \,  u_x \right) 
 \\ &  + \int_0^T\intt  ( m\log(m) - \tilde m \log(\tilde m) ) (m-\tilde m) \leq 0,
\end{align*}
which implies $m= \tilde m$ and $u_x= \tilde u_x$. From the HJ equation we deduce that $u-\tilde u= C$, and this concludes the proof. For the problem \rife{log-per}--\rife{mfg-log}, we proceed similarly and we get uniqueness using the coupled condition $u(T)= g(m(T))$.
\qed

%
 

\subsection{Preservation of monotonicity of the solutions}

In this subsection, we show that the MFG system preserves a certain monotonicity property. As the phenomenon does not depend on the specific form of the coupling functions $f$ and $g$, we suppose here that $f$ and $g$ are smooth and nondecreasing on $(0,\infty)$. We work in the periodic setting and assume the structure condition:
\begin{equation}\label{hyp.symm}
\begin{array}{c}
\text{the densities $m_0, m_T: \T\to \R$ are  even,}\\
\text{ nonincreasing on $[0,R/2]$ and nondecreasing on $[-R/2,0]$.}
\end{array}
\end{equation}

\begin{lem}\label{lem.convexity}  Assume that $(u,m)\in C^2(\T\times [0,T])\times C^1( (\T\times [0,T]))$ is the unique classical solution to \eqref{eq:Per MFG}--\rife{couplingT}  or \eqref{eq:Per MFG}--\rife{per-MFGP}, such that $m$ is positive on $\T\times [0,T]$ and  \eqref{hyp.symm} holds. Then 
\begin{equation}\label{PptExtram}
\forall t\in [0,T], \; m(\cdot,t):\T\to \R\;\text{is even, nonincreasing on $[0,R/2]$ and nondecreasing on $[-R/2,0]$.}
\end{equation}
In addition any optimal trajectory $\gamma(x,\cdot)$ starting from $x\in [0,R/2]$ is concave in time. Finally, for the MFG problem \eqref{eq:Per MFG}--\rife{couplingT}, $\gamma(x,\cdot)$  is nondecreasing in time for any $x\in [0,R/2]$ and $u_x$ is nonpositive in $[0,R/2]\times [0,T]$.
\end{lem}

Of course, par approximation,  this preservation of the structure also holds in the whole space: in the next subsection, we shall use it in the case $f(m)=\log(m)$ to build classical solutions in the whole space. In the case $f(m)=m^\theta$, it shows that, if $m_0,m_T$ are even on $\R$ and nonincreasing on $[0,+\infty)$, then any trajectory starting from $x\in [0,R/2]$ is concave in time: compare with Theorem \ref{thm:free boundary regularity}.  

\begin{proof} We do the proof in the  MFG case, i.e., when $(u,m)$ solves \eqref{eq:Per MFG}--\eqref{couplingT}, with the proof for the planning problem \eqref{eq:Per MFG}--\eqref{per-MFGP} being similar and simpler. By the symmetry assumption and the uniqueness of the solution, $m(\cdot,t), u(\cdot, t) $ are even  for any $t\in [0,T]$. Thus $m_x(0, t)=u_x(0,t)=0$. Let us set 
$$
M(x,t)= 1/2-\int_0^x m(y,t)dy, \qquad  (x,t)\in [0,R/2]\times [0,T].
$$
We first note that $M$ is a classical solution to 
\begin{equation} \label{eq.MM}
\displaystyle - {\rm Tr} \left( \left( \begin{array}{cc}
\frac{M_t^2}{M_x^2}  & -\frac{M_t}{M_x}\\  -\frac{M_t}{M_x} & 1
\end{array}\right) D^2_{x,t}M\right) +M_xf'(-M_x)M_{xx}=0  \qquad \text{in} \; [0,R/2]\times (0,T)\\
\end{equation}
with boundary condition, for  $(x,t)\in [0,R/2]\times [0,T]$, 
\begin{equation} \label{eq.Mbc1}
M(0,t)= 1/2, \; M(R/2,t)=0,\; M(x,0)=M_0(x), \qquad \text{where}\; M_0(x)=1/2- \int_0^x m_0(y)dy,
\end{equation}
and 
\begin{equation} \label{eq.Mbc2}
M_t(x,T)+M_x(x,T) g'(-M_x(x,T))M_{xx}(x,T)=0.
\end{equation}
The elliptic equation \eqref{eq.MM} was proved in Lemma \ref{lem.eqMMBIS} (where we also explained that $u_x=M_t/M_x$). 
The boundary conditions \eqref{eq.Mbc1} at $x=0$ and $t=0$ hold by definition. For $x=R/2$, it comes from the fact that $m$ is a probability measure and from the symmetry. The boundary condition  \eqref{eq.Mbc2} at $t=T$ comes from the boundary condition for $u$, which implies that 
$$
(M_t/M_x)(x,T)=u_x(x,T) = (g'(m)m_x)(x,T)= g'(-M_x(x,T))(-M_{xx}(x,T)).
$$

The main part of the proof consists in showing that $x\to M(x,t)$ is convex on $[0,R/2]$ for any $t\in [0,T]$. For this we consider the map 
$$
\tilde M(x,t)= \inf_{\begin{array}{c} \lambda y+(1-\lambda)z=x, \\ y,z\in [0,R/2], \; \lambda\in [0,1]\end{array}} \lambda M(y,t)+(1-\lambda)M(z,t). 
$$
Note that $\tilde M\leq M$ and  that $\tilde M$ is continuous and satisfies the boundary condition \eqref{eq.Mbc1} by our assumption on $m_0$. We now prove that $\tilde M$ is a viscosity supersolution to the elliptic equation \eqref{eq.MM} and satisfies the boundary condition \eqref{eq.Mbc2} in the viscosity sense.

Assume that $\phi$ is a test function touching $\tilde M$ from below at $(x_0,t_0)\in (0,R/2)\times (0,T]$. If $t_0<T$, we have to check that $\phi_x(x_0,t_0)\neq 0$ and that 
\begin{equation}\label{kklaim}
\displaystyle - {\rm Tr} \left( \left( \begin{array}{cc}
\frac{\phi_t^2}{\phi_x^2}  & -\frac{\phi_t}{\phi_x}\\  -\frac{\phi_t}{\phi_x} & 1
\end{array}\right) D^2_{x,t}\phi\right) +\phi_xf'(-\phi_x)\phi_{xx}\geq 0  \qquad \text{at} \; (x_0,t_0).
\end{equation}
If $t_0=T$, we have to prove that
\begin{equation}\label{kklaimBC} 
\phi_t(x_0,T)+ \phi_x(x_0,T) g'(-\phi_x(x_0,T))\phi_{xx}(x_0,T)\geq 0.
\ee
Note that, if $\tilde M(x_0,t_0)=M(x_0,t_0)$,  these inequalities hold because $M$ satisfies \eqref{eq.MM} and \eqref{eq.Mbc2}. Thus, we assume from now on that $\tilde M(x_0,t_0)<M(x_0,t_0)$. Let $y_0<x_0<z_0$ and $\lambda_0\in (0,1)$ be optimal in the definition of $\tilde M(x_0,t_0)$. \\

{\it In this step we assume that $t_0<T$.} By optimality of $(y_0,z_0,\lambda_0)$,  and the fact that, by symmetry, $M_{xx}(0,t_0)=m_x(0,t_0)=M_{xx}(R/2,t_0)=m_x(R/2,t_0)=0$, we have that
\begin{equation}\label{eziurledkgf}
\begin{array}{c}
 \phi_x(x_0,t_0)=\frac{M(z_0,t_0)-M(y_0,t_0)}{z_0-y_0}<0 \\
 \phi_x(x_0,t_0)= M_x(y_0,t_0) \;\; \text{or} \;\; y_0=0,  \quad \phi_x(x_0,t_0)=M_x(z_0,t_0) \;\; \text{or} \;\; z_0=R/2 , \\
\phi_{xx}(x_0,t_0)\leq 0, \; M_{xx}(y_0,t_0)\geq 0, \; M_{xx}(z_0,t_0)\geq 0. 
\end{array}
\end{equation}
Fix $\theta,\theta_1,\theta_2\in \R$ such that $\lambda_0\theta_1+(1-\lambda_0)\theta_2=\theta$, with $\theta_1=0$ if $y_0=0$ and $\theta_2=0$ if $z_0=R/2$. For $h$ and $s$ small, we have 
$$
\phi(x_0+h\theta,t_0+s) \leq \tilde M(x_0+h,t_0+s) \leq \lambda_0M(y_0+\theta_1h,t_0+s)+(1-\lambda_0) M(z_0+\theta_2 h, t_0+s),
$$
with an equality at $h=s=0$. This implies that 
\begin{equation}\label{iauezlinrsdfg}
\phi_t(x_0,t_0)= \lambda_0 M_t(y_0,t_0)+(1-\lambda_0) M_t(z_0,t_0),
\end{equation}
and
$$ \left( \begin{array}{cc} \theta^2\phi_{xx}& \theta \phi_{xt}\\ \theta\phi_{xt}& \phi_{tt}\end{array}\right)(x_0,t_0)\leq \lambda_0 \left( \begin{array}{cc} \theta_1^2 M_{xx}& \theta_1 M_{xt}\\ \theta_2 M_{xt}& M_{tt}\end{array}\right)(y_0,t_0)+
(1-\lambda_0)\left( \begin{array}{cc} \theta_2^2 M_{xx}& \theta_2 M_x\\ \theta_2 M_x& M_{tt}\end{array}\right)(z_0,t_0).
$$
Multiplying  the previous inequality by $\left(\begin{array}{cc}1&-1\\-1&1\end{array}\right)\geq0$ and taking the trace gives 
\begin{equation}\label{iauezlinrsdfg2}
\begin{array}{l}\Bigl(\theta^2\phi_{xx} -2 \theta \phi_{xt}+\phi_{tt}\Bigr)(x_0,t_0)\\
\qquad \leq  \lambda_0\Bigl( \theta_1^2 M_{xx}   -2\theta_1M_{xt} +M_{tt}\Bigr)(y_0,t_0)+
(1-\lambda_0)\Bigl( \theta_2^2 M_{xx}   -2\theta_2 M_{xt} +M_{tt}\Bigr)(z_0,t_0) . 
\end{array}
\end{equation}
Let us choose $\theta=\phi_t(x_0,t_0)/\phi_x(x_0,t_0)$,  $\theta_1=M_t(y_0,t_0)/M_x(y_0,t_0)$ and $\theta_2 = M_t(z_0,t_0))/M_x(z_0,t_0)$: this choice is licit because, if $y_0=0$, the boundary conditions \eqref{eq.Mbc1} imply that $M_t(0,t_0)=0$, and thus $\theta_1=0$. We obtain in the same way that $\theta_2=0$ if $z_0=R/2$.
Then \eqref{eziurledkgf} and \eqref{iauezlinrsdfg} imply that  $\lambda_0\theta_1+(1-\lambda_0)\theta_2=\theta$ holds. With this choice of $\theta$, $\theta_1$ and $\theta_2$, \eqref{iauezlinrsdfg2} becomes 
\begin{align*}
&\Bigl(\frac{\phi_t^2}{\phi_x^2}\phi_{xx} -2 \frac{\phi_t}{\phi_x}\phi_{xt}+\phi_{tt}\Bigr)(x_0,t_0)\\
&\qquad \leq  \lambda_0\Bigl(\frac{M_t^2}{M_x^2} M_{xx}   -2\frac{M_t^2}{M_x^2} M_{xt} +M_{tt}\Bigr)(y_0,t_0)+
(1-\lambda_0)\Bigl(\frac{M_t^2}{M_x^2} M_{xx}   -2\frac{M_t^2}{M_x^2} M_{xt} +M_{tt}\Bigr)(z_0,t_0)  
\\
&\qquad  =  \lambda_0 (M_xf'(-M_x)M_{xx})(y_0,t_0) +(1-\lambda_0)(M_xf'(-M_x)M_{xx})(z_0,t_0),
\end{align*}
where we used the equation satisfied by $M$ for the last equality. Recalling that $M_x<0$, that $f'\geq0$ and that $M_{xx}(y_0,t_0)\geq 0$ and $M_{xx}(z_0,t_0)\geq 0$ while $\phi_{xx}(x_0,t_0)\leq 0$ (by \eqref{eziurledkgf}) gives \eqref{kklaim}. \\

{\it We now assume that $t_0=T$} and check the boundary condition \eqref{kklaimBC}. To fix the ideas, we assume that $y_0>0$ and $z_0<R/2$, the other cases being similar. Note that \eqref{eziurledkgf} still holds in this case. Moreover, as, for any $s\leq 0$  small we have 
$$
\phi(x_0,T+s) \leq \tilde M(x_0,T+s) \leq \lambda_0M(y_0,T+s)+(1-\lambda_0) M(z_0,T+s),
$$
we get 
$$
\phi_t(x_0,T)\geq \lambda_0M_t(y_0,T)+(1-\lambda_0) M_t(z_0,T).
$$
Thus, as  $\phi_x(x_0,T)<0$, $g'\geq0$ and $\phi_{xx}(x_0,T) \leq 0$ while $M_{xx}(y_0,T)\geq 0$ and $M_{xx}(z_0,T)\geq0$, 
\begin{align*}
& \phi_t(x_0,T)+ \phi_x(x_0,T) g'(-\phi_x(x_0,T))\phi_{xx}(x_0,T) \geq \lambda_0M_t(y_0,T)+(1-\lambda_0) M_t(z_0,T) \\ 
&\qquad \geq    \lambda_0(M_t(y_0,T)+ M_x(y_0,T) g'(-M_x(y_0,T))M_{xx}(y_0,T))\\
& \qquad \qquad +(1-\lambda_0) (M_t(z_0,T)+ M_x(z_0,T) g'(-M_x(z_0,T))M_{xx}(z_0,T)) \ = \ 0,
\end{align*}
which is \eqref{kklaimBC}.\\

{\it Conclusion.} We have proved that $\tilde M$ is a viscosity supersolution to the elliptic equation satisfied by $M$ (including the boundary conditions). Using the regularity of $M$, we can choose $\lambda>0$ (large) and $\epsilon>0$ (arbitrarily small) such that the map $\hat M_{ \epsilon,\lambda}(x,t)= M(x,t)-2\epsilon +\epsilon\exp\{-\lambda t\}$ is a classical strict subsolution of this equation (including the boundary conditions). This implies that $\hat M_{ \epsilon,\lambda}\leq \tilde M$. Then, letting $\epsilon\to0$, we get $M\leq \tilde M$. As by construction, $\tilde M\leq M$, we conclude that $M=\tilde M$; hence $M$ is convex with respect to the $x$ variable, and then  $m=-M_x$ is nonincreasing with respect to the $x$ variable on $[0,R/2]\times [0,T]$. Fix now $x\in [0,R/2]$ and let $\gamma(x,\cdot)$ be optimal solution starting from $(x,0)$, i.e. the solution to 
$$
 \gamma_t = -u_x(\gamma, t), \qquad \gamma(x,0)= x. 
 $$
By symmetry and periodicity, $u_x(0,t)= u_x(R/2,t)=0$ for  $t\in [0,T]$. Therefore $\gamma(x,t)\in [0,R/2]$ for $(x,t)\in [0,R/2]\times [0,T]$ and, in view of the Euler equation \eqref{eq.EulerBIS} and the monotonicity of $m$, 
 $\gamma$ is concave in $t$.  Differentiating the terminal condition \eqref{couplingT} with respect to the space variable implies that 
$$
 \gamma_t(x,T)= -g'(m(\gamma(x,T),T))m_x(\gamma(x,T),T)\geq0.
$$ 
As $\gamma(x,\cdot)$ is concave, we infer that $ \gamma_t(x, \cdot)$ is nonnegative on $[0,T]$. As $\gamma_t = -u_x(\gamma)$ and $x\mapsto \gamma(x,t)$ is onto from $[0,R/2]$ to itself, this implies that $u_x$ is nonpositive. 
 \end{proof}


 \subsection{Solutions in the whole space} 
 
We now work in the whole space, returning to the entropic coupling function $f(m)=\log(m)$ (and $g(m)=c_T \log (m(T))$). Our main result is the existence of a classical solution to \rife{eq:MFGR} or to \rife{eq:MFGPR}  under the structure condition \eqref{hyp.symm}. We adapt Definition \ref{sollog} to the case that $x$ belongs to the whole space. In what follows, $\cP_2(\R)$ will denote the set of Borel probability measures on $\R$ with a finite second order moment, equipped with the $2$--Wasserstein distance.

\begin{defn}\label{sollog2} We say that $(u,m)$ is a (classical) solution to  \rife{eq:MFGR}   if $(u,m)  \in C^2(\R\times (0,T))\times C^1(\R\times (0,T)),$ with $m>0$  in $\R\times (0,T]$, if  $m\in C([0,T];\cP_2(\R))$ with $m(0)=m_0$, if the equations are satisfied in the classical sense for $t\in (0,T)$ and, finally, if  $m(T)\in  L^1(\R)$ and $\lim_{t\to T^-} u(x,t)= g(m(x,T))$ for every $x\in \R$.

Similarly, we say that $(u,m)$ is a (classical) solution to   \rife{eq:MFGPR}  if $(u,m)  \in C^2(\R\times (0,T))\times C^1(\R\times (0,T)),$ with $m>0$  in $\R\times (0,T)$, if  $m\in C([0,T];\cP_2(\R))$ with $m(0)=m_0$ and $m(T)=m_T$, and if the equations are satisfied in the classical sense for $t\in (0,T)$.
\end{defn}
Let us notice that, whenever \eqref{hyp.symm} holds, in view of the preservation of symmetry property of Lemma \ref{lem.convexity}, the solutions to the MFG system with periodic and Neumann boundary conditions coincide. For this reason, we will not require the analysis of Subsection \ref{subsubsec:Neumann} in this case.

\begin{thm}\label{thm.caselogRsym} Assume that  $f(m)=\log(m)$. 
\begin{enumerate}
\item  Assume that  $m_0,m_T$ are continuous, compactly supported densities on $\R$, even, nonincreasing on $[0,\infty)$, with $m_0\in C^{1,\alpha}_{\emph{loc}}(\{m_0>0\})$. Then there exists a unique (up to addition of a constant to $u$) solution $(u,m)\in C^2(\R\times (0,T))\times C^1(\R\times (0,T))$ of \rife{eq:MFGPR}  such that  $m$ is continuous   and bounded on  $\R\times [0,T]$ and $\frac {u(t)}{(1+ |x|^2)}\in L^\infty(\R)$, for every $t\in (0,T)$. 

\item  Assume that $m_0$ is a continuous, compactly supported density on $\R$, even, nonincreasing on $[0,\infty)$, with with $m_0\in C^{1,\alpha}_{\emph{loc}}(\{m_0>0\})$, and $g(s)= c_T \log (s)$, for some $c_T\geq 0$.  Then there exists a unique  solution $(u,m)\in C^2(\R\times (0,T])\times C^1(\R\times (0,T])$ of \rife{eq:MFGR}  such that  $m$ is continuous   and bounded on  $\R\times [0,T)$ and $\frac {u(t)}{(1+ |x|^2)}\in L^\infty(\R)$, for every $t\in (0,T)$. 

\end{enumerate}
\end{thm}

Let us start with a priori estimates for positive periodic solutions: 

\begin{lem}\label{lem.esticaselogRsym} Suppose that \eqref{hyp.symm} holds and that $(u,m)\in C^2(\T\times [0,T])\times C^1(\T\times [0,T])$  is a classical solution  to \eqref{log-per}--\eqref{plan-log} or to \eqref{log-per}--\eqref{mfg-log} on $\T\times (0,T)$ with $m$ positive in $\T\times [0,T]$.  Let $\gamma :\T \times [0,T] \to \T$ be the associated flow of optimal trajectories.
\begin{enumerate}
\item (Global estimates) There exists $C_0>0$ depending only on $T$,   $\cE(m_0)$, $\cE(m_T)$, $M_2(m_0)$ (and $M_2(m_T)$ for the planning problem), such that: 
\begin{equation}\label{defedfC0}
\sup_{t\in [0,T]} \int_{\T} x^2m(x,t)dx\leq C_0,\qquad W_2(m(t),m(s)) \leq  C_0|t-s|^{1/2}\qquad \forall s, t\in [0,T].
\end{equation}
and
\begin{equation}\label{defedfC0bis}
-\frac{C_0}{t}(x^2+1)\leq  u(x,t) \leq \frac{C_0}{(T-t)}(x^2+1).
\end{equation}
Moreover, for any  $(x,t)\in (-R/2,R/2)\times [0,T]$, 
\be\label{boundabovemm}
m(x,t) \leq \left\{\begin{array}{ll}
\max\{\|m_0\|_\infty, \|m_T\|_\infty\} & \text{if $(u,m)$ satisfies  \eqref{log-per}--\eqref{plan-log}}\\
\|m_0\|_\infty & \text{if $(u,m)$ satisfies  \eqref{log-per}--\eqref{mfg-log}} \end{array}\right. 
\ee
and
\begin{equation}\label{lizajqesdfg}
( 8C_0)^{-1/2} \left( \int_{|x|}^{R/2} m_0(y)dy\right)^{3/2}\leq m(\gamma(x,t),t). 
\end{equation}

\item (Interior estimates)  Fix any $\delta\in (0,(1/2)\wedge (T/4))$ and $a\in (1,(R/2-1))$, with $R>4$. For any  $\eta\in (0,R/2)$ and $\theta\in (\eta, R/2)$  such that 
\begin{equation}\label{condcondtheta}
  \delta^{3/2} ( 8C_0)^{-1/2}\left| \log \left( \int_{\theta-\eta}^{R/2} m_0(y)dy\right)\right|^{1/2}  > 2a, 
\end{equation}
one has 
\be\label{boundbelowgammaBIS}
\min_{t\in [\delta, T-\delta]}\gamma(\theta-\eta,t)>a
\ee
and
$$
\|m, 1/m\|_{C^{2,\alpha}((-a,a)\times (\delta,T-\delta))} \leq C( \eta^{-1},   \delta^{-1}, \|m_0\|_{\infty},  \|m_T\|_\infty, K_\theta, C_0), 
$$
where  
\begin{equation}\label{defKtheta}
K_\theta:= \|m_0^{-1}\|_{L^\infty((-\theta,\theta))} + \|m_0\|_{C^{1,\alpha}((-\theta,\theta))}+ \left( \int_{\theta}^{R/2} m_0(y)dy\right)^{-1}.
\end{equation}
\end{enumerate}
\end{lem}

\begin{proof} Estimates \eqref{defedfC0} and \eqref{defedfC0bis} are given in Proposition \ref{energy-log}. By Lemma \ref{lem.convexity}  the solution $(u,m)$ satisfies \eqref{PptExtram}, and $x\to \gamma(x,t)$ is increasing on $[0,R/2]$, with $\gamma(0,t)=0$ and $\gamma(R/2,t)=R/2$ for any $t\in [0,T]$. \\

\noindent{\it Step 1: Bounds on the density.} 
The upper bounds on $m$ in \eqref{boundabovemm} hold by Proposition \ref{prop: oscillation bound}. 
Let us now prove the lower bound \eqref{lizajqesdfg}. Let $x\in [0,R/2)$ and $k=(2C_0)^{1/2} (\int_x^{R/2} m_0(y)dy)^{-1/2}$. We first assume that $k<R/2$. Then, by \eqref{PptExtram} (for the second inequality) and \eqref{defedfC0} and the choice of $k$ (for the last one), we have 
\begin{align*}
\int_x^{R/2} m_0(y)dy&= \int_{\gamma(x,t)}^{R/2} m(y,t)dy \leq \int_{\gamma(x,t)\wedge k}^{k} m(y,t)dy+ \int_{k}^{R/2} m(y,t)dy \\
&\leq m(k\wedge\gamma(x,t),t) (k-k\wedge \gamma(x,t)) + k^{-2} \int_k^{R/2} x^2m(y,t)dy \\
& \leq  m(k\wedge\gamma(x,t),t) (k-k\wedge \gamma(x,t)) +\frac12  \int_x^{R/2} m_0(y)dy . 
\end{align*}
This implies that $\gamma(x,t)< k$ and 
\eqref{lizajqesdfg} in this case. Next we suppose that $k\geq R/2$. Then the same computation shows that 
$$
m(\gamma(x,t),t)\geq (2/R) \int_x^{R/2} m_0(y)dy\geq (2C_0)^{ -1/2} \left(\int_x^{R/2} m_0(y)dy\right)^{3/2},
$$
where the last inequality holds because $k\geq R/2$. Thus we have proved \eqref{lizajqesdfg} for $x\in [0,R/2)$. The result for negative $x$ holds by symmetry. \\

\noindent{\it Step 2: Elliptic estimates.} We now prove $C^{2,\alpha}$ estimates for $\gamma_x$ and  $w=\log(m(\gamma))$. Recalling from \rife{masspreservationBIS} that $\gamma_x(x,t)= m_0(x)/m(\gamma(x,t),t)$, we note that \eqref{boundabovemm} and \eqref{lizajqesdfg} imply that $\gamma_x$ is locally bounded above and below: 
\begin{equation}\label{boundgammax}
\frac{m_0(x)}{\max\{\|m_0\|_\infty, \|m_T\|_\infty\}} \leq  \gamma_x(x,t)\leq \frac{( 8C_0)^{1/2} m_0(x)}{\left( \int_{|x|}^{R/2} m_0(y)dy\right)^{3/2}}\qquad  \forall (x,t)\in (-R/2,R/2)\times [0,T].
\end{equation}
Let $w(x,t)=\log(m(\gamma(x,t),t))$. Then $w$ solves the elliptic equation in divergence form (see \rife{eq.v=fmgammalog}):
\begin{equation}\label{eqloggammam}
-(\gamma_x w_t)_t -\left(\frac{1}{\gamma_x}w_x\right)_x = 0 \quad \text{in}\; (-R/2,R/2)\times (0,T).
\end{equation}
Fix $\eta\in (0,R/2)$ and $\delta\in (0,T/4)$. Let $\theta\in (\eta ,R/2)$ and $K_\theta$ be defined by \eqref{defKtheta}. As, by \eqref{boundabovemm}, \eqref{lizajqesdfg} and \eqref{boundgammax}, 
$$
|w|+|\gamma_x|+|1/\gamma_x| \leq C(K_\theta, \|m_0\|_\infty, \|m_T\|_\infty, C_0)\qquad \text{on}\; [-\theta,\theta]\times [0,T], 
$$ 
we infer by elliptic regularity that
$$
\|w\|_{C^{0,\alpha}([-\theta+\eta/3,\theta-\eta/3]\times [\delta/3,T-\delta/3])}\leq C(\eta^{-1}, \delta^{-1},K_\theta, \|m_0\|_\infty, \|m_T\|_\infty, C_0). 
$$
Recalling that $\gamma_x= m_0/m(\gamma) = m_0 e^{-w}$, this implies that 
$$
\|\gamma_x, 1/\gamma_x\|_{C^{0,\alpha}([-\theta+\eta/3,\theta-\eta/3]\times [\delta/3,T-\delta/3])} \leq C(\eta^{-1}, \delta^{-1},K_\theta,\|m_0\|_\infty,\|m_T\|_\infty, C_0).
$$
On the other hand (see \rife{eq.ell.gammaBIS} with $f(s)=\log s$), $\gamma$ solves the elliptic equation 
$$
\gamma_{tt}+ \frac{\gamma_{xx}}{\gamma_x^2}= \frac{1}{\gamma_x}\frac{(m_0)_x}{m_0} \quad \text{on}\; (-R/2,R/2)\times (0,T), \qquad \gamma(x,0)=x \ \text{on} \; [-R/2,R/2].
$$
Using the Schauder estimates, we therefore have
\begin{equation}\label{estiC2gamma}
\|\gamma\|_{C^{2,\alpha}([-\theta+\eta/2,\theta-\eta/2]\times [\delta/2,T-\delta/2])}\leq C(\eta^{-1}, \delta^{-1},K_\theta, \|m_0\|_\infty, \|m_T\|_\infty,  C_0).
\end{equation}
Returning to \eqref{eqloggammam}, we obtain, again by the Schauder estimates, 
\begin{equation} \label{estiC2w}
\|w\|_{C^{2,\alpha}([-\theta+\eta,\theta-\eta]\times [\delta,T-\delta])}\leq C(\eta^{-1}, \delta^{-1},K_\theta, \|m_0\|_\infty,\|m_T\|_\infty,   C_0).
\end{equation}
\smallskip

\noindent{\it Step 3: Lower bound on $\gamma$.} We claim that, for any $\delta\in (0,(1/2)\wedge (T/4))$ and any $x\in [0, R/2)$, 
\begin{equation}\label{boundbelowgamma}
\min_{t\in [\delta, T-\delta]} \gamma(x,t)\geq (1-\delta)\left( (R/2)\wedge  \left(\delta^{3/2} (C_0)^{-1/2}\left| \log \left( \int_x^{R/2} m_0(y)dy\right)\right|^{1/2}\right) -1\right).
\end{equation}
{\it Proof of \eqref{boundbelowgamma}:} Fix $x\in [0,R/2)$, $t\in [\delta, T-\delta]$  and set 
$$
a=\max_{s\in [t-\delta^2, t+\delta^2]} \gamma(x,s) + 1. 
$$
Since we want a lower bound for $a$, we can assume that $a\leq R/2$. Then, recalling \rife{masspreservationTER} and the fact that $m(\cdot,t)$ is nonincreasing,   we have for any $s\in [t-\delta^2, t+\delta^2]$, 
$$
\int_x^{R/2} m_0(y)dy = \int_{\gamma(x,s)}^{R/2} m(y,s)dy \geq \int_{\gamma(x,s)}^{a} m(y,s)dy \geq (a-\gamma(x,s))m(a,s)\geq  m(a,s).
$$
  Using the previous inequality together with the HJ equation, and integrating in  $(t-\delta^2,t+\delta^2)$, we get 
\begin{align*}
2\delta^2 \log \left( \int_x^{R/2} m_0(y)dy\right) & \geq \int_{t-\delta^2}^{t+\delta^2} \log(m(a,s))ds   \geq \int_{t-\delta^2}^{t+\delta^2} -u_t(a,s)ds \\ & \geq -\left(\frac{C_0}{t-\de^2}+\frac{C_0}{T-t-\de^2}\right)(a^2+1) \geq - 8a^2\, C_0 \delta^{-1}     
\end{align*}
where we used \eqref{defedfC0bis} and the fact that $a\geq 1$ and $\delta/2\leq t-\delta^2\leq t+\delta^2\leq T-\delta/2$. Thus, up to increasing the value of $C_0$, we obtain 
$$
a \geq \delta^{3/2} C_0^{-1/2}\left| \log \left( \int_x^{R/2} m_0(y)dy\right)\right|^{1/2}.
$$

This proves that 
$$
\max_{s\in [t-\delta^2, t+\delta^2]} \gamma(x,s) \geq (R/2)\wedge  \left(\delta^{3/2} (C_0)^{-1/2}\left| \log \left( \int_x^{R/2} m_0(y)dy\right)\right|^{1/2}\right) -1.
$$
As $\gamma(x,\cdot)$ is nonnegative and concave, letting $s_0$ be a maximum point of $\gamma(x,\cdot)$ in $[t-\delta^2, t+\delta^2]$, 
$$
\gamma(x,t) \geq \left( \frac{t}{s_0}{\bf 1}_{s_0\in [t,t+\delta^2]}+ \frac{T-t}{T-s_0}{\bf 1}_{s_0\in [t-\delta^2, t]}\right) \gamma(s_0)  \geq \min\left\{ \frac{t}{t+\delta^2}, \frac{T-t}{T-t+\delta^2}\right\} \gamma(s_0) \geq (1- \delta)\gamma(s_0) . 
$$
Using our estimate on $\gamma(x,s_0)=\max_{[t-\delta^2, t+\delta^2]} \gamma$ gives \eqref{boundbelowgamma}. \\

\noindent{\it  Step 4: Interior estimate of $m$.}   Fix $\delta\in (0,(1/2)\wedge T/4)$ and $a\in (1,(R/2-1))$, with $R>4$. Now assume that $\eta\in (0,R/2)$ and  $\theta\in (\eta, R/2)$ are such that \eqref{condcondtheta} holds. Then 
\be\label{kuzaqesnrd}
 (1-\delta)\left( (R/2)\wedge  \left(\delta^{3/2} (C_0)^{-1/2}\left| \log \left( \int_{\theta-\eta}^{R/2} m_0(y)dy\right)\right|^{1/2}\right) -1\right) >a+1. 
\ee
We claim that
$$
\|m, m^{-1}\|_{C^{2,\alpha}((-a,a)\times (\delta,T-\delta))} \leq C(\eta^{-1}, \delta^{-1}, K_\theta, \| m_0\|_\infty, \|m_T\|_\infty,  C_0).
$$
Indeed, let $\gamma^{-1}(\cdot,t):(-R/2,R/2)\to (-R/2,R/2)$ denote the inverse in space of $\gamma(\cdot, t)$. 
We first claim that $\gamma^{-1}([-a,a]\times [\delta,T-\delta]) \subset (-\theta+\eta,\theta-\eta)\times [\delta, T-\delta]$. Indeed, otherwise, there exists $t\in [\delta, T-\delta]$ and $x\in [-a,a]$ such that $(x,t)\notin \gamma((-\theta+\eta,\theta-\eta)\times [\delta, T-\delta])$. This means that $x\notin \gamma((-\theta+\eta,\theta-\eta),t)$. As $\gamma(\cdot, t)$ is continuous, vanishes at $x=0$ and is increasing, this implies that $a\geq |x|\geq \gamma(\theta-\eta, t)$, which, by  \eqref{boundbelowgamma},  contradicts  \eqref{kuzaqesnrd}. In particular, \eqref{boundbelowgammaBIS} holds. Recalling \eqref{estiC2gamma}, $\gamma^{-1}$ is bounded in $C^{2,\alpha}$ in $[-a,a]\times [\delta,T-\delta]$ and then \eqref{estiC2w} implies the $C^{2,\alpha}$ bound on $\log(m)=w\circ \gamma^{-1}$. This implies that $m$ and $1/m$ are bounded in $C^{2,\alpha}$ in $[-a,a]\times [\delta,T-\delta]$. 
\end{proof}

\begin{proof}[Proof of Theorem \ref{thm.caselogRsym}] 

 (i) We prove the existence of a classical solution. We start with the planning problem \eqref{eq:MFGPR} and explain at the end of the proof the necessary changes for the \eqref{eq:MFGR} problem. 

Let  $[-b_0,b_0]$ (resp. $[-b_T,b_T]$) be the support of $m_0$ (resp. of $m_T$).  For $R\geq \bar R:=\max\{b_0,b_T\}$,  let $\tilde m^R_0$ and $\tilde m^R_T$ be the continuous periodic map of period $R$ such that $\tilde m^R_0=m_0$ and $\tilde m^R_T=m_T$ on $(-R/2,R/2)$. We let $m_0^{R,\epsilon} = \tilde m^R_0 \ast\eta_{\epsilon}$ and $m_T^{R,\epsilon} = \tilde m^R_T \ast\eta_{\epsilon}$ where $\eta^{\epsilon}$ is a standard mollifier, smooth, even and positive on $\R$. Then, for $R\geq \bar R$ and $\epsilon>0$, $m_0^{R,\epsilon}$ and $m_T^{R,\epsilon}$ are smooth, positive and satisfy \eqref{hyp.symm}. Let $(u^{R,\epsilon}, m^{R,\epsilon})$ be the classical solution to \eqref{eq:Per MFG}--\eqref{per-MFGP} given by Theorem \ref{thm:existence}. By Lemma \ref{lem.convexity}, $m^{R,\epsilon}$ satisfies \eqref{PptExtram}. 

We now prove an interior estimate for $m^{R,\epsilon}$. Fix $a>1$ and $\delta \in (0,(1/2)\wedge (T/4))$. Choose $R_0>4$ such that $(R_0/2-1)\geq a+1$. Fix also $r_0\in (0,b^0)$ large enough such that 
$$
\delta^{3/2} (C_0)^{-1/2}\left| \log \left( \int_{r_0}^{b_0} m_0(y)dy\right)\right|^{1/2}  > 2a+1.  
$$
This is possible as $\int_{r_0}^{b_0} m_0(y)dy\to 0$ as $r_0\to (b^0)^-$. We then set $\theta = (r_0+b_0)/2$ and $\eta= (b_0-r_0)/2$.  We finally choose $\epsilon_0\in (0, \eta/4)$ such that, for $\epsilon\in (0,\epsilon_0)$, 
$$
\delta^{3/2} ( 8C_0)^{-1/2}\left| \log \left( \int_{r_0}^{b_0}  m_0^{R,\epsilon}(y)dy\right)\right|^{1/2}  > 2a\qquad \text{and}\qquad 
\int_{\theta}^{R/2}  m_0^{R,\epsilon}(y)dy\geq \frac12 \int_{\theta}^{R/2} m_0(y)dy.
$$
Then, for any $R\geq R_0$, $\epsilon\in (0,\epsilon_0)$, condition \eqref{condcondtheta} holds for $m_0^\epsilon$. By Lemma \ref{lem.esticaselogRsym},
\begin{multline}
\|m^{R,\epsilon}, 1/m^{R,\epsilon}\|_{C^{2,\alpha}((-a,a)\times (\delta,T-\delta))}  \leq C( \eta^{-1},  \delta^{-1}, K^\epsilon_\theta,\|\tilde m^R_0\|_\infty,, \|\tilde m^R_T\|_\infty,  C_0^\epsilon) \\
=  C( \eta^{-1},  \delta^{-1}, K^\epsilon_\theta, \|m_0\|_\infty,\|m_T\|_\infty,  C_0^\epsilon) ,
\end{multline}
where $C_0^\epsilon$ depends only on $T$,   $\cE(\tilde m_0^R)$, $\cE(\tilde m_T^R)$, $M_2(\tilde m^R_0)$, $M_2(\tilde m^R_T)$,  and thus only on $T$,   $\cE(m_0)$, $\cE( m_T)$, $M_2( m_0)$, $M_2( m_T)$, and 
\begin{multline}
K^\epsilon_\theta = \|(\tilde m^{R,\epsilon}_0)^{-1}\|_{L^\infty((-\theta,\theta))}+  \|\tilde m^{R,\epsilon}_0\|_{C^{1,\alpha}((-\theta,\theta))}+\left( \int_{\theta}^{R/2} \tilde m^{R,\epsilon}_0(y)dy\right)^{-1}
 \\
 \leq  \|m_0^{-1}\|_{L^\infty((-\theta-\epsilon_0,\theta+\epsilon_0))}+ \|m_0\|_{C^{1,\alpha}((-\theta-\epsilon_0,\theta+\epsilon_0))}+2 \left( \int_{b_0-\eta/2}^{b_0} m_0(y)dy\right)^{-1}, 
\end{multline}
which is finite since $\theta+\epsilon_0 = b_0-\eta$. This shows that 
 $$
 \|m^{R,\epsilon}, 1/m^{R,\epsilon}\|_{C^{2,\alpha}((-a,a)\times (\delta,T-\delta))} \leq C(a, \delta^{-1}, T, m_0, m_T). 
 $$
 Since $\log(m^{R,\epsilon})$ is uniformly Lipschitz continuous in $(-a,a)\times (\delta,T-\delta)$ and since the map $u^{R,\epsilon}$ is a locally uniformly bounded solution of a HJ equation with  r.h.s. 
$\log(m^{R,\epsilon})$,  $u^{R,\epsilon}$ is uniformly Lipschitz continuous in  $(-a/2,a/2)\times (2\delta,T-2\delta)$. By  
\rife{eq.elluBIS} (with $f=\log s$), we know that $u^{R,\epsilon}$ satisfies the elliptic equation 
\be\label{ureps}
-u_{tt}+2u_xu_{xt} -(u_x^2+1)u_{xx}=0\,.
\ee
Hence, by elliptic regularity,  we obtain
$$
\|u^{R,\epsilon}\|_{C^{2,\alpha}((-a/2,a/2)\times (2\delta,T-2\delta))} \leq C(a, \delta^{-1},m_0,m_T, C_0). 
$$
We can now use the estimates above and the first part of Lemma \ref{lem.esticaselogRsym}  to find (a subsequence) of $(u^{R,\epsilon}, m^{R,\epsilon})$ which converges, as $\epsilon\to 0^+$ and then $R\to \infty$, to a pair $(u,m)$ which is a $C^{2,\alpha}$ solution of  the MFG system \eqref{mfgsys} in $\R\times (0,T)$  with $f(m)=\log(m)$ and such that:
$$
m\in C^0([0,T], \mathcal P_2(\R)), \; m(0)=m_0, \; m(T)= m_T,. 
$$
Moreover, by construction, $m(\cdot, t)$ is even and $x\to m(x,t)$ is nonincreasing on $[0,+\infty)$ for any $t\in [0,T]$. 
 
 Let us finally check the continuity of $m$ at $t=0$ (the case $t=T$ being symmetric). Let $t_n\to 0^+$. Then the maps $m(\cdot, t_n)$ are nonincreasing on $[0,\infty)$ and  converge weakly-* (as measures) to $m_0$ which has a continuous density:  this limit is therefore locally uniform. \\
 
 We now consider the \eqref{eq:MFGR} problem. We regularize $m_0^{R,\epsilon}$ as above and let $(u^{R,\epsilon},m^{R,\epsilon})$ be the classical solution to \eqref{log-per}--\eqref{mfg-log}. Let $\gamma^{R,\epsilon}$ be the associated flow of optimal trajectories.  Let us recall that, under our structure condition, $\gamma^{R,\epsilon}(x,\cdot)$  is concave and nondecreasing in time for any $x\in [0,R/2]$. 
 
 Fix $a>1$ and $\delta \in (0,(1/2)\wedge (T/4))$. We can choose $R_0$, $\theta$, $\eta$ and $\epsilon_0>0$ as in the first part of the proof such that, for $R\geq R_0$ and $\epsilon\in (0,\epsilon_0)$, 
 $$
 \int_{\theta}^{R/2}  m_0^{R,\epsilon}(y)dy\geq \frac12 \int_{\theta}^{R/2} m_0(y)dy >0
$$
and 
$$
 \|m^{R,\epsilon}, 1/m^{R,\epsilon}\|_{C^{2,\alpha}((-a,a)\times (\delta,T-\delta))} \leq C(a, \delta^{-1}, T, m_0, m_T). 
 $$
 With this choice we also have (estimate \eqref{boundbelowgamma} in Lemma \ref{lem.esticaselogRsym})
 $$
\min_{t\in [\delta, T-\delta]}\gamma(\theta-\eta,t)>a. 
$$
Recall also that, by \eqref{boundabovemm} and \eqref{lizajqesdfg}, for any $(x,t)\in (-R/2,R/2)\times [0,T]$, 
$$
( 8C_0)^{-1/2} \left( \int_{|x|}^{R/2} m^{R,\epsilon}_0(y)dy\right)^{3/2}\leq m^{R,\epsilon}(\gamma(x,t),t) \leq \|m_0^{R,\epsilon}\|_\infty\leq \|m_0\|_\infty. 
$$
Hence for any $(y,t)\in [0, a]\times [\delta, T]$, there exists $x\in [0,\theta-\eta]$ with $\gamma^{R,\epsilon}(x,t)=y$ and therefore 
$$
 \|m_0\|_\infty \geq m^{R,\epsilon}(y,t)= m^{R,\epsilon}(\gamma(x,t),t)\geq 
\frac12( 8C_0)^{-1/2} \left( \int_{\theta}^{R/2} m_0(y)dy\right)^{3/2},
$$
which proves that $m^{R, \epsilon}$ is bounded above and below in $[-a,a]\times [\delta,T]$. 
Next we show a Lipschitz bound for $u^{R,\epsilon}$ in $[0,a]\times [\delta,T]$. For any $(y,t)\in [0,a]\times [\delta, T]$, there exists $x\in [0,y]$ such that $\gamma^{R,\epsilon}(x,t)=y$. Recall that $\gamma^{R,\epsilon}_t(x,t)=-u_x(y,t)$. On the other hand, by concavity of $\gamma^{R,\epsilon}(x,\cdot)$, 
 $$
0\leq  \gamma^{R,\epsilon}(x,s)\leq y+ \gamma^{R,\epsilon}_t(x,t)(s-t) \qquad \forall s\in [0,T].
$$
Thus (choosing $s=0$ and using that $y\leq a$ and $t\geq \delta$)
$$
0\leq \gamma^{R,\epsilon}_t(x,t)=-u^{R,\epsilon}_x(y,t) \leq y/t \leq a/\delta.
$$
By symmetry, this proves that 
$$
\|u^{R,\epsilon}_x\|_{L^\infty([-a,a]\times [\delta, T])} \leq a/\delta.
$$
Let us finally check a local bound for $u^{R,\epsilon}$. We already have a bound below (Lemma \ref{lem.esticaselogRsym}). As $u^{R,\epsilon}(\cdot,t)$ is nonincreasing on $[0,R/2]$ and $\gamma(0,\cdot)=0$, we have, for $(x,t)\in [0,a]\times [\delta, T]$:
\begin{align*}
u^{R,\epsilon}(x,t) & \leq u^{R,\epsilon}(0,t)= \int_t^T 
\frac12 |\gamma^{R,\epsilon}_s(0,s)|^2 + 
\ln(m^{R,\epsilon}(\gamma^{R, \epsilon}(0,s))) \ ds
+ c_T\ln(m^{R,\epsilon}(\gamma^{R, \epsilon}(0,T),T)) \\
& \leq C \|m^{R,\epsilon}, 1/m^{R,\epsilon}\|_{L^\infty(\{0\}\times [\delta, T])} \leq C(a, \delta^{-1}).
\end{align*}

We have proved positive upper and lower bounds for $m^{R, \epsilon}$ and Lipschitz bounds for $u^{R, \epsilon}$ independent of $R,\epsilon$  on $[-a,a]\times [\delta, T]$.  By the elliptic equation \rife{ureps} satisfied by $u^{R, \epsilon}$, we can infer that 
$$
\|u^{R,\epsilon}\|_{C^{2,\alpha}([-a/2,a/2]\times [2\delta,T]}\leq C(a, \delta^{-1}). 
$$
We can then conclude as before. 

 (ii) The uniqueness of solutions is proved with the same kind of argument used in Theorem \ref{ex-log-torus}. First of all, we observe that, since $m(t) \in \cP_2(\R)$ and $u(t)/(1+ |x|^2)\in L^\infty(\R)$, then we have $u(t)m(t) \in L^1(\R)$ for any $t\in (0,T)$.  This implies that  a similar  equality as \rife{t0t1} holds, namely
\be\label{t0t12}
\intr u(t_0)m(t_0) - \intr u(t_1)m(t_1) = \int_{t_0}^{t_1}  \intr mu_x^2 +\int_{t_0}^{t_1}\intr m\log(m) \qquad\forall \, 0<t_0<t_1<T\,.
\ee
As in the proof of Theorem \ref{ex-log-torus}, we deduce from \rife{t0t12} (and the time-monotonicity of $u$) that $u(0^+)\in L^1(dm_0)$ and $u(T^-)\in L^1(dm(T))$. Then, using a truncation argument for $u$ and the continuity of $m$, inequality \rife{t0t12} is   extended up to $t_0=0$ and $t_1=T$, obtaining the equivalent of \rife{ide} but integrated for $x\in \R$. Notice that the same truncation argument for $u$ as in Theorem \ref{ex-log-torus} works here, because $m$ has finite moments (uniformly in time), so actually $\intr u_k(t)m(t)$ ends up being continuous in $[0,T]$ for fixed $k$.  In a similar way,  we obtain the equivalent of \rife{ideh} for $x\in \R$, and we conclude from the Lasry-Lions monotonicity argument as in the compact case.

\end{proof}
\subsection*{Acknowledgments}
The authors would like to thank P.E. Souganidis for his great help in facilitating this project, which was partly developed during the second author's visit at the University of Rome Tor Vergata. The first author was partially supported  by AFOSR grant FA9550-18-1-0494. The second author was partially supported by P. E. Souganidis? NSF grant DMS-1900599, ONR grant N000141712095 and AFOSR grant FA9550-18-1-0494, as well as funds provided by the University of Rome Tor Vergata (Fondi di Ateneo ``ConDyTransPDE''). The third author was also supported by Indam Gnampa Projects 2022, by the Excellence  Projects Math@TOV and MatMod@TOV of the Department of Mathematics of the University of Rome Tor Vergata; he also wishes to thank the University of Paris Dauphine on occasion of the invitation received in the spring 2023 at CEREMADE.

\appendix
\section{Construction of the self-similar solutions} \label{sec:appendix}

The goal of this appendix is to prove the statements made in Proposition \ref{self-building} regarding the construction of the self-similar solution, as well as to provide a precise analysis of the regularity of the value function $u$. We begin by showing that, inside the support of $m$, the system is solved in the classical sense.

\begin{lem} Let  $u,m$ be defined as in Proposition \ref{self-building}. Then $m$ is a weak solution of the continuity equation, $u$ is $C^1$ in the support of $m$, and $u$ satisfies, in the classical sense,
$$
-u_t+\frac12 |u_x|^2= m^\theta \quad \hbox{in $\,\{m>0\}$.}
$$ 
\end{lem}

\begin{proof}  We  set $  \gamma^\pm(t)= \pm (2R/(\overline{\alpha}-\overline{\alpha}^2))^{1/2} t^{\overline{\alpha}}$, and we note that 
 $$
 \{m(t,\cdot)>0\}=\{(x,t), \;  \gamma^-(t)<x<  \gamma^+(t)\}.
 $$ 
Here we have, according to \rife{usuppo} and if $m(x,t)>0$, :  
$$
-u_t+\frac12 |u_x|^2 = -\overline{\alpha}\frac{x^2}{2t^2}+ R t^{-2\theta/(2+\theta)} +\frac12 \overline{\alpha}^2\frac{x^2}{t^2}= -(\overline{\alpha}-\overline{\alpha}^2) \frac{x^2}{2t^2} + R t^{-2\theta/(2+\theta)}
 $$
 while 
 $$
 f( m(x,t))= t^{-\overline{\alpha} \theta}  \left( R-\frac12(\overline{\alpha}-\overline{\alpha}^2)(\frac x{t^{\overline{\alpha}}})^2\right)=  t^{-\overline{\alpha} \theta}R- \frac12(\overline{\alpha}-\overline{\alpha}^2)\frac{x^2}{t^{\overline{\alpha}\theta+2\overline{\alpha}}}\,.
 $$
By the definition of $\overline{\alpha}$, we have $\overline{\alpha}\theta= 2\theta/(2+\theta)$, and $\overline{\alpha}\theta+2\overline{\alpha}=2$. Thus we conclude that 
 $$
 -u_t+\frac12 u_x^2=   f( m(x,t))\qquad {\rm in}\; \{m>0\}.
 $$
 On the other hand, for any test function $\varphi\in C^\infty_c(\R \times (0,T))$, we have 
 \begin{align*}
 &\frac{d}{dt} \int_\R \varphi(x,t) m(x,t)dx=  \frac{d}{dt} \int_{  \gamma^-(t)}^{  \gamma^+(t)} \varphi(x,t)t^{-\overline{\alpha}} \phi(x/t^{\overline{\alpha}})dx \\ 
 &\qquad = \int_{  \gamma^-(t)}^{  \gamma^+(t)} ( \varphi_t(x,t)t^{-\overline{\alpha}} \phi(x/t^{\overline{\alpha}})+  \varphi(x,t) (-\overline{\alpha} t^{-\overline{\alpha}-1})
 (\phi(x/t^{\overline{\alpha}})+ xt^{-\overline{\alpha}}\phi'(x/t^{\overline{\alpha}})))dx
 \end{align*}
 where 
$$
 \int_{ \gamma^-(t)}^{  \gamma^+(t)}\varphi(x,t) x\phi'(x/t^{\overline{\alpha}})dx = - t^{\overline{\alpha}}\int_{\bar \gamma^-(t)}^{\bar \gamma^+(t)}(\varphi_x(x,t) x+\varphi(x,t))\phi(x/t^{\overline{\alpha}})dx.
 $$
 So 
  \begin{align*}
\frac{d}{dt} \int_\R \varphi(x,t) m(x,t)dx& =   \int_{  \gamma^-(t)}^{  \gamma^+(t)} (\varphi_t(x,t)t^{-\overline{\alpha}} \phi(x/t^{\overline{\alpha}})- \varphi_x(x,t) (-\overline{\alpha} t^{-\overline{\alpha}-1})x\phi(x/t^\alpha)))dx \\
&= \int_{\R} (\varphi_t(x,t)- \varphi_x(x,t) u_x(x,t)) m(x,t)dx .
 \end{align*}
This shows that $m$ solves the continuity equation.  
\end{proof}

\bigskip
We now extend the definition of $u$ outside the support of $m$, and analyze its behavior near the interface. We recall that the free  boundary is the set $\{\Delta=0\}$, where $
\Delta=|x|-\sqrt{\frac{2R}{\overline{\alpha}(1-\overline{\alpha})}}t^{\overline{\alpha}}$.

\begin{lem} \label{prop:0}
For each $(x,t)\text{\ensuremath{\in\{\Delta>0\}}}$, the equation
(\ref{eq:1}) has a unique positive solution $S\in(0,t)$, and S is
a smooth function of $(x,t)$. Furthermore, the function $S=S(x,t)$
extends continuously to $\{\Delta\geq0\}$, $S(x,t)=t$ on $\Delta=0$,
and one has the estimates
\be\label{ests1}
S(x,t)\geq c_{0}\left(\frac{t}{|x|+\Delta}\right)^{\frac{1}{1-\overline{\alpha}}},
\ee
\be\label{ests2}
|t-S(x,t)|\leq C_{0}t^{1-\frac{\overline{\alpha}}{2}}\Delta^{\frac{1}{2}}.
\ee
where $c_{0}=\left(\frac{2R\overline{\alpha}}{1-\overline{\alpha}}\right)^{\frac{1}{2(1-\overline{\alpha})}}$ and $C_{0}=\left(\frac{2}{R\overline{\alpha}(1-\overline{\alpha})}\right)^{\frac{1}{4}}$.
\end{lem}

\begin{proof}
Let us  set $C_R:= \sqrt{\frac{2R}{\overline{\alpha}(1-\overline{\alpha})}}$; 
hence the interface $\Delta=0$ is the curve $|x|= C_R t^{\overline{\alpha}}$. 
We  define 
\begin{equation}
F(x,t,s)=  -|x|s^{1-\overline{\alpha}}+  C_R( \overline{\alpha} t + (1-\overline{\alpha}) s).
\label{eq:F}
\end{equation}
Then, we have

\begin{equation}
\frac{\partial F}{\partial s}(x,t,s)=-(1-\overline{\alpha})s^{-\overline{\alpha}}\left(|x|-C_R s^{\overline{\alpha}}\right).\label{eq:3}
\end{equation}
Notice that, since $\Delta(x,t)>0$, $\frac{\partial F}{\partial s}(x,t,s)<0$
for $s\in[0,t]$. Moreover, $F(x,t,0)=\overline{\alpha} C_R t>0$,
and $F(x,t,t)=-  t^{1-\overline{\alpha}}  \Delta<0$. Hence there exists a unique $S\in(0,t)$
with $F(x,t,S)=0$. Now, since $\frac{\partial F}{\partial s}(x,t,S)<0$,
the implicit function theorem guarantees that the function $S$ is
smooth in $\{\Delta>0\}$. We now show a lower bound on $S$, by taking
$s= \left(\frac{\overline{\alpha} C_R t}{|x|+\Delta}\right)^{\frac{1}{1-\overline{\alpha}}}$.  First we note
that, since $\overline{\alpha}<1$, we have 
\[
s< \left( \frac{C_R t}{|x|}\right)^{\frac{1}{1-\overline{\alpha}}}  \leq t \qquad \hbox{in $\{\Delta>0\}$.}
\]
Hence $s\in(0,t)$. Moreover, we have
$$
F(x,t,s)    = - \frac{ |x| }{|x|+\Delta}\overline{\alpha} C_R t+  C_R( \overline{\alpha} t + (1-\overline{\alpha}) s) \geq 0\,.
$$
Consequently, since $\frac{\partial F}{\partial s}<0$ on $(0,t)$,
and $F(x,t,S)=0$, we have $S\geq s$, that is,
\[
S(x,t)\geq  \left(  \frac{\overline{\alpha} C_R t}{|x|+\Delta}\right)^{\frac{1}{1-\overline{\alpha}}}.
\]
This is \rife{ests1}.
Finally, we are now concerned with the continuous extension to $\Delta=0$. First of all, we rewrite \rife{eq:1} as 
\begin{equation}
\overline{\alpha} C_R (t-S) =S^{1-\overline{\alpha}}(|x|-C_R S^{\overline{\alpha}})\label{eq:t-S}
\end{equation}
which yields
\begin{align}
\overline{\alpha} C_R (t-S)  & \left(1-\frac{1}{\overline{\alpha}}S^{1-\overline{\alpha}}\frac{t^{\overline{\alpha}}-S^{\overline{\alpha}}}{t-S}\right)=S^{1-\overline{\alpha}}\Delta.
\label{t-S}
\end{align}
We can then use the elementary inequality 
\[
 \frac{(1-\overline{\alpha})}{2}(1-w)\leq\left(1-\frac{1}{\overline{\alpha}}w^{1-\overline{\alpha}}\frac{1-w^{\overline{\alpha}}}{1-w}\right),
\]
valid for all real numbers $w\in(0,1)$. Using this inequality with $w=\frac{S}{t}$, we deduce from \rife{t-S}
%
\[
\frac{\overline{\alpha} C_R (1-\overline{\alpha}) }{2t} (t-S)^{2}\leq S^{1-\overline{\alpha}}\Delta\leq t^{1-\overline{\alpha}}\Delta,
\]
which reduces to
\[
|t-S|\leq C_{0}t^{1-\frac{\overline{\alpha}}{2}}\Delta^{\frac{1}{2}}
\]
 by setting $C_0=  \left(\frac{2}{R\overline{\alpha}(1-\overline{\alpha})}\right)^{\frac{1}{4}}$.
From this estimate, one sees that if $(x_{n},t_{n})\in\{\Delta>0\}$
is such that $(x_{n},t_{n})\rightarrow(x_{0},t_{0})\in\{\Delta=0\}$,
then $|S(x,t)-t|\rightarrow0$. 
\end{proof}

We can now establish the H\"older regularity of $Du$.

\begin{prop} \label{prop:self similar regu}
There exists $0<s<1$ such that the function $u$  (defined in \rife{u2} or \rife{uneq2}) is smooth away from $\Delta=0$, and is uniformly $C^{1,s}$
on compact subsets of $(0,\infty)\times\mathbb{R}.$ Moreover, it is a  viscosity solution of 
$$
-u_{t}+\frac{1}{2}u_{x}^{2}=m^\theta \qquad \hbox{in $\R\times(0,T)$.}
$$
\end{prop}

\begin{proof}
 Given $(x,t) \in \{\Delta >0\}$, $u(x,t)$ has been defined through the method of characteristics;  outside the support of $m$, the characteristics are straight lines that join $(x,t)$ to  a unique point $(\bar x, S)$ belonging to the curve $\{\Delta =0\}$. This means that
$$
u(x,t)= u(\bar x, S)+ \frac12 (t-S)  |\lambda|^2 \,,\qquad \begin{cases} x=\lambda(t- S)+ \bar x \,,& \quad \lambda= -u_x(\bar x, S)= \frac{\alpha \bar x}{S}\\
\bar x= C_R S^{\overline{\alpha}} & \end{cases}
$$
which leads to the value $S=S(x,t)$ defined by \rife{eq:t-S} and, correspondingly to the formulas \rife{u2} or \rife{uneq2}.  By construction, relying on the method of characteristics, it follows that $u$ satisfies  $-u_{t}+\frac{1}{2}u_{x}^{2}=0$ in $ \{\Delta >0\}$, and $u$ is actually a viscosity solution in the whole $\R\times(0,T)$. 

Let us now look at the regularity of $u$. 
Since $S$ is smooth away from $\Delta=0$, so is $u$. More precisely, recalling that $S$ is given by the equation
\be\label{Seq}
\overline{\alpha} C_R t = S^{1-\overline{\alpha}} |x|- (1-\overline{\alpha})C_R S
\ee
we have, by implicit differentiation,
\be\label{sx}
S_{x}=\frac{-\text{sgn}(x)S}{(1-\overline{\alpha})(|x|-C_R S^{\overline{\alpha}})}\,.
\ee
Suppose that $\theta\neq 2$, so $u$ is given by \rife{uneq2}, which simplifies into
$$
u= -\frac{2\overline{\alpha} R}{2\overline{\alpha}-1} S^{2\overline{\alpha}-1}- \frac{\overline{\alpha} R}{1-\overline{\alpha}}t \, S^{2\overline{\alpha}-2}\,.
$$
Therefore,
\[
u_{x}=-2\overline{\alpha} R \left( 1-\frac t S\right) S^{2\overline{\alpha}-2} S_x,
\]
and from \rife{sx} and (\ref{eq:t-S}) we see that this simplifies to:
\begin{equation}
u_{x}= - \frac{2 R}{C_R(1-\overline{\alpha})}  S^{\overline{\alpha}-1}   \text{sgn}(x)\label{selfux}
\end{equation}
By Proposition \ref{prop:0}, we deduce that $u\in C^1(\R\times(0,T))$; in fact, we can see that $u\in C^{1,s}$, if we 
prove a uniform $C^{s}$ bound for $S$ on compact subsets
near $\Delta=0$.  To this purpose,  let $K$ be a compact subset of $\Delta>0$, and let $(x,t),(\overline{x},\overline{t})$
be two points in $K$. We write $\overline{S}=S(\overline{x},\overline{t})$.
With no loss of generality, we may assume that $S\geq\overline{S}$.  
By formula \rife{Seq}, we have
$$
S^{1-\overline{\alpha}} |x|- {\bar S}^{1-\overline{\alpha}} |\bar x| - (1-\overline{\alpha})C_R (S-\bar S)= C_R(t-\bar t)
$$
which yields, by using $S^{1-\overline{\alpha}}\geq {\bar S}^{1-\overline{\alpha}} + (1-\overline{\alpha}) S^{-\overline{\alpha}} (S-\bar S)$
$$
\frac{(1-\overline{\alpha}) }{S^{\overline{\alpha}}} \left(|x|- C_R S^{\overline{\alpha}}\right) (S-\bar S) \leq C_R |t-\bar t| + {\bar S}^{1-\overline{\alpha}} |x-\bar x|\,.
$$
By definition of $\Delta$, and Proposition \ref{prop:0}, we deduce
$$
|S-\bar S| \leq K_{t,\bar t} \, \frac{\left(|x-\overline{x}| +|t-\overline{t}|\right)} {\Delta}\,.
$$
Suppose now that $\Delta \geq \left(|x-\overline{x}| +|t-\overline{t}|\right)^{\sigma}$, for some $\sigma\in (0,1)$. Then we get
$$
|S-\bar S| \leq K_{t,\bar t} \,  \left(|x-\overline{x}| +|t-\overline{t}|\right)^{1-\sigma}\,.
$$
On the other hand, if $\Delta < \left(|x-\overline{x}| +|t-\overline{t}|\right)^{\sigma}$, then we also have 
$$
\bar \Delta < \left(|x-\overline{x}| +|t-\overline{t}|\right)^{\sigma}+ |\Delta -\bar \Delta | \leq C_{t,\bar t} \left(|x-\overline{x}| +|t-\overline{t}|\right)^{\sigma}
$$
because $\Delta $ is a  Lipschitz function of $(x,t)$ far from $t=0$ (it is also globally $\overline{\alpha}$-H\"older, as well). Hence we estimate, using \rife{ests2}
\begin{equation*}
|S-\bar S|  \leq |S-t|+ |t-\bar t| + |\bar t - \bar S| \leq  C_0 \max(t,\bar t)^{1-\frac{\overline{\alpha}}{2}} \left( \Delta^{\frac12}+ {\bar\Delta}^{\frac12}\right) +  |t-\bar t| \leq {\tilde C}_{t, \bar t}  \left(|x-\overline{x}| +|t-\overline{t}|\right)^{\frac\sigma2}\,.
\end{equation*}
This concludes with the H\"older bound of $S(x,t)$, and therefore with the $C^{1,s}$ regularity of $u$. Finally, for the case $\theta=2$, we argue in the same way  by using formula \rife{u2}; notice that $u$ is explicit in this case, because (\ref{eq:1}) is a quadratic equation in
$\sqrt{S}$.

\end{proof}

\begin{rem} \label{rem:uxx unbounded}
We observe that the solution $u$ found above is not $W^{2,\infty}$. In fact, by differentiating once more \rife{selfux} one gets
\[
u_{xx}=-\sqrt{\frac{2R\overline{\alpha}}{1-\overline{\alpha}}}\frac{S^{\overline{\alpha}-1}}{(|x|-\sqrt{\frac{2R}{\overline{\alpha}(1-\overline{\alpha})}}S^{\overline{\alpha}})}<0\,.
\]
Now, as $\Delta\rightarrow0$, $S\rightarrow t$, and thus the denominator
$|x|-\sqrt{\frac{2R}{\overline{\alpha}(1-\overline{\alpha})}}S^{\overline{\alpha}}\rightarrow0$. Hence
$u_{xx}$ is unbounded.

The same holds for the case $\theta=2$, since we have
\[
u_{x}=-\frac{4R}{x-\sqrt{\Delta}},u_{xx}=-\frac{4R}{\sqrt{\Delta}(x-\sqrt{\Delta})}\,.
\]
So $u_{xx}$ is again unbounded as $\Delta\rightarrow0$ .
\end{rem}

\end{document}